\documentclass[reqno]{amsart}

\theoremstyle{plain}
\newtheorem{thm}{Theorem}[section]
\newtheorem{cor}[thm]{Corollary} 
\newtheorem{lemma}[thm]{Lemma} 
\newtheorem{prop}[thm]{Proposition}

\newtheorem{thmsubs}{Theorem}[subsection]
\newtheorem{corsubs}[thmsubs]{Corollary} 
\newtheorem{lemmasubs}[thmsubs]{Lemma} 
\newtheorem{propsubs}[thmsubs]{Proposition}

\theoremstyle{remark}

\theoremstyle{definition}
\newtheorem{defi}[thm]{Definition}
\newtheorem{example}[thm]{Example}

\newtheorem{ques}[thm]{Question}

\usepackage{amscd,amssymb,comment,epic,eepic,euscript,graphics}
\usepackage{enumerate}
\usepackage[initials]{amsrefs}

\usepackage{color}

\newcommand\Afr{{\mathfrak A}}
\newcommand\alg{{\operatorname{alg}}}
\newcommand\Bfr{{\mathfrak B}}
\newcommand\clspan{{\overline{\mathrm{span}}\,}}
\newcommand\Cpx{{\mathbb C}}
\newcommand\eps{\epsilon}
\newcommand\Dc{{\mathcal{D}}}
\newcommand\diag{\operatorname{diag}}
\newcommand\dist{\operatorname{dist}}
\newcommand\dsp{d\hspace*{0.1em}}
\newcommand\Ec{{\mathcal{E}}}
\newcommand\EEu{{\EuScript E}}                   

\newcommand\FEu{{\EuScript F}}                   

\newcommand\GEu{{\EuScript G}}                   
\newcommand\HEu{{\EuScript H}}                   
\newcommand\Hol{\mathrm{Hol}}

\newcommand\id{{\operatorname{id}}}
\newcommand\Mcal{{\mathcal{M}}}
\newcommand\Nats{{\mathbb N}}
\newcommand\Nc{{\mathcal{N}}}

\newcommand\oneh{{\hat 1}}
\newcommand\oup{^{\mathrm o}}
\newcommand\partialsp{\partial\hspace*{0.1em}}

\newcommand\Pt{{\widetilde P}}

\newcommand\RealPart{{\mathrm{Re}\,}}
\newcommand\Reals{{\mathbb R}}
\newcommand\restrict{{\upharpoonright}}

\newcommand\Sp{{\operatorname{Sp}}}
\newcommand\spec{\operatorname{spec}}
\newcommand\supp{\operatorname{supp}}

\newcommand\Uc{{\mathcal{U}}}
\newcommand\Vc{{\mathcal{V}}}
\newcommand\Wc{{\mathcal{W}}}

\newcommand\xh{{\hat x}}

\newcommand\zbar{{\overline z}}

\begin{document}

\title[Simultaneous upper triangular forms]{Simultaneous upper triangular forms for commuting operators
in a finite von Neumann algebra}

\author[Charlesworth]{Ian Charlesworth$^\circ$}
\address{I.\ Charlesworth, Department of Mathematics, UC--Berkeley, Berkeley, CA 94720-3840, USA}
\email{ilc@math.berkeley.edu}
\thanks{{}$^\circ$Supported by a grant from the NSF (DMS-1803557).}

\author[Dykema]{Ken Dykema$^{\dag*}$}
\address{K.\ Dykema, Department of Mathematics, Texas A\&M University, College Station, TX 77843-3368, USA}
\email{kdykema@math.tamu.edu}
\thanks{{}$^*$Supported by a grant from the Simons Foundation/SFARI (524187, K.D.)
and by a grant from the NSF (DMS-1800335).}

\author[Sukochev]{Fedor Sukochev$^{\dag}$}
\address{F.\ Sukochev, School of Mathematics and Statistics, University of New South Wales, Kensington, NSW, Australia.}
\email{f.sukochev@unsw.edu.au}
\thanks{{}$^\dag$Supported by the ARC}

\author[Zanin]{Dmitriy Zanin$^{\dag}$}
\address{D.\ Zanin, School of Mathematics and Statistics, University of New South Wales, Kensington, NSW, Australia.}
\email{d.zanin@unsw.edu.au}


\date{April 26, 2019}

\begin{abstract}
The joint Brown measure and joint Haagerup--Schultz projections for tuples of commuting operators in
a von Neumann algebra equipped with a faithful tracial state are investigated, and several natural properties are proved for these.
It is shown that the support of the joint Brown measure is contained in the Taylor joint spectrum of the tuple,
and also in the ostensibly smaller left Harte spectrum.
A simultaneous upper triangularization result for finite commuting tuples is proved
and the joint Brown measure and joint Haagerup--Schultz projections are shown to behave well under
the Arens multivariate holomorphic functional calculus of such a commuting tuple.
\end{abstract}

\maketitle

\section{Introduction and statement of main results}

A well known classical result is that commuting matrices can be simultaneously upper-triangularized.
Namely, given a commuting family $A_1,\ldots, A_n$ in $M_k(\Cpx)$, there is a unitary $U$ such that
each $U^*A_jU$ is an upper triangular matrix.
Moreover, if $(\lambda^{(j)}_1,\ldots,\lambda^{(j)}_k)$ is the diagonal of $U^*A_jU$, then the
set
\[
\{(\lambda_p^{(1)},\ldots,\lambda_p^{(n)})\mid 1\le p\le k\}\subseteq\Cpx^n
\]
is the joint spectrum of $(A_1,\ldots,A_n)$.

Let $\Mcal$ be a von Neumann algebra equipped with a normal, faithful, tracial state $\tau$.
In this paper, we prove analogous results for commuting families of elements of $\Mcal$.

For an element $S\in\Mcal$, there is a spectral distribution measure $\nu_S$, that was found by L.\ Brown~\cite{B86};
it is a Borel probability measure whose support is contained in the spectrum of $S$
and is called the {\em Brown measure} of $S$.
In a fundamental paper~\cite{HS09}, Uffe Haagerup and Hanne Schultz found $S$-hyperinvariant projections $P(S,B)$
for Borel sets $B\subseteq\Cpx$;
these projections decompose the Brown measure (see \S\ref{subsec:HS} below for a more precise statement).
We call these $P(S,B)$ the Haagerup--Schultz projections of $S$.

In~\cite{S06}, building on results that were eventually published in~\cite{HS09},
Schultz constructed analogues of Brown measure and Haagerup--Schultz projections for an $n$-tuple $T=(T_1,\ldots,T_n)$
of commuting elements $T_j$ of $\Mcal$.
Her proof used elegant arguments involving an idempotent--valued measure in the algebra of (unbounded) operators affiliated to $\Mcal$.

The following is a combination of Theorems~4.1 and~5.3 of~\cite{S06}.
\begin{thm}[\cite{S06}]\label{thm:Schultz}
Let $T=(T_1,\ldots,T_n)$ be a tuple of commuting elements of $\Mcal$.
Then there is a probability measure $\nu_T$ on $\Cpx^n$ and there is a map $B\mapsto P(T:B)$ from Borel subsets of $\Cpx^n$ to projections in $\Mcal$ satisfying
\begin{enumerate}[(a)]
\item $\nu_T(B)=\tau(P(T:B))$ for all Borel subsets $B$ of $\Cpx^n$;
\item if $B$ is a Borel rectangle, namely, if $B=B_1\times\cdots\times B_n$ for Borel sets $B_j$ of $\Cpx$, then
\[
P(T:B)=\bigwedge_{j=1}^n P(T_j,B_j);
\]
\item if $B=\bigcup_{k=1}^\infty B^{(k)}$ is a countable disjoint union of Borel rectangles $B^{(k)}$, then
\[
P(T:B)=\bigvee_{k=1}^\infty P(T:B^{(k)});
\]
\item for a general Borel subset $B$ of $\Cpx^{(n)}$,
\[
P(T:B)=\bigwedge_{\substack{B\subseteq U\subseteq\Cpx^n\\[0.5ex] U\text{open}}}P(T:U).
\]
\end{enumerate}
\end{thm}

We will call $\nu_T$ the {\em joint Brown measure} and the projections $P(T:B)$ the {\em joint Haagerup--Schultz projections}
of the tuple $T=(T_1,\ldots,T_n)$.
Note the subtle notational difference, which we adopt in this paper:
$P(T,B)$ is the Haagerup-Schultz projection when $T$ is a single operator and $B$ is a Borel subset of $\Cpx$,
while $P(T:B)$ is the joint Haagerup--Schultz projection
when $T=(T_1,\ldots,T_n)$ is an $n$-tuple of commuting operators and $B$ is a Borel subset of $\Cpx^n$.

It is clear from the above that the marginal distributions of $\nu_T$ are the Brown measures $\nu_{T_j}$ of the individual operators.
Schultz also proved (Theorem~6.6 of~\cite{S06}) that $\nu_T$ is characterized by the equality
\[
\tau(\log|\alpha_1T_1+\cdots+\alpha_nT_n-1|)=\int_{\Cpx^n}\log|\alpha_1z_1+\cdots\alpha_nz_n-1|\,d\nu_T(z_1,\ldots,z_n),
\]
holding for all $\alpha_1,\ldots,\alpha_n\in\Cpx$,
and she shows (Theorem~7.1 of~\cite{S06}) that for every polynomial $q$ in $n$ commuting variables,
the Brown measure of the operator $q(T_1,\ldots,T_n)$
equals the push-forward measure $q_*\nu_T$ of $\nu_T$ via $q$.

\medskip

In this this paper we will need some stronger properties of the joint Brown measure and joint Haagerup--Schultz projections.
We did not see how to prove these directly from Schultz's derivation.
Thus, we make a different construction of these objects,
in Sections~\ref{sec:meetsjoins}, \ref{sec:specdistr} and~\ref{sec:decproj},
culminating in Theorems~\ref{thm:main.nu.P}, \ref{thm:PtTQ} and~\ref{thm:PTlargest} and Proposition~\ref{prop:PT*}.
Furthermore, we do this for arbitrary (not necessarily finite) families of commuting operators in $\Mcal$.
In order to make our paper self-contained, we do not use Schultz's results, though it is clear, comparing Schultz's Theorem~\ref{thm:Schultz}
and our construction, that the two constructions yield the same objects in the case of a finite tuple of commuting operators.

After constructing joint Brown measure and joint Haagerup--Schultz projections, 
we relate them,
in the case of a finite tuple of commuting operators,
to various notions of joint spectrum (including the Taylor joint spectrum)
in Section~\ref{sec:jointspec}.
These results show that the joint Brown measure is truly a spectral distribution.

In Section~\ref{sec:ut}
we use the joint Haagerup--Schultz subspaces to find simultaneous Schur-type upper triangular forms
of tuples $T=(T_1,\ldots,T_n)$ of commuting operators in $\Mcal$.
This extends the main result of~\cite{DSZ15}, where the case of a single operator was treated.
More specifically, let 
\begin{equation}\label{eq:rho}
\rho:[0,1]\to\prod_{k=1}^n\big\{z\in\mathbb{C}:\ |z|\leq\|T_k\|\big\}
\end{equation}
be a Peano curve (i.e., surjective and continuous; see Lemma \ref{n peano lemma} for a proof of existence).
We will call $\rho$ a {\em continuous spectral ordering} for $T=(T_1,\cdots,T_n)$.
Let
\begin{equation}\label{eq:D}
\Dc=W^*(\{P(T:\rho([0,t]))\mid 0\le t\le 1\}).
\end{equation}
Note that $\Dc\subseteq\Mcal$ is an abelian von Neumann algebra with separable predual.
We denote by $\Ec_\Dc$ and $\Ec_{\Mcal\cap\Dc'}$,
respectively, the normal, $\tau$-preserving conditional expectations from $\Mcal$ onto $\Dc$ and,
respectively, the relative commutant $\Mcal\cap\Dc'$.

Here is a simultaneous upper triangularization result.
\begin{thm}\label{thm:simultut}
Let $S$ belong to the unital algebra generated by $\{T_1,\ldots,T_n\}$
and let $N=\Ec_\Dc(S)$.
Then the Brown measures $\nu_N$ and $\nu_S$ agree and the Brown measure of $S-N$ is concentrated at $0$.
\end{thm}

Note that, for $S$ as in the theorem, writing $S=f(T_1,\ldots,T_n)$ for a polynomial $f$ in $n$ commuting variables,
by Schultz's result mentioned above, 
$\nu_S$ is the push-forward measure $f_*\nu_T$ of the joint Brown measure of $T$ under $f$.

In Section~\ref{sec:hfcnu}, we consider the Arens multivariate holomorphic functional calculus applied to a commuting tuple $T$.
The above Theorem~\ref{thm:simultut} and the fact about push-forward measures will,
in Theorem~\ref{thm:holomut}, be extended to give the same conclusions also for operators $S$
that arise as $f(T_1,\ldots,T_n)$ for suitable multivariate holomorphic functions $f$.
Though Theorem~\ref{thm:holomut} includes Theorem~\ref{thm:simultut} as a subcase, we have stated and proved the special case separately,
because the proof avoids significant technical difficulties of the proof of the full result.
Furthermore, in Theorem~\ref{thm:PhT}, we prove a multivariate analogue of the push-forward result for joint Brown measures
and a similar property for joint Haagerup--Schultz projections.

In addition to the sections of the paper mentioned above,
Section~\ref{sec:prelims} contains preliminary results (many recalled from the literature),
including:
\S\ref{subsec:Notation} Notation;
\S\ref{subsec:proj} Projections in von Neumann algebras (with some additional examples provided in Appendix~\ref{app:projs});
\S\ref{subsec:sotqn} S.o.t.-quasinilpotent operators;
\S\ref{subsec:HS} Haagerup--Schultz projections;
\S\ref{subsec:hiproj} Hyperinvariant projections;
\S\ref{subsec:spfill} Space filling curves.
In Section~\ref{sec:HS}, we show that Haagerup--Schultz projections
satisfy the natural lattice properties.
In Section~\ref{sec:Pdi} we show that the joint Haagerup--Schultz projection of a direct integral of commuting operators is
the direct integral of the corresponding joint Haagerup--Schultz projections;
this result is needed later in the paper.
Section~\ref{sec:holofc} contains some remarks and results about the multivariate functional calculi
of Arens and Taylor for tuples of commuting operators on Hilbert space.
In particular, it is proved that the Taylor spectrum and the Taylor functional calculus behave well with respect
to direct integral constructions;
this result is used in Section~\ref{sec:hfcnu}.
Section~\ref{sec:sims} contains a proof that the joint spectral distribution measures and
the joint Haagerup--Schultz projections behave well with respect to conjugation by invertible operators.

\medskip
\noindent{\bf Acknowledgements:}
The authors thank Ernst Albrecht for valuable discussions about multivariate holomorphic functional calculus.
I.C.\ and K.D.\ gratefully acknowledge the support of the Hausdorff Institute in Mathematics,
during the Tri\-mester on Von Neumann Algebras,
where some of this research was conducted.
K.D.\ thanks faculty and staff of the University of New South Wales for their kind hospitality during a visit when some of this research
was conducted.
Finally, the authors thank an anonymous referee for many suggestions that improved the exposition of this paper.

\section{Preliminaries}
\label{sec:prelims}

\subsection{Notation}
\label{subsec:Notation}

$\Cpx$ denotes the complex plane.
For $\lambda\in\Cpx$ and $r>0$, $\overline{B_r(\lambda)}$ denotes the closed disk in $\Cpx$ with center $\lambda$ and radius $r$.

Given measurable spaces $X$ and $Y$, a measurable function $f:X\to Y$ and a measure $\mu$ on $X$, we employ the standard notation
$f_*\mu$ for the push--forward measure of $\mu$ under $f$, namely, the measure $\nu$ on $Y$ given by $\nu(B)=\mu(f^{-1}(B))$.

Throughout, $\Mcal$ is a von Neumann algebra equipped with a normal, faithful, tracial state $\tau$.
We will refer to the pair $(\Mcal,\tau)$ as a {\em tracial von Neumann algebra.}
As is standard, completion of $\Mcal$ with respect to the norm $\|a\|_2=\tau(a^*a)^{1/2}$ is denoted $L^2(\Mcal,\tau)$,
and $\Mcal$ will be understood to be represented on $L^2(\Mcal,\tau)$ via the standard representation,
i.e., the Gelfand--Naimark--Segal representation.
Note that, for bounded sequences, convergence in strong operator topology is equivalent to convergence with respect to $\|\cdot\|_2$.

As is standard, by {\em projection} we mean a self-adjoint idempotent element of a C$^*$-algebra or von Neumann algebra.

For $T\in\Mcal$ $\nu_T$ will denote the Brown measure
and $B\subseteq\Cpx$, $P(T,B)$ will denote the corresponding Haagerup--Schultz projection.
If $Q\in\Mcal$ is a nonzero projection and $T\in Q\Mcal Q$, then we will denote by $\nu_T^{(Q)}$ the Brown measure of $T$
as an element of $Q\Mcal Q$ endowed with the trace $\tau(Q)^{-1}\tau\restrict_{Q\Mcal Q}$.

\subsection{On projections in tracial von Neumann algebras.}
\label{subsec:proj}

 We don't claim novelty for any of the results in this subsection, but we do include some proofs, for convenience.

\begin{lemmasubs}\label{lem:PvQ}
Suppose $P$ and $Q$ are projections in $B(\HEu)$ and $(P_n)_{n=1}^\infty$, $(Q_n)_{n=1}^\infty$ are sequences of projections in $B(\HEu)$
that converge to $P$ and $Q$, respectively, in strong operator topology.
Suppose that for every $n$, we have $P_n\le P$ and $Q_n\le Q$.
Then $P_n\vee Q_n$ converges to $P\vee Q$ in strong operator topology as $n\to\infty$.
\end{lemmasubs}
\begin{proof}
Since $P_n\vee Q_n\le P\vee Q$ for every $n$, it will suffice to show that for every $\xi\in(P\vee Q)(\HEu)$, we have
\[
\lim_{n\to\infty}\dist(\xi,(P_n\vee Q_n)(\HEu))=0.
\]
Let $\eps>0$.
Then there exist $x\in P\HEu$ and $y\in Q\HEu$ such that $\|\xi-(x+y)\|<\eps$.
Since $P_nx$ converges to $x$ and $Q_ny$ converges to $y$, for all $n$ sufficiently large, we have $\|\xi-(P_nx+Q_ny)\|<\eps$.
Since $P_nx+Q_ny\in (P_n\vee Q_n)(\HEu)$, we are done.
\end{proof}

\begin{lemmasubs}\label{lem:meetDirected}
Let $I$ be a set directed by a partial ordering $\ge$
and suppose $(P_i)_{i\in I}$ is a decreasing net of projections in $B(\HEu)$, namely, such that $i_1\ge i_2$
implies $P_{i_1}\le P_{i_2}$.
Then
\[
\lim_{i\in I}P_i=\bigwedge_{i\in I}P_i,
\]
where the limit is in strong-operator topology.
Moreover, if $\HEu$ is separable,
then there exists a totally ordered sequence $i(1)\le i(2)\le\cdots$ in $I$ such that
\[
\lim_{n\to\infty}P_{i(n)}=\bigwedge_{i\in I}P_i,
\]
where the limit is in strong-operator topology.
Furthermore, if $P_i$ is a decreasing net of projections in a tracial von Neumann algebra $\Mcal$, then there
exists a totally ordered sequence $i(1)\le i(2)\le\cdots$ in $I$ such that 
\[
\lim_{n\to\infty}P_{i(n)}=\bigwedge_{i\in I}P_i,
\]
where the limit is with respect to $\|\cdot\|_2$.
\end{lemmasubs}
\begin{proof}
We may without loss of generality assume $\bigwedge_{i\in I}P_i=0$.
Then $\bigvee_{i\in I}(1-P_i)=1$.
Let $\xi\in\HEu$ and take $\eps>0$.
Then there exists $n\in\Nats$, $i(1),\ldots,i(n)\in I$ and $v_j\in(1-P_{i(j)})\HEu$
such that $\|\xi-\sum_{j=1}^nv_j\|<\eps$.
Thus, $\|(\bigwedge_{j=1}^nP_{i(j)})\xi\|<\eps$.
Thus, if $i\in I$ and $i\ge i(j)$ for all $j$, then $\|P_i\xi\|<\eps$.
This proves that $P_i$ converges in strong-operator topology to $0$.

In the case that $\HEu$ is separable, by choosing a countable dense subset of $\HEu$ and using a standard diagonalisation argument,
the desired sequence can be found.

In the case that $P_i$ is a net in $\Mcal$, convergence in strong operator topology (for any normal, faithful
representation of $\Mcal$)
implies convergence in $\|\cdot\|_2$.
Since $\tau(P_i)$ converges to $\tau(\bigwedge_{i\in I}P_i)$, we can find an increasing sequence $i(j)$ in $I$
so that $\lim_{j\to\infty}\tau(P_{i(j)})=\tau(\bigwedge_{i\in I}P_i)$.
This completes the proof.
\end{proof}

The following result, which we will use quite frequently, is standard.
(For a proof, one can use the description of the C$^*$-algebra generated by two projections, found in~\cite{RS89}.)
\begin{lemmasubs}\label{lem:tracevw}
Let $(\Mcal,\tau)$ be a tracial von Neumann algebra and suppose $P$ and $Q$ are projections in $\Mcal$.
Then
\[
\tau(P\vee Q)=\tau(P)+\tau(Q)-\tau(P\wedge Q).
\]
\end{lemmasubs}

We will need the following easy consequence:
\begin{lemmasubs}\label{lem:tauPQcorners}
Let $(\Mcal,\tau)$ be a tracial von Neumann algebra and suppose $P$ and $Q$ are projections in $\Mcal$.
Then
\[
\tau(P\wedge Q)+\tau\big((P\vee Q)\wedge(1-Q)\big)=\tau(P).
\]
\end{lemmasubs}
\begin{proof}
Applying Lemma~\ref{lem:tracevw} twice, we have
\begin{align*}
\tau&(P\wedge Q)+\tau\big((P\vee Q)\wedge(1-Q)\big) \\
&=\big(\tau(P)+\tau(Q)-\tau(P\vee Q)\big)+\big(\tau(P\vee Q)+\tau(1-Q)-\tau(P\vee Q\vee(1-Q)\big) \\
&=\big(\tau(P)+\tau(Q)-\tau(P\vee Q)\big)+\big(\tau(P\vee Q)-\tau(Q)\big)=\tau(P).
\end{align*}
\end{proof}

\begin{lemmasubs}\label{lem:PwQ}
Let $(\Mcal,\tau)$ be a tracial von Neumann algebra and suppose $P$ and $Q$ are projections in $\Mcal$ and that
$(P_n)_{n=1}^\infty$, $(Q_n)_{n=1}^\infty$ are sequences of projections in $\Mcal$
that converge to $P$ and $Q$, respectively, in $\|\cdot\|_2$.
Suppose that for every $n$, we have $P_n\le P$ and $Q_n\le Q$.
Then $P_n\wedge Q_n$ converges to $P\wedge Q$ in $\|\cdot\|_2$ as $n\to\infty$.
\end{lemmasubs}
\begin{proof}
By Lemma~\ref{lem:PvQ}, we have $P_n\vee Q_n\to P\vee Q$, in strong operator topology, which implies
$\lim_{n\to\infty}\tau(P_n\vee Q_n)=\tau(P\vee Q)$.
From the identity
\[
\tau(P_n\wedge Q_n)=\tau(P_n)+\tau(Q_n)-\tau(P_n\vee Q_n),
\]
and likewise for $P$ and $Q$, we get $\lim_{n\to\infty}\tau(P_n\wedge Q_n)=\tau(P\wedge Q)$.
Since we have $P_n\wedge Q_n\le P\wedge Q$ for every $n$, this completes the proof.
\end{proof}

Lemma~\ref{lem:PwQ} and an induction argument immediately give the following generalisation.
\begin{lemmasubs}\label{lem:wedgePin}
Let $(\Mcal,\tau)$ be a tracial von Neumann algebra.
Let $I$ be a finite set and suppose that for every $i\in I$ and $n\ge1$, $P^{(i)}\in\Mcal$ is a projection
and $(P^{(i)}_n)_{n=1}^\infty$ is a sequence of projections in $\Mcal$ satisfying $P^{(i)}_n\le P^{(i)}$ for all $n$
and such that $P^{(i)}_n$ converges in $\|\cdot\|_2$ to $P^{(i)}$.
Then
$\bigwedge_{i\in I}P^{(i)}_n$
converges in $\|\cdot\|_2$ to $\bigwedge_{i\in I}P^{(i)}$ as $n\to\infty$.
\end{lemmasubs}

\begin{lemmasubs}\label{lem:pvqwedger}
Let $\Mcal$ be a von Neumann algebra with normal, faithful, tracial state $\tau$.
Suppose $P$, $Q$ and $R$ are projections in $\Mcal$ with $P\le R$.
Then
\[
(P\vee Q)\wedge R=P\vee(Q\wedge R).
\]
\end{lemmasubs}
\begin{proof}
Let $e=(P\vee Q)\wedge R$ and $f=P\vee(Q\wedge R)$.
Since $P\le R$, we have $P\le e$.
Also, we clearly have $Q\wedge R\le e$.
Thus, we have $f\le e$.
Using $P\le R$ and Lemma~\ref{lem:tracevw}, we compute
\begin{align*}
\tau(e)&=\tau(P\vee Q)+\tau(R)-\tau(P\vee Q\vee R) 
=\tau(P\vee Q)+\tau(R)-\tau(Q\vee R) \\
&=\big(\tau(P)+\tau(Q)-\tau(P\wedge Q)\big)+\tau(R)-\big(\tau(Q)+\tau(R)-\tau(Q\wedge R)\big) \\
&=\tau(P)-\tau(P\wedge Q)+\tau(Q\wedge R)
\end{align*}
and also
\begin{align*}
\tau(f)&=\tau(P)+\tau(Q\wedge R)-\tau(P\wedge Q\wedge R) \\
&=\tau(P)+\tau(Q\wedge R)-\tau(P\wedge Q).
\end{align*}
Thus, we have $\tau(e)=\tau(f)$ and we conclude $e=f$.
\end{proof}

See Appendix~\ref{app:projs} for examples showing that the conclusions of Lemmas~\ref{lem:PwQ} and \ref{lem:pvqwedger}
can fail if we don't require existence of a finite trace.

\subsection{S.O.T.-quasinilpotent operators}
\label{subsec:sotqn}

We let $\Mcal$ act on $L^2(\Mcal,\tau)$ via the standard representation, and the standard embedding of $\Mcal$ into $L^2(\Mcal,\tau)$ will be denoted
$x\mapsto\xh$.

The following is part of Theorem~8.1 of~\cite{HS09}, by Haagerup and Schultz:
\begin{thmsubs}\label{thm:HSsot}
For any $T\in\Mcal$,
$|T^n|^{1/n}$ converges as $n\to\infty$ in strong operator topology.
Moreover, letting $A\ge0$ be the strong operator limit of the above sequence, 
for every $r>0$, the spectral projection $1_{[0,r]}(A)$ is equal to the Haagerup--Schultz projection $P(T,\overline{B_r(0)})$
of $T$ for the closed disk of radius $r$ centered at $0$.
\end{thmsubs}
Recall that an operator $T\in B(\HEu)$
is said to be s.o.t.-quasinilpotent if $|T^n|^{1/n}$ converges in strong operator topology to $0$ as $n\to\infty$.
Haagerup and Schultz~\cite{HS09} prove that $T\in\Mcal$ is s.o.t.-quasinilpotent if and only if $\nu_T=\delta_0$.
We will prove (and recall) some basic results about s.o.t.-quasinilpotent operators.

\begin{lemmasubs}\label{lem:sotqnequivs}
Suppose $A_n$ is a bounded sequence in $\Mcal$.
Then the following are equivalent:
\begin{enumerate}[(a)]
\item\label{it:An} $A_n$ converges in strong operator topology to $0$,
\item\label{it:|An|} $|A_n|$ converges in strong operator topology to $0$,
\item\label{it:somet} for some $p>0$, we have $\lim_{n\to\infty}\tau(|A_n|^p)=0$,
\item\label{it:allt} for every $p>0$, we have $\lim_{n\to\infty}\tau(|A_n|^p)=0$.
\end{enumerate}
\end{lemmasubs}
\begin{proof}
The equivalence \eqref{it:An}$\Leftrightarrow$\eqref{it:|An|} follows by considering the polar decompositions of the $A_n$.
The implication \eqref{it:An}$\Rightarrow$\eqref{it:somet} follows because $\|A_n\oneh\|^2=\tau(|A_n|^2)$, so from~\eqref{it:An}
we get $\lim_{n\to\infty}\tau(|A_n|^2)=0$.

Let us show \eqref{it:somet}$\Rightarrow$\eqref{it:allt}. We may, without loss of generality, assume $\|A_n\|\le1$ for all $n$.
Suppose $\|A_n\|_p\to0$ as $n\to\infty$.
If $r<p,$ then $\|A_n\|_r\leq\|A_n\|_p$ for all $n$ and, therefore, $\|A_n\|_r\to0$ as $n\to\infty$.
If $p<r,$ then,
because $\|A_n\|\le1$, by the continuous functional calculus, we have $|A_n|^r\le|A_n|^p$.
Hence, $\tau(|A_n|^r)\to0$
as $n\to\infty.$

The implication~\eqref{it:allt}$\Rightarrow$\eqref{it:An} is standard.
Taking $p=2$ and using $\|A_n\oneh\|^2=\tau(|A_n|^2)$, we conclude $\lim_{n\to\infty}\|A_n\oneh\|=0$.
Since the image of $\oneh$ under that action of the commutant of $\Mcal$ is dense in $L^2(\Mcal,\tau)$,
we get $\lim_{n\to\infty}\|A_nv\|=0$ for all $v$ in a dense subspace of $L^2(\Mcal,\tau)$.
Since the sequence $A_n$ is bounded in norm, we conclude~\eqref{it:An}.
\end{proof}

\begin{lemmasubs}\label{lem:sotqnsubseq}
Take positive integers $n(1)<n(2)<\cdots$ such that for some $p\in(0,1)$ and all $k\ge1$, the inequality $n(k+1)<n(k)/p$ holds.
Suppose $T\in\Mcal$ satisfies that $|T^{n(k)}|^{1/n(k)}$ converges in strong operator topology to $0$ as $k\to\infty$.
Then $T$ is s.o.t.-quasinilpotent.
\end{lemmasubs}
\begin{proof}
We may, without loss of generality, assume $\|T\|\le1$.
Suppose $n(k)\le n\le n(k+1)$.
Then
\begin{multline*}
(T^*)^nT^n=(T^*)^{n(k)}(T^*)^{n-n(k)}T^{n-n(k)}T^{n(k)} \\
\leq\|T^{n-n(k)}\|^2(T^*)^{n(k)}T^{n(k)}\leq(T^*)^{n(k)}T^{n(k)}.
\end{multline*}
Since the function $s\mapsto s^{1/2n}$ is operator monotone, we have
\[
|T^n|^{\frac1n}=\big((T^*)^nT^n\big)^{\frac1{2n}}\le\big((T^*)^{n(k)}T^{n(k)}\big)^{\frac1{2n}}=|T^{n(k)}|^{\frac1n}=\left(|T^{n(k)}|^{\frac1{n(k)}}\right)^{\frac{n(k)}{n}}.
\]
Then, since $\|T\|\le1,$ we have
\[
\left(|T^{n(k)}|^{\frac1{n(k)}}\right)^{\frac{n(k)}{n}}\le\left(|T^{n(k)}|^{\frac1{n(k)}}\right)^{\frac{n(k)}{n(k+1)}}\le\left(|T^{n(k)}|^{\frac1{n(k)}}\right)^p.
\]
By hypothesis and Lemma~\ref{lem:sotqnequivs}, we have
\[
\lim_{k\to\infty}\tau(\left(|T^{n(k)}|^{\frac1{n(k)}}\right)^p)=0.
\]
Thus, we get $\lim_{n\to\infty}\tau(|T^n|^{\frac1n})=0$.
Using Lemma~\ref{lem:sotqnequivs}, we conclude that $T$ is s.o.t.-quasinilpotent.
\end{proof}

The next lemma follows from Lemma~3 of~\cite{DSZ15b}.
\begin{lemmasubs}\label{lem:prodqn}
If $A,T\in\Mcal$ commute and if $T$ is s.o.t.-quasinilpotent, then $AT$ is s.o.t.-quasinilpotent.
\end{lemmasubs}

\begin{lemmasubs}\label{lem:sumqn}
Suppose $T_1,T_2\in\Mcal$ commute and are both s.o.t.-quasinilpotent.
Then $T_1+T_2$ is s.o.t.-quasinilpotent.
\end{lemmasubs}
\begin{proof}
We will prove $\lim_{n\to\infty}\tau(|(T_1+T_2)^{2n}|^{\frac1{2n}})=0$, which by Lemmas~\ref{lem:sotqnequivs} and~\ref{lem:sotqnsubseq}
will imply that $T_1+T_2$ is s.o.t.-quasinilpotent.
Without loss of generality, we assume $\|T_j\|\le1$ ($j=1,2$).

By commutativity, we have
$$(T_1+T_2)^{2n}=A_nT_1^n+B_nT_2^n,$$
where
$$A_n=\sum_{k=n+1}^{2n}\binom{2n}{k}T_1^{k-n}T_2^{2n-k},\qquad B_n=\sum_{k=0}^n\binom{2n}{k}T_1^kT_2^{n-k}.$$
Clearly,
$$\tau(|(T_1+T_2)^{2n}|^{\frac1{2n}})=\|(T_1+T_2)^{2n}\|_{\frac1{2n}}^{\frac1{2n}}=\|A_nT_1^n+B_nT_2^n\|_{\frac1{2n}}^{\frac1{2n}}.$$
Using the standard inequality
$$\|X+Y\|_p^p\leq\|X\|_p^p+\|Y\|_p^p$$
which is valid for every $p\leq 1$ and every $X,Y\in\Mcal$ (see~\cite{FK86}, Theorem 4.9),
we get
$$\tau(|(T_1+T_2)^{2n}|^{\frac1{2n}})\leq\|A_nT_1^n\|_{\frac1{2n}}^{\frac1{2n}}+\|B_nT_2^n\|_{\frac1{2n}}^{\frac1{2n}}.$$
Since $\|A_n\|\leq 2^{2n}$ and $\|B_n\|\leq 2^{2n},$ we get
$$\tau(|(T_1+T_2)^{2n}|^{\frac1{2n}})\leq 2\|T_1^n\|_{\frac1{2n}}^{\frac1{2n}}+2\|T_2^n\|_{\frac1{2n}}^{\frac1{2n}}
=2\tau(|T_1^n|^{\frac1{2n}})+2\tau(|T_2^n|^{\frac1{2n}}).$$
Since $T_1$ and $T_2$ are s.o.t.-quasi-nilpotent, using Lemma~\ref{lem:sotqnequivs} it follows that the right hand side goes to $0$ as $n\to\infty$.
Applying Lemma~\ref{lem:sotqnsubseq} completes the proof.
\end{proof}

Combining Lemmas~\ref{lem:prodqn} and~\ref{lem:sumqn}, we get the following.
\begin{propsubs}\label{prop:polyqn}
Suppose $n\in\Nats$ and $T_1,\ldots,T_n\in\Mcal$ are commuting s.o.t.-quasinilpotent operators.
Suppose $f$ is a polynomial in $n$ commuting variables so that $f(0,\ldots,0)=0$.
Then $f(T_1,\ldots,T_n)$ is s.o.t.-quasinilpotent.
\end{propsubs}

\subsection{Haagerup--Schultz projections}
\label{subsec:HS}

For $T\in\Mcal$ and for a Borel set $B\subseteq\Cpx$,
the Haagerup--Schultz projection $P(T,B)$ is the unique $T$-invariant projection $Q$ with the property that that
the Brown measure $\nu_{TQ}^{(Q)}$ is concentrated in $B$ and the Brown measure $\nu_{(1-Q)T}^{(1-Q)}$ is concentrated in $B^c$.
It is also characterized as the largest $T$-invariant projection $Q$ such that 
the Brown measure $\nu_{TQ}^{(Q)}$ is concentrated in $B$.

The following is a basic fact about Brown measure, proved in~\cite{B86}.

\begin{propsubs}\label{prop:nuTQ+}
Let $T\in\Mcal$ and let $Q\in\Mcal$ be a $T$-invariant projection. If $Q\neq 0,1,$ then
\[
\nu_T=\tau(Q)\nu_{TQ}^{(Q)}+\tau(1-Q)\nu_{(1-Q)T}^{(1-Q)},
\]
where the Brown measures $\nu_{TQ}^{(Q)}$ and $\nu_{(1-Q)T}^{(1-Q)}$ are computed in the algebras $Q\Mcal Q$ and, respectively,
$(1-Q)\Mcal(1-Q)$.
\end{propsubs}

Recall (Theorem 7.1 of~\cite{HS09}), that for $T\in\Mcal$,
$P(T,B)$ is the largest among all of the $T$-invariant projections $R$
in $\Mcal$ such that the Brown measure of $TR$, taken as an element of the von Neumann algebra $R\Mcal R$, with respect 
to the  renormalized trace $\tau(R)^{-1}\tau\restrict_{R\Mcal R}$, is concentrated in $B$.
In particular, we have $\tau(P(T,B))=\nu_T(B)$ for every Borel set $B$.
Clearly, $P(T,B)$ is monotone increasing in $B$.

We will use the fact, which is part of the construction (see Corollary~7.19 of~\cite{HS09}), that when $\Mcal$ acts via a normal representation
on a Hilbert space $\HEu$, for the closed disk $\overline{B_r(\lambda)}$ of radius $r>0$ around $\lambda\in\Cpx$, $P(T,\overline{B_r(\lambda)})$ is the projection onto
the closed subspace $E(T-\lambda,r)$ of $\HEu$ given by
\begin{multline}\label{eq:ET}
E(T-\lambda,r) \\
=\{\xi\in\HEu\mid \exists\,(\xi_n)_{n=1}^\infty\subseteq\HEu,\,\lim_{n\to\infty}\|\xi_n-\xi\|=0,\,
\limsup_{n\to\infty}\|(T-\lambda)^n\xi_n\|^{1/n}\le r\}.
\end{multline}

The following lemma is obvious.
For completeness we provide a quick proof.
\begin{lemmasubs}\label{lem:PTBsymdif}
Suppose $B_1$ and $B_2$ are Borel sets whose symmetric difference is $\nu_T$-null.
Then $P(T,B_1)=P(T,B_2)$.
\end{lemmasubs}
\begin{proof}
It suffices to show $P(T,B_1)=P(T,B_1\cup B_2)$,
so we may without loss of generality assume $B_1\subseteq B_2$.
Then $P(T,B_1)\le P(T,B_2)$, but $\tau(P(T,B_1))=\nu_T(B_1)=\nu_T(B_2)=\tau(P(T,B_2))$.
so we have the desired equality.
\end{proof}

The following is Corollary 7.27 of~\cite{HS09}.
\begin{propsubs}\label{prop:HSPT*}
For a Borel set $B\subseteq\Cpx$,
\[
P(T^*,B)=1-P(T,\Cpx\setminus B^*),
\]
where $B^*=\{\zbar\mid z\in B\}$.
\end{propsubs}

The following is a Lemma~3.3 of~\cite{S06}.

\begin{thmsubs}\label{thm:PTQB}
If $Q\in\Mcal$ is a non-zero $T$-invariant projection, then for all Borel sets $B\subseteq\Cpx$, we have
\[
P(T,B)\wedge Q=P^{(Q)}(TQ,B),
\]
where $P^{(Q)}(TQ,B)$ is the Haagerup-Schultz projection taken in the compressed von Neumann algebra $Q\Mcal Q$, with respect
to the renormalized trace $\tau(Q)^{-1}\tau\restrict_{Q\Mcal Q}$.
\end{thmsubs}

\begin{corsubs}\label{cor:P(1-Q)TB}
Let $T\in\Mcal$.
If $Q\in\Mcal$ is a $T$-invariant projection that is not equal to the identity $1$, then for all Borel sets $B\subseteq\Cpx$, we have
\[
P^{(1-Q)}((1-Q)T,B)=Q\vee P(T,B)-Q=\big(Q\vee P(T,B)\big)\wedge(1-Q),
\]
where $P^{(1-Q)}((1-Q)T,B)$ is the Haagerup-Schultz projection of $(1-Q)T(1-Q)$
taken in the compressed von Neumann algebra $(1-Q)\Mcal(1-Q)$, with respect
to the renormalized trace $(1-\tau(Q))^{-1}\tau\restrict_{(1-Q)\Mcal(1-Q)}$.
\end{corsubs}
\begin{proof}
Using Proposition~\ref{prop:HSPT*} twice and Theorem~\ref{thm:PTQB}, we have
\begin{align*}
P^{(1-Q)}((1-Q)T,B)&=(1-Q)-P^{(1-Q)}(T^*(1-Q),\Cpx\setminus B^*) \\
&=(1-Q)-P(T^*,\Cpx\setminus B^*)\wedge(1-Q) \displaybreak[1]\\
&=(1-Q)-\big(1-P(T,B))\wedge(1-Q) \displaybreak[2]\\
&=(1-Q)-\big(1-Q\vee P(T,B)\big) \\
&=Q\vee P(T,B)-Q 
=\big(Q\vee P(T,B)\big)\wedge(1-Q).
\end{align*}
where $B^*$ is the set obtained from $B$ by complex conjugation.
\end{proof}

\subsection{Hyperinvariant projections}
\label{subsec:hiproj}

Recall that for $S\in B(\HEu)$, a closed subspace $\Vc\subseteq\HEu$ is said to be $S$-hyperinvariant
if $X(\Vc)\subseteq\Vc$ for all $X\in B(\HEu)$ satisfying $XS=SX$, namely, if it is invariant under the commutant $\alg(S)'$
of the algebra of operators generated by $S$ (note: algebra, not a $*$-algebra).
Let $(S_i)_{i\in I}$ be an arbitrary family of elements of $B(\HEu)$.
We say that a closed subspace $\Vc$ is {\em $(S_i)_{i\in I}$-hyperinvariant}
if $X(\Vc)\subseteq\Vc$ for all $X\in B(\HEu)$ satisfying $XS_i=S_iX$ for all $i\in I$, namely, if it is invariant under the commutant
$\alg(\{S_i\mid i\in I\})'$ of the algebra of operators generated by the family.
If $P$ is the orthogonal projection of $\HEu$ onto $\Vc$, then this is equivalent to the condition
$XP=PXP$ for all $X\in B(\HEu)$ satisfying $XS_i=S_iX$ for all $i\in I$.
Such a projection will be called an {\em $(S_i)_{i\in I}$-hyperinvariant} projection.

The following is well known and easy to prove.
\begin{lemmasubs}\label{lem:hyperinvlattice}
Let $J$ be a set and suppose, for every $j\in J$,
$P_j$ is an $(S_i)_{i\in I}$-hyperinvariant projection in $B(\HEu)$.
Then $\bigvee_{j\in J}P_j$ and $\bigwedge_{j\in J}P_j$ are $(S_i)_{i\in I}$-hyper\-invariant projections.
\end{lemmasubs}

The following is well known in the case that $I$ is a singleton set, and the proof in general is an equally easy application
of the double commutant theorem of Murray and von Neumann.

\begin{lemmasubs}\label{lem:hypinvvN}
If $P$ is an $(S_i)_{i\in I}$-hyperinvariant projection, then $P$ lies in the von Neumann algebra generated by $\{S_i\mid i\in I\}$.
\end{lemmasubs}

\subsection{On space-filling curves and probability measures}
\label{subsec:spfill}

In this subsection, we prove an elementary result about a curve mapping onto a space equipped with a probability measure.
We will use it in Section~\ref{sec:ut}.
But first, for completeness, we prove existence of space-filling curves onto polydisks, which is, of course, a well known result.
\begin{lemmasubs}\label{n peano lemma}
Let $n\ge2$ and consider the closed polydisk
\[
\mathbb{D}^n=\{(z_1,\ldots,z_n)\in\Cpx^n\mid\forall j\;|z_j|\le1\}.
\]
Then there exists a continuous surjection
$\rho:[0,1]\to\mathbb{D}^n.$
\end{lemmasubs}
\begin{proof}
We will prove instead the existence of a continuous surjection $\rho^{(d)}:[0,1]\to[0,1]^d$ for all integers $d\ge2$.
This suffices because the polydisk $\mathbb{D}^n$ and the cube $[0,1]^d$ are homeomorphic when $d=2n$.
As usual, given $r\in[0,1]^d$, we write $r=(r_1,\ldots,r_d)$.
Let $\rho^{(2)}:[0,1]\to[0,1]^2$ be the usual (surjective, continuous) Peano curve.
Let $\rho^{(3)}: [0,1]\to[0,1]^3$ be given by the formula
\[
\rho^{(3)}(x)=((\rho^{(2)}(x))_1,\rho^{(2)}((\rho^{(2)}(x))_2)).
\]
It is easy to see that $\rho^{(3)}$ is a continuous surjection.
Define the mapping $\rho^{(k)}:[0,1]\to[0,1]^k$ recursively by the formula
$$\rho^{(k)}(x)=((\rho^{(k-1)}(x))_1,\cdots,(\rho^{(k-1)}(x))_{k-2},\rho^{(2)}((\rho^{(k-1)}(x))_{k-1})),\quad x\in [0,1].$$
Using induction, we easily see that $\rho^{(d)}$ is continuous and surjective.
\end{proof}

Let $K$ be a compact Hausdorff space and suppose $\rho:[0,1]\to K$ is continuous and surjective.
Let $\mu$ be a Borel probability measure on $K$ and let $\sigma$ be the probability measure on $[0,1]$ defined by
\[
\sigma([0,t])=\mu(\rho([0,t])),\qquad(t\in[0,1]).
\]
Note that such a measure $\sigma$ exists;
it is just the restriction to $[0,1]$ of the Lebesgue--Stieltjes measure on $\Reals$ corresponding to the function
\[
\Reals\ni t\mapsto\begin{cases}
0,&t<0 \\
\mu(\rho([0,t])),&0\le t\le 1 \\
1,&t>0.
\end{cases}
\]
Since $\rho^{-1}(\{x\})$ is closed for every $x\in K$, we can define $g:K\to[0,1]$ by 
\[
g(x)=\min(\rho^{-1}(\{x\})).
\]
Let $W=\supp(\mu)$ be the closed support of $\mu$.
\begin{lemmasubs}\label{lem:spfill}
We have
\begin{enumerate}[(a)]
\item\label{it:rhog} $\rho\circ g=\id_K$,
\item\label{it:glsc} $g$ is lower semicontinuous on $K$, and, therefore, Borel measurable,
\item\label{it:g*mu} $g_*\mu=\sigma$,
\item\label{it:sigconc} $\sigma$ is concentrated in $g(W)$,
\item\label{it:ginv} $g$ is the inverse function of the restriction $\rho\restrict_{g(W)}$ of $\rho$ to $g(W)$,
\item\label{it:rho*sig} $\mu=\rho_*\sigma$.
\end{enumerate}
Moreover, there is an isomorphism of von Neumann algebras
\[
L^\infty(W,\mu)\to L^\infty([0,1],\sigma)
\]
sending $f\in L^\infty(W,\mu)$
to the function $h$ given by
\[
h(t)=\begin{cases} f(\rho(t)),&t\in g(W), \\ 0,&t\notin g(W).\end{cases}
\]
\end{lemmasubs}
\begin{proof}
Part~\eqref{it:rhog} is obvious from the definition.

For~\eqref{it:glsc}, we will show that $g^{-1}([0,s])$ is closed in $K$, for every $s\in[0,1]$.
Suppose $x_n$ is a sequence in $g^{-1}([0,s])$ converging to $x\in K$.
Since $x_n\in g^{-1}([0,s])$, it follows that $g(x_n)\in[0,s]$.
That is, $t_n\stackrel{def}{=}\min(\rho^{-1}(\{x_n\}))\in[0,s]$.
Thus, for each $n$, there exists $t_n\in[0,s]$ such that $\rho(t_n)=x_n$.
Passing to a subsequence, if necessary, we may without loss of generality assume that $t_n$ converges to some $t\in[0,s]$.
By continuity of $\rho$, we have $\rho(t)=\lim_{n\to\infty}\rho(t_n)=x$.
Thus, $t\in\rho^{-1}(\{x\})$ and $g(x)\le t\le s$.
So $x\in g^{-1}([0,s])$.

For~\eqref{it:g*mu}, given $t\in[0,1]$, we have
\[
g^{-1}([0,t])=\{x\in K\mid \exists s\in[0,t],\,\rho(s)=x\}=\rho([0,t]).
\]
Thus, we have
\[
g_*\mu([0,t])=\mu(g^{-1}([0,t]))=\mu(\rho([0,t]))=\sigma([0,t]).
\]
Since $g_*\mu$ and $\sigma$ agree on all intervals of the form $[0,t]$, they agree on all Borel subsets of $[0,1]$.

Part~\eqref{it:sigconc} is immediate from~\eqref{it:g*mu}.

For~\eqref{it:ginv}, note that, by definition, $g$ is one-to-one.
Thus, it has an inverse function, mapping $g(K)$ onto $K$.
From~\eqref{it:rhog}, we see that this inverse function must be $\rho\restrict_{g(K)}$.

For~\eqref{it:rho*sig}, given a Borel subset $B$ of $K$, using~\eqref{it:g*mu} and~\eqref{it:rhog} we obtain
\[
\rho_*\sigma(B)=\sigma(\rho^{-1}(B))=\mu(g^{-1}(\rho^{-1}(B)))=\mu((\rho\circ g)^{-1}(B))=\mu(B).
\]

The final statement about the isomorphism $L^\infty(W,\mu)\to L^\infty([0,1])$ now follows directly.
\end{proof}

\section{Lattice properties of Haagerup--Schultz projections}
\label{sec:HS}

Throughout this section, $T\in\Mcal$.
Our goal in this section is to show (Theorem~\ref{thm:countable}) that the map $B\mapsto P(T,B)$ preserves lattice operations.
In the case of a normal operator $T$, $P(T,B)$ is just the spectral projection of $T$ for the set $B$.
However, for general $T\in\Mcal$,
we need not have $P(T,B)=1-P(T,B^c)$ and
$P(T,B_1)$ and $P(T,B_2)$ need not commute for Borel sets $B_1$ and $B_2$.
Of course, the idempotents from the idempotent measure constructed by Schultz~\cite{S06} do satisfy the analogous properties,
and the results of this section could be proved from Schultz's results.
However, here we present straightforward proofs that do not rely on the technology of unbounded affiliated operators.

\begin{lemma}\label{disjoint lemma} If $A_1$ and $A_2$ are disjoint subsets in $\mathbb{C},$ then
$$P(T,A_1)\wedge P(T,A_2)=0.$$
\end{lemma}
\begin{proof} Assume the contrary and denote, for brevity,
$$P=P(T,A_1)\wedge P(T,A_2).$$
Using a basic property of Brown measure (see Theorem~10 in~\cite{DSZ15},
which is, effectively, a restatement of Proposition 2.24 in \cite{HS07}), we obtain
$$\nu_{TP}\leq\frac{\tau(P(T,A_1))}{\tau(P)}\nu_{TP(T,A_1)}.$$
Thus, $\nu_{TP}$ is supported on $A_1.$ Similarly, it is supported on $A_2.$ Since $A_1\cap A_2=\emptyset$,
it follows that $\nu_{TP}$ is supported nowhere.
This contradiction proves the lemma.
\end{proof}

\begin{thm}\label{thm:2sets}
If $A_1$ and $A_2$ are Borel subsets of $\mathbb{C}$, then
\begin{align}
P(T,A_1)\vee P(T,A_2)&=P(T,A_1\cup A_2), \label{eq:2cup} \\
P(T,A_1)\wedge P(T,A_2)&=P(T,A_1\cap A_2). \label{eq:2cap}
\end{align}
\end{thm}
\begin{proof}
Set $B_1=A_1$ and $B_2=A_2\backslash A_1.$ It is immediate that $$B_1\cup B_2=A_1\cup A_2,\quad B_1\cap B_2=\emptyset.$$

We have
$$P(T,A_1),\,P(T,A_2)\leq P(T,A_1\cup A_2).$$
Thus,
\begin{equation}\label{lattice0}
P(T,A_1)\vee P(T,A_2)\leq P(T,A_1\cup A_2).
\end{equation}
Similarly, we have
\begin{equation}\label{lattice1}
P(T,B_1)\vee P(T,B_2)\leq P(T,B_1\cup B_2).
\end{equation}

By Lemma \ref{disjoint lemma}, we have
$$P(T,B_1)\wedge P(T,B_2)=0.$$
Using the equality
\begin{equation}\label{lattice main}
\tau(p\vee q)+\tau(p\wedge q)=\tau(p)+\tau(q),
\end{equation}
we obtain
\begin{multline*}
\tau(P(T,B_1)\vee P(T,B_2))=\tau(P(T,B_1))+\tau(P(T,B_2))
=\nu_T(B_1)+\nu_T(B_2) \\
=\nu_T(B_1\cup B_2)=\tau(P(T,B_1\cup B_2)).
\end{multline*}
It follows now from \eqref{lattice1} that
$$P(T,B_1)\vee P(T,B_2)=P(T,B_1\cup B_2).$$
Thus,
\begin{equation}\label{lattice2}
P(T,A_1\cup A_2)=P(T,B_1)\vee P(T,B_2)\leq P(T,A_1)\vee P(T,A_2).
\end{equation}
A combination of \eqref{lattice0} and \eqref{lattice2} yields~\eqref{eq:2cup}.

To prove~\eqref{eq:2cap}, note first that we have
$$P(T,A_1),\,P(T,A_2)\geq P(T,A_1\cap A_2).$$
Thus,
\begin{equation}\label{lattice4}
P(T,A_1)\wedge P(T,A_2)\geq P(T,A_1\cap A_2).
\end{equation}
On the other hand, from~\eqref{lattice main}, we have
$$\tau(P(T,A_1)\wedge P(T,A_2))=\tau(P(T,A_1))+\tau(P(T,A_2))-\tau(P(T,A_1)\vee P(T,A_2)).$$
Using \eqref{eq:2cup}, we obtain
\begin{multline*}
\tau(P(T,A_1)\wedge P(T,A_2))=\tau(P(T,A_1))+\tau(P(T,A_2))-\tau(P(T,A_1\cup A_2)) \\
=\nu_T(A_1)+\nu_T(A_2)-\nu_T(A_1\cup A_2)=\nu_T(A_1\cap A_2).
\end{multline*}
Thus,
$$\tau(P(T,A_1\cap A_2))=\nu_T(A_1\cap A_2)=\tau(P(T,A_1)\wedge P(T,A_2)).$$
This, combined with \eqref{lattice4}, yields
$$P(T,A_1)\wedge P(T,A_2)=P(T,A_1\cap A_2).$$
This concludes the proof.
\end{proof}

\begin{thm}\label{thm:countable}
If $(A_n)_{n=1}^\infty$ is a sequence of Borel subsets of $\mathbb{C}$, then
\begin{align}
\bigvee_{n=1}^\infty P(T,A_n)&=P(T,\bigcup_{n=1}^\infty A_n), \label{eq:infcup} \\
\bigwedge_{n=1}^\infty P(T,A_n)&=P(T,\bigcap_{n=1}^\infty A_n). \label{eq:infcap}
\end{align}
\end{thm}
\begin{proof}
We have
\[
\bigvee_{n=1}^\infty P(T,A_n)=
\text{s.o.t.-}\lim_{N\to\infty}\bigvee_{n=1}^NP(T,A_n)=\text{s.o.t.-}\lim_{N\to\infty}P(T,\bigcup_{n=1}^NA_n),
\]
where we have used Theorem~\ref{thm:2sets} in the second equality.
So $\le$ holds in~\eqref{eq:infcup}.
But we have
\[
\lim_{N\to\infty}\tau(P(T,\bigcup_{n=1}^NA_n))=\lim_{N\to\infty}\nu_T(\bigcup_{n=1}^NA_n)
=\nu_T(\bigcup_{n=1}^\infty A_n)=\tau(P(T,\bigcup_{n=1}^NA_n)),
\]
from which we conclude equality in~\eqref{eq:infcup}.

The proof of~\eqref{eq:infcap} is similar.
\end{proof}

\section{Meets and joins of Haagerup--Schultz projections of commuting operators}
\label{sec:meetsjoins}

In this section, we begin our construction of joint Brown measure and joint Haagerup-Schultz projections.
We will construct some $T$-hyperinvariant projections corresponding to certain sets
(belonging to an algebra of sets generated by rectangles).
Our construction is based on finite meets and joins of Haagerup--Schultz projections.

Let $I$ be a non-empty set and suppose $T=(T_i)_{i\in I}$ is a family of pairwise commuting element of $\Mcal$.
Let $Z=\prod_{i\in I}\sigma(T_i)$ be the product space, endowed with the product topology.
It is of course, compact, by Tychonoff's theorem.
By a {\em coordinate-finite rectangle} in $Z$ we will mean a product $R=\prod_{i\in I}B_i\subseteq Z$
for non-empty Borel subsets $B_i\subseteq\sigma(T_i)$, with $B_i=\sigma(T_i)$ for all but finitely many $i\in I$.
Let $\Afr_0$ denote the algebra of subsets of $Z$ consisting of the empty set and all finite unions of coordinate-finite rectangles.
Note that every $X\in\Afr_0$ can be written as a disjoint union of finitely many coordinate-finite rectangles.
We begin with coordinate-finite rectangles.

\begin{defi}\label{def:PTR}
If $R=\prod_{i\in I}B_i\subseteq Z$ is a coordinate-finite rectangle,
then we set
\[
P(T:R)=\bigwedge_{i\in I}P(T_i,B_i).
\]
\end{defi}

The goal of this section is to prove Theorem~\ref{thm:ARS}, which shows that
the following definition makes sense.
\begin{defi}\label{def:PTX}
For $X\in\Afr_0$, writing $X=\bigcup_{j=1}^kR^{(j)}$ as a union of finitely many
pairwise disjoint coordinate-finite rectangles $R^{(1)},\ldots,R^{(k)}$, we set
\[
P(T:X)=\bigvee_{j=1}^kP(T:R^{(j)}).
\]
\end{defi}

\begin{lemma}\label{lem:PThyperinv}
Let $R$ be a coordinate-finite rectangle.
Then
$P(T:R)$ is a $T$-hyper\-in\-var\-iant projection.
\end{lemma}
\begin{proof}
For each $i\in I$, the Haagerup-Schultz projection $P(T_i,B_i)$ is a $T_i$-hyper\-in\-vari\-ant projection, and is, therefore,
also $T$-hyperinvariant.
Now by Lemma~\ref{lem:hyperinvlattice}, the result follows.
\end{proof}

\begin{lemma}\label{lem:PTRincr}
If $R_1$ and $R_2$ are coordinate-finite rectangles and $R_1\subseteq R_2$, then
\[
P(T:R_1)\le P(T:R_2).
\]
\end{lemma}
\begin{proof}
Writing $R_j=\prod_{i\in I}B_{i,j}$ for $j=1,2$, we have $B_{i,1}\subseteq B_{i,2}$ for every $i$.
Thus, $P(T_i,B_{i,1})\le P(T_i,B_{i,2})$ for every $i$, and we have
\[
P(T:R_1)=\bigwedge_{i\in I}P(T_i,B_{i,1})\le\bigwedge_{i\in I}P(T_i,B_{i,2})=P(T:R_2).
\]
\end{proof}

\begin{lemma}\label{lem:PTRsymdif}
Suppose that $R=\prod_{i\in I}B_i$ and $R'=\prod_{i\in I}B_i'$ are coordinate-finite rectangles.
Suppose that for each $i\in I$, the symmetric difference of $B_i$ and $B_i'$ is $\nu_{T_i}$-null.
Then $P(T:R)=P(T:R')$.
\end{lemma}
\begin{proof}
The follows immediately from Lemma~\ref{lem:PTBsymdif} and Definition~\ref{def:PTR}.
\end{proof}

The following lemma shows that $P(T:R)$ behaves well under monotone limits.
\begin{lemma}\label{lem:monotoneR}
Suppose $R$ and $R_1,R_2,\ldots$ are coordinate-finite rectangles.
\begin{enumerate}[(i)]
\item\label{it:Rjincr} If $R_1\subseteq R_2\subseteq\cdots$ and $R=\bigcup_{j\ge1} R_j$, then
\[
P(T:R)=\bigvee_{j\ge1}P(T:R_j).
\]
\item\label{it:Rjdecr} If $R_1\supseteq R_2\supseteq\cdots$ and $R=\bigcap_{j\ge1} R_j$, then
\[
P(T:R)=\bigwedge_{j\ge1}P(T:R_j).
\]
\end{enumerate}
\end{lemma}
\begin{proof}
Let $R=\prod_{i\in I}B_i$ and $R_j=\prod_{i\in I} B_{i,j}$.

For~\eqref{it:Rjincr}, let $I_0$ be the finite set of all $i$ such that $B_{i,1}\ne\sigma(T_i)$.
The desired conclusion is equivalent to the convergence of $P(T:R_j)$
to $P(T:R)$ in $\|\cdot\|_2$.
We must have $B_i=B_{i,j}=\sigma(T_i)$ for all $j\ge1$ and all $i\in I\setminus I_0$
and for each $i\in I_0$ we have $B_{i,1}\subseteq B_{i,2}\subseteq\cdots$ with $B_i=\bigcup_{j\ge1}B_{i,j}$.
From the properties of Haagerup--Schultz projections, we have $P(T_i,B_{i,j})\le P(T_i,B_i)$ for all $j$ and
that $P(T_i,B_{i,j})$ converges in $\|\cdot\|_2$ to $P(T_i,B)$ as $j\to\infty$.
Applying Lemma~\ref{lem:wedgePin}, we have that
\[
P(T:R_j)=\bigwedge_{i\in I_0}P(T_i,B_{i,j})
\]
converges in $\|\cdot\|_2$ to $\bigwedge_{i\in I_0}P(T_i,B_i)=P(T:R)$ as $j\to\infty$.

For~\eqref{it:Rjdecr}, for each $i\in I$ we have $B_{i,1}\supseteq B_{i,2}\supseteq\cdots$ with $B_i=\bigcap_{j\ge1}B_{i,j}$.
From the properties of Haagerup--Schultz projections, we have $P(T_i,B_i)=\bigwedge_{j\ge1}P(T_i,B_{i,j})$.
Thus,
\begin{multline*}
P(T:R)=\bigwedge_{i\in I}P(T_i,B_i)=\bigwedge_{i\in I}\left(\bigwedge_{j\ge1}P(T_i,B_{i,j})\right) \\
=\bigwedge_{j\ge1}\left(\bigwedge_{i\in I}P(T_i,B_{i,j})\right)=\bigwedge_{j\ge1}P(T:R_j).
\end{multline*}
\end{proof}

The following lemma shows that $P(T:R)$ satisfies countable lattice properties for decompositions of $R$ in one coordinate.
\begin{lemma}\label{lem:ilattice}
Suppose $R=\prod_{i\in I}B_i$ is a coordinate-finite rectangle.
Fix $i_1\in I$ and consider Borel subsets $B_{i_1,j}\subseteq\sigma(T_{i_1})$ for all integers $j\ge1$.
Let
\[
R_j=\prod_{i\in I}B_{i,j},
\]
where
$B_{i,j}=B_i$ whenever $i\ne i_1$.
\begin{enumerate}[(i)]
\item\label{it:Rjcup} If $B_{i_1}=\bigcup_{j=1}^\infty B_{i_1,j}$, then
\[
P(T:R)=\bigvee_{j\ge1}P(T:R_j).
\]
\item\label{it:Rjcap} If $B_{i_1}=\bigcap_{j=1}^\infty B_{i_1,j}$, then
\[
P(T:R)=\bigwedge_{j\ge1}P(T:R_j).
\]
\end{enumerate}
\end{lemma}
\begin{proof}
If $I=\{i_1\}$, then this follows from the lattice properties of the Haagerup-Schultz projections, Theorem~\ref{thm:countable}.

Suppose $|I|>1$ and let $I_1=\{i_1\}$ and $I_2=I\setminus I_1$.
Let $T^{(2)}=(T_i)_{i\in I_2}$ and $R^{(2)}=\prod_{i\in I_2}B_i$.
Let $Q=P(T^{(2)}:R^{(2)})$.
Since $T_{i_1}$ commutes with $T_i$ for every $i\in I_2$, by Lemma~\ref{lem:PThyperinv}, the projection $Q$ is $T_{i_1}$-invariant.
Applying Theorem~\ref{thm:PTQB},
we have
\[
P(T:R)=P^{(Q)}(T_{i_1}Q,B_{i_1})
\]
and, for every $j\ge1$,
\[
P(T:R_j)=P^{(Q)}(T_{i_1}Q,B_{i_1,j}).
\]
Now using the lattice property of the Haagerup-Schultz projections for the operator $T_{i_1}Q$, if $B_{i_1}=\bigcup_{j\ge1}B_{i_1,j}$, then we have
\[
\bigvee_{j\ge1}P(T:R_j)=\bigvee_{j\ge1}P^{(Q)}(T_{i_1}Q,B_{i_1,j})=P^{(Q)}(T_{i_1}Q,B_{i_1})=P(T:R),
\]
while  if $B_{i_1}=\bigcap_{j\ge1}B_{i_1,j}$, then we have
\[
\bigwedge_{j\ge1}P(T:R_j)=\bigwedge_{j\ge1}P^{(Q)}(T_{i_1}Q,B_{i_1,j})=P^{(Q)}(T_{i_1}Q,B_{i_1})=P(T:R).
\]
\end{proof}

\begin{lemma}\label{lem:rectcompl1}
Let $R=\prod_{i\in I}B_i$ be a coordinate-finite rectangle that is a proper subset of $Z$.
Enumerate the finite, non-empty set $\{i\in I\mid B_i\ne\sigma(T_i)\}$, writing it as $\{i_1,\ldots,i_n\}$, and consider the rectangles
$S^{(j)}=\prod_{i\in I}C_i^{(j)}$,
where 
\[
C_i^{(j)}=\begin{cases}
\sigma(T_i),&i\ne i_j,\\
B_i^c,&i=i_j,
\end{cases}
\]
where $B_i^c=\sigma(T_i)\setminus B_i$.
Thus, we have $Z\setminus R=\bigcup_{j=1}^nS^{(j)}$.
Then
\[
P(T:R)\wedge\left(\bigvee_{j=1}^n P(T:S^{(j)})\right)=0.
\]
\end{lemma}
\begin{proof}
If $n=1$ then this follows because $P(T_{i_1},B_{i_1})\wedge P(T_{i_1},B_{i_1}^c)=0$, which is a property of the Haagerup--Schultz projections.

Supposing $n\ge2$, we proceed by induction on $n$.
Since coordinates in $I\setminus\{i_1,\ldots,i_n\}$ play no role,
we may without loss of generality assume $I=\{i_1,\ldots,i_n\}$.
To ease notation, we will write simply $i_j=j$.
Thus, we have
\[
R=B_1\times B_2\times\cdots\times B_n.
\]
Writing
\[
S^{(n)}=(B_1\cup B_1^c)\times(B_2\cup B_2^c)\times\cdots\times(B_{n-1}\cup B_{n-1}^c)\times B_n^c
\]
and applying Lemma~\ref{lem:ilattice} several times, we obtain
\[
P(T:S^{(n)})\le\left(\bigvee_{j=1}^{n-1}P(T:S^{(j)})\right)\vee P(T:B_1\times\cdots\times B_{n-1}\times B_n^c).
\]
Thus, we get
\[
\bigvee_{j=1}^n P(T:S^{(j)})=\left(\bigvee_{j=1}^{n-1}P(T:S^{(j)})\right)\vee P(T:B_1\times\cdots\times B_{n-1}\times B_n^c).
\]
Let
\[
R'=B_1\times B_2\times\cdots\times B_{n-1}\times\sigma(T_n).
\]
We have
\[
P(T:R)=P(T:R')\wedge P(T_n,B_n).
\]
By the induction hypothesis, we have
\begin{equation}\label{eq:PTR'wv}
P(T:R')\wedge\left(\bigvee_{j=1}^{n-1}P(T:S^{(j)})\right)=0.
\end{equation}
We also have
\[
P(T:B_1\times\cdots\times B_{n-1}\times B_n^c)=P(T:R')\wedge P(T_n,B_n^c)\le P(T:R').
\]
Therefore, applying Lemma~\ref{lem:pvqwedger} to
\[
p=P(T:B_1\times\cdots\times B_{n-1}\times B_n^c),\qquad q=\left(\bigvee_{j=1}^{n-1}P(T:S^{(j)})\right),\qquad r=P(T:R')
\]
(with the lower-case letters corresponding to upper-case letters in Lemma~\ref{lem:pvqwedger})
and using~\eqref{eq:PTR'wv} to show $q\wedge r=0$,
we get 
\begin{multline*}
P(T:R)\wedge\left(\bigvee_{j=1}^n P(T:S^{(j)})\right)
=P(T_n,B_n)\wedge\big(r\wedge(p\vee q)\big) \\
=P(T_n,B_n)\wedge p\le P(T_n,B_n)\wedge P(T_n,B_n^c)=0.
\end{multline*}
\end{proof}

\begin{lemma}\label{lem:rectcompl2}
Let $R=\prod_{i\in I}B_i$ be a coordinate-finite rectangle.
Let $\Lambda$ be any set and suppose for each $\lambda\in\Lambda$, $A^{(\lambda)}$ is a coordinate-finite rectangle
and $A^{(\lambda)}\cap R=\emptyset$.
Then
\[
P(T:R)\wedge\left(\bigvee_{\lambda\in\Lambda}P(T:A^{(\lambda)})\right)=0.
\]
\end{lemma}
\begin{proof}
If $R=Z$, then $A^{(\lambda)}=\emptyset$ and $P(T:A^{(\lambda)})=0$ for every $\lambda$.
So we may assume $R\ne Z$ and Lemma~\ref{lem:rectcompl1} applies,
and we adopt the notation used there.
For each $\lambda$, since $A^{(\lambda)}\cap R=\emptyset$, there is $j\in\{1,\ldots,n\}$ such that $A^{(\lambda)}\subseteq S^{(j)}$ and
(by Lemma~\ref{lem:PTRincr}), $P(T:A^{(\lambda)})\le P(T:S^{(j)})$.
Now the conclusion follows from Lemma~\ref{lem:rectcompl1}.
\end{proof}

Every element $X$ of the algebra of sets $\Afr_0$ can be written as a union,
\[
X=\bigcup_{j=1}^kR^{(j)}
\]
of pairwise disjoint coordinate-finite rectangles $R^{(j)}$.

\begin{thm}\label{thm:ARS}
Let $X\in\Afr_0$.
If 
\[
X=\bigcup_{j=1}^kR^{(j)}\quad\text{and}\quad X=\bigcup_{j=1}^\ell S^{(j)}
\]
are ways of writing $X$ as unions of finitely many pairwise disjoint coordinate-finite rectangles,  then
\begin{equation}\label{eq:PTRPTS}
\bigvee_{j=1}^kP(T:R^{(j)})=\bigvee_{j=1}^\ell P(T:S^{(j)}).
\end{equation}
\end{thm}
\begin{proof}
Since only finitely many coordinates are involved in the rectangles $R^{(j)}$ and $S^{(j)}$,
we may without loss of generality assume $I=\{1,\ldots,n\}$.
Given a rectangle
$R=B_1\times\cdots\times B_n$ and, for each $i$, a Borel partition of $B_i$ into subsets $A_{i,1},\ldots,A_{i,p(i)}$, by repeated
application of Lemma~\ref{lem:ilattice}, we have
\begin{equation}\label{eq:PTRuPTA}
P(T:R)=\bigvee_{1\le q(1)\le p(1),\cdots,1\le q(n)\le p(n)}P(T:A_{1,q(1)}\times\cdots\times A_{n,q(n)}).
\end{equation}
Now, to prove~\eqref{eq:PTRPTS}, we consider a coordinate-wise common refinement.
In particular, writing
\begin{align*}
R^{(j)}&=B_1^{(j)}\times\cdots\times B_n^{(j)}, \\
S^{(j)}&=C_1^{(j)}\times\cdots\times C_n^{(j)},
\end{align*}
for each $1\le i\le n$ there are disjoint Borel sets $A_{i,1},\ldots,A_{i,r(i)}$ such that for each $j$ $B_i^{(j)}$ and $C_i^{(j)}$
are unions of some subcollection of $A_{i,1},\ldots,A_{i,r(i)}$.
Thus, we have
\[
X=\bigcup_{(q_1,\ldots,q_n)\in Q}A_{1,q_1}\times\cdots\times A_{n,q_n}
\]
for a unique subset $Q$ of $\{1,\ldots,r(1)\}\times\cdots\times\{1,\ldots,r(n)\}$.
By repeated application of the formula~\eqref{eq:PTRuPTA} proved for rectangles,
we get
\[
\bigvee_{j=1}^kP(T:R^{(j)})=\bigvee_{(q_1,\ldots,q_n)\in Q}P(T:A_{1,q_1}\times\cdots\times A_{n,q_n})=\bigvee_{j=1}^\ell P(T:S^{(j)}).
\]
\end{proof}

\section{Joint Brown measures}
\label{sec:specdistr}

Let $(T_i)_{i\in I}$ be a commuting family of operators in $\Mcal$, as in Section~\ref{sec:meetsjoins}.
In this section we construct a probability measure $\nu_T$ on $\Cpx^I$ whose marginals are the Brown measures $\nu_{T_i}$
and such that for all $X\in\Afr_0$, $\nu_T(X)=\tau(P(T:X))$, where $P(T:X)$ is as in Definition~\ref{def:PTX}.
This will, of course, be the joint Brown measure of $T$.

\begin{lemma}\label{lem:mu0fadd}
For $X\in\Afr_0$, let $\mu_0(X)=\tau(P(T:X))$.
Then $\mu_0$ is a finitely additive measure on the algebra of sets $\Afr_0$.
\end{lemma}
\begin{proof}
Let $X\in\Afr_0$ and write $X=\bigcup_{j=1}^kR^{(j)}$ for disjoint coordinate-finite rectangles $R^{(1)},\ldots,R^{(j)}$.
We claim that then
\[
\mu_0(X)=\sum_{j=1}^k\mu_0(R^{(j)}).
\]
We use induction on $k$.
For $k=1$ this is a tautology.
Suppose $k\ge2$.
Let $Y=\bigcup_{j=1}^{k-1}R^{(j)}$.
We have
\[
P(T:Y)=\bigvee_{j=1}^{k-1}P(T:R^{(j)})
\]
and,
by the induction hypothesis, $\mu_0(Y)=\sum_{j=1}^{k-1}\mu_0(R^{(j)})$.
By Lemma~\ref{lem:rectcompl2}, we have $P(T:Y)\wedge P(T:R^{(k)})=0$.
Since $P(T:X)=P(T:Y)\vee P(T:R^{(k)})$, we have
\begin{multline*}
\mu_0(X)=\tau(P(T:Y))+\tau(P(T:R^{(k)}))-\tau\big(P(T:Y)\wedge P(T:R^{(k)})\big) \\
=\mu_0(Y)+\mu_0(R^{(k)})=\sum_{j=1}^k\mu_0(R^{(j)})
\end{multline*}
and the claim is proved.

Now, given disjoint $X_1,X_2\in\Afr_0$, writing each of them as a finite union of disjoint coordinate-finite rectangles and using the
claim that we proved above, we conclude $\mu_0(X_1\cup X_2)=\mu_0(X_1)+\mu_0(X_2)$.
\end{proof}

\begin{lemma}\label{lem:Platticefinite}
Let $E,F,G\in\Afr_0$.
Then 
\begin{gather}
E\subseteq G\implies P(T:E)\le P(T:G), \label{eq:Pincr} \\
P(T:E\cup F)=P(T:E)\vee P(T,F), \label{eq:PEcupF} \\
P(T:E\cap F)=P(T:E)\wedge P(T,F). \label{eq:PEcapF}
\end{gather}
\end{lemma}
\begin{proof}
To prove the identity~\eqref{eq:PEcupF}
we may write $E$ and $F$ each as a disjoint union of finite families of coordinate-finite rectangles in such a way
that $E\cap F$ is the union of a common subfamily of each of them.
Namely, there exists a finite collection $(R^{(j)})_{j\in J}$ of pairwise disjoint coordinate-finite rectangles
and there exists a partition $J=J_1\cup J_2\cup J_3$ of $J$ such that
\[
E\cap F=\bigcup_{j\in J_1}R^{(j)},\qquad E\setminus F=\bigcup_{j\in J_2}R^{(j)},\qquad F\setminus E=\bigcup_{j\in J_3}R^{(j)}.
\]
Invoking Definition~\ref{def:PTX}, we have
\begin{multline*}
P(T:E\cup F)=\bigvee_{j\in J_1\cup J_2\cup J_3}P(T:R^{(j)}) \\
=\bigg(\bigvee_{j\in J_1\cup J_2}P(T:R^{(j)})\bigg)\vee\bigg(\bigvee_{j\in J_1\cup J_3}P(T:R^{(j)})\bigg)
=P(T:E)\vee P(T:F)
\end{multline*}
and the identity~\eqref{eq:PEcupF} is proved.

The property~\eqref{eq:Pincr} follows from the identity~\eqref{eq:PEcupF}.

The inequality $\le$ in~\eqref{eq:PEcapF} follows from the property~\eqref{eq:Pincr}.
But from finite additivity of $\mu_0$, the identity~\eqref{eq:PEcupF} and Lemma~\ref{lem:tracevw}, we have
\begin{multline*}
\tau(P(T:E\cap F))=\mu_0(E\cap F)=\mu_0(E)+\mu_0(F)-\mu_0(E\cup F) \\
=\tau(P(T:E))+\tau(P(T:F))-\tau(P(T:E)\vee P(T:F)). \\
=\tau(P(T:E)\wedge P(T:F)),
\end{multline*}
The identity~\eqref{eq:PEcapF} follows from this.
\end{proof}

\begin{lemma}\label{lem:FRU}
Suppose $R$ is a coordinate-finite rectangle and let $\eps>0$.
There there exist coordinate-finite rectangles $F$ and $U$ such that $F$ is compact and $U$ is open, and
\begin{gather*}
F\subseteq R\subseteq U, \\
\mu_0(U)-\eps\le\mu_0(R)\le\mu_0(F)+\eps.
\end{gather*}
\end{lemma}
\begin{proof}
Write $R=\prod_{i\in I}B_i$ for Borel sets $B_i\subseteq\sigma(T_i)$.
Since each of the probability measures $\nu_{T_i}$ is regular (see, for example, Theorem 7.8 of~\cite{F99}),
there exist sequences
\[
F_{i,1}\subseteq F_{i,2}\subseteq\cdots\subseteq B_i\subseteq\cdots\subseteq U_{i,2}\subseteq U_{i,1}
\]
of compact sets $F_{i,j}$ and open sets $U_{i,j}$ such that the sets
\[
B_i\setminus\left(\bigcup_{j\ge1}F_{i,j}\right)\quad\text{and}\quad\left(\bigcap_{j\ge1}U_{i,j}\right)\setminus B_i
\]
are $\nu_{T_i}$-null,
where in the case $B_i=\sigma(T_i)$ we choose $F_{i,j}=\sigma(T_i)=U_{i,j}$ for all $j$.
Consider the compact, respectively, open coordinate-finite rectangles $F_j=\prod_{i\in I}F_{i,j}$ and $U_j=\prod_{i\in I}U_{i,j}$.
Using Lemmas~\ref{lem:PTRsymdif} and~\ref{lem:monotoneR}, we have that $\mu_0(F_j)$ converges to $\mu_0(R)$ from below
and $\mu_0(U_j)$ converges to $\mu_0(R)$ from above.
Selecting $U=U_j$ and $F=F_j$ for suitably large $j$ finishes the proof.
\end{proof}

\begin{lemma}\label{lem:subadd}
Suppose $X\in\Afr_0$ and that $(Y_j)_{j=1}^\infty$
is a sequence of pairwise disjoint elements  of $\Afr_0$ such that
$X=\bigcup_{j=1}^\infty Y_j$.
Then
\[
\mu_0(X)=\sum_{j=1}^\infty\mu_0(Y_j).
\]
\end{lemma}
\begin{proof}
Since $\mu_0$ is finitely additive and since $X$ is a disjoint union of finitely many co\-ordinate-finite
rectangles $R_1,\ldots,R_n$,
it will suffice to show that for all $k\in\{1,\ldots,n\}$,
\[
\mu_0(R_k)=\sum_{j=1}^\infty\mu_0(R_k\cap Y_j).
\]
Thus, we may without loss of generality assume $X$ itself is a coordinate-finite rectangle.
Furthermore, since each $Y_j$ is itself the union of finitely many coordinate-finite rectangles,
we may without loss of generality assume that also each $Y_j$ is such a rectangle.
Let $\eps>0$.
By Lemma~\ref{lem:FRU}, there exists a compact coordinate-finite rectangle $F$ such that $F\subseteq X$ and $\mu_0(X)\le\mu_0(F)+\eps$.
Moreover, for each $j\ge1$, there exists an open coordinate-finite
rectangle $U_j$ such that $Y_j\subseteq U_j$ and $\mu_0(U_j)\le\mu_0(Y_j)+\eps/2^j$.
Since
$(U_j)_{j=1}^\infty$ is an open cover of $F$, there exists $N\in\Nats$ such that
\[
F\subseteq\bigcup_{j=1}^NU_j.
\]
Thus, we have
\[
\mu_0(X)\le\eps+\mu_0(F)\le\eps+\sum_{j=1}^N\mu_0(U_j)\le\eps+\sum_{j=1}^N\left(\mu_0(Y_j)+\eps/2^j\right)
<2\eps+\sum_{j=1}^\infty\mu_0(Y_j).
\]
Letting $\eps\to0$ proves
\[
\mu_0(X)\le\sum_{j=1}^\infty\mu_0(Y_j).
\]
For the reverse inequality, we note that for every $N\in\Nats$, the set $\bigcup_{j=1}^NY_j$ belongs to $\Afr_0$ and
is a subset of $X$, so by finite additivity of $\mu_0$ on $\Afr_0$, we have
\[
\sum_{j=1}^N\mu_0(Y_j)=\mu_0\left(\bigcup_{j=1}^NY_j\right)\le\mu_0(X).
\]
Letting $N\to\infty$ proves
\[
\sum_{j=1}^\infty\mu_0(Y_j)\le\mu_0(X).
\]
\end{proof}

The above lemma shows that $\mu_0$ is a so-called pre-measure on the algebra $\Afr_0$.
Now an application of Carath\'eodory's Extension Theorem (see Theorem~1.11 and Proposition~1.13 of~\cite{F99})
yields the following:
\begin{prop}\label{prop:nuT}
There is a unique Borel probability measure $\nu_T$ on $Z$ extending $\mu_0$, defined by, for every Borel
subset $E\subseteq Z$, 
\[
\nu_T(E)=\inf\left\{\sum_{j=1}^\infty\mu_0(A_j)\;\bigg|\;A_j\in\Afr_0,\,E\subseteq\bigcup_{j=1}^\infty A_j\right\}.
\]
\end{prop}

\begin{defi}\label{def:nuT}
The measure $\nu_T$ from Proposition~\ref{prop:nuT} is the {\em joint Brown measure} of the tuple
$T=(T_i)_{i\in I}$.
We consider it to be a Borel probability measure on $\Cpx^I$ by defining $\nu_T(X)=\nu_T(X\cap Z)$ for Borel sets $X\subseteq\Cpx^I$.
It is (in the case of $I$ finite) the same as the measure constructed by Schultz~\cite{S06}, as is apparent from
Theorem~\ref{thm:Schultz} of this paper, which is from~\cite{S06}.
\end{defi}

\section{Decomposing projections of commuting operators}
\label{sec:decproj}

In this section we use the joint Brown measure $\nu_T$ to extend Definition~\ref{def:PTX} of
$P(T:X)$ to allow arbitrary Borel sets $X$ of $Z$ and
so that we have $\tau(P(T:X))=\nu_T(X)$.
We then prove Theorems~\ref{thm:main.nu.P} and~\ref{thm:PtTQ} and some further properties of these projections.

First, we prove an analogue of the pre-measure result (Lemma~\ref{lem:subadd}) for the projections $P(T:\cdot)$.
\begin{lemma}\label{lem:PTAcup}
If $A\in\Afr_0$ and $A=\bigcup_{j=1}^\infty A_j$ for sets $A_j\in\Afr_0$,
\[
P(T:A)=\bigvee_{j=1}^\infty P(T:A_j).
\]
\end{lemma}
\begin{proof}
The inequality $\ge$ follows from the fact that for every $N\ge1$, by Lemma~\ref{lem:Platticefinite}, we have
\[
\bigvee_{j=1}^NP(T:A_j)=P\bigg(T:\bigcup_{j=1}^NA_j\bigg)\le P(T:A).
\]
We have
\[
\tau\bigg(P\bigg(T,\bigcup_{j=1}^NA_j\bigg)\bigg)=\mu_0\bigg(\bigcup_{j=1}^N A_j\bigg)=\nu_T\bigg(\bigcup_{j=1}^N A_j\bigg).
\]
Since this quantity tends to $\nu_T(A)=\tau(P(T:A))$ as $N\to\infty$,
we have the desired equality.
\end{proof}

The construction contained in the
next proposition can be viewed as doing for our projection-valued set function
$P(T:\cdot)$
something like the construction of a measure from an outer measure
using Caratheodory's Theorem.
Of course, because we already have $\nu_T$ in hand, the proof goes quite easily.
\begin{prop}\label{prop:PtTE}
Let $X\subseteq Z$ be a Borel set.
Define
\[
\Pt(T:X)=\bigwedge\left\{\bigvee_{j=1}^\infty P(T:A_j)\;\bigg|\;A_j\in\Afr_0,\,X\subseteq\bigcup_{j=1}^\infty A_j\right\}.
\]
Then
$\tau(\Pt(T:X))=\nu_T(X)$.
Moreover, if $X\in\Afr_0$, then $\Pt(T:X)=P(T:X)$.
\end{prop}
\begin{proof}
If $A_j\in\Afr_0$, then
\begin{multline}\label{eq:tauveeP}
\tau\bigg(\bigvee_{j=1}^\infty P(T:A_j)\bigg)=\lim_{N\to\infty}\tau\bigg(\bigvee_{j=1}^N P(T:A_j)\bigg)
=\lim_{N\to\infty}\tau\bigg(P\bigg(T:\bigcup_{j=1}^NA_j\bigg)\bigg) \\
=\lim_{N\to\infty}\mu_0\bigg(\bigcup_{j=1}^NA_j\bigg)
=\lim_{N\to\infty}\nu_T\bigg(\bigcup_{j=1}^NA_j\bigg)
=\nu_T\bigg(\bigcup_{j=1}^\infty A_j\bigg).
\end{multline}

Consider the set $\Omega$ consisting of all sequences $(A_j)_{j=1}^\infty$ of elements $A_j\in\Afr_0$
such that $X\subseteq\bigcup_{j=1}^\infty A_j$.
Consider the partial ordering $\ge$ on $\Omega$ defined by $(A_j)_{j=1}^\infty\ge(B_k)_{k=1}^\infty$ if and only if
for every $j\ge1$ there is $k\ge1$ such that $A_j\subseteq B_k$ (i.e., \lq\lq larger\rq\rq{} means \lq\lq finer\rq\rq).
The set $\Omega$ is directed  by $\ge$, for given $(C_j)_{j=1}^\infty$ and $(D_k)_{k=1}^\infty$ in $\Omega$,
a common upper bound is
$(C_j\cap D_k)_{j,k\ge1}$.
Clearly, $(A_j)_{j=1}^\infty\ge(B_k)_{k=1}^\infty$ implies
\[
\bigvee_{j=1}^\infty P(T:A_j)\le\bigvee_{k=1}^\infty P(T:B_k).
\]
Applying Lemma~\ref{lem:meetDirected} and~\eqref{eq:tauveeP}, we have
\begin{multline*}
\tau(\Pt(T:X))=\inf\bigg\{\tau\bigg(\bigvee_{j=1}^\infty P(T:A_j)\bigg)\;\bigg|\;(A_j)_{j=1}^\infty\in\Omega\bigg\} \\
=\inf\bigg\{\nu_T\bigg(\bigcup_{j=1}^\infty A_j\bigg)\;\bigg|\;(A_j)_{j=1}^\infty\in\Omega\bigg\}=\nu_T(X).
\end{multline*}

Now suppose $X\in\Afr_0$.
By considering the sequence $(X,\emptyset,\emptyset,\ldots)$, which belongs to $\Omega$, we see $\Pt(T,X)\le P(T,X)$.
But now for any $(A_j)_{j=1}^\infty\in\Omega$, we have
\[
\bigvee_{j=1}^\infty P(T:A_j)\ge\bigvee_{j=1}^\infty P(T:A_j\cap X)=P(T:X),
\]
where the last equality is by Lemma~\ref{lem:PTAcup}.
\end{proof}

The following \lq\lq up-down\rq\rq consequence of the construction and Lemma~\ref{lem:meetDirected} will be useful.
\begin{lemma}\label{lem:updown}
Let $\Afr_0^\infty$ be the set of all countable unions $\bigcup_{j=1}^\infty A_j$ of elements $A_j\in\Afr_0$.
Let $X\subseteq Z$ be a Borel subset.
Then there is a decreasing sequence $E_1\supseteq E_2\supseteq\cdots$ of sets $E_n\in\Afr_0^\infty$ such that $X\subseteq E_n$
for all $n$ and such that 
\[
\Pt(T:X)=\bigwedge_{n=1}^\infty\Pt(T:E_n).
\]
\end{lemma}

Here is a multivariate analogue of Lemma~\ref{disjoint lemma}.
\begin{lemma}\label{lem:Ptdisj}
Suppose $E,F\subseteq Z$ are disjoint Borel sets.
Then 
\[
\Pt(T:E)\wedge\Pt(T:F)=0.
\]
\end{lemma}
\begin{proof}
Let $\eps>0$.
Because $\nu_T$ is defined as in Proposition~\ref{prop:nuT},
we may choose sequences $(A_j)_{j=1}^\infty$ and $(B_j)_{j=1}^\infty$ in $\Afr_0$ such that
\begin{alignat*}{2}
E&\subseteq\bigcup_{j=1}^\infty A_j,\qquad&\nu_T\bigg(\bigg(\bigcup_{j=1}^\infty A_j\bigg)\setminus E\bigg)<\eps \\
F&\subseteq\bigcup_{k=1}^\infty B_k,\qquad&\nu_T\bigg(\bigg(\bigcup_{k=1}^\infty B_k\bigg)\setminus F\bigg)<\eps.
\end{alignat*}
Then for all $N\ge1$, making use of Lemma~\ref{lem:Platticefinite}, we have
\begin{multline*}
\tau\bigg(\bigg(\bigvee_{j=1}^NP(T:A_j)\bigg)\wedge\bigg(\bigvee_{k=1}^NP(T:B_k)\bigg)\bigg) \\
=\nu_T\bigg(\bigg(\bigcup_{j=1}^NA_j\bigg)\cap\bigg(\bigcup_{k=1}^NB_k\bigg)\bigg)
<2\eps+\nu_T(E\cap F)=2\eps.
\end{multline*}
Now letting $N\to\infty$ and using Lemma~\ref{lem:PwQ}, we obtain
\[
\tau(\Pt(T:E)\wedge\Pt(T:F))\le
\tau\bigg(\bigg(\bigvee_{j=1}^\infty P(T:A_j)\bigg)\wedge\bigg(\bigvee_{k=1}^\infty P(T:B_k)\bigg)\bigg)\le2\eps.
\]
Letting $\eps\to0$ completes the proof.
\end{proof}

Now that the construction and proofs of basic properties of the projections $\Pt(T:X)$ are complete,
we will simplify the notation:
\begin{defi}\label{def:PTBorel}
For every Borel set $X\subseteq Z$, we set
the {\em joint Haagerup--Schultz projection} $P(T:X)$ to be the projection $\Pt(T:X)$ constructed in Proposition~\ref{prop:PtTE},
thus extending Definition~\ref{def:PTX}.
Furthermore, we let $P(T:X)$ be defined for arbitrary Borel subsets $X$ of $\Cpx^I$, by setting
$P(T:X)=P(T:X\cap Z)$.
\end{defi}

We summarize the results of our constructions so far.
\begin{thm}\label{thm:main.nu.P}
For each family $T=(T_i)_{io\in I}$ of commuting operators in $\Mcal$, there exists a Borel probability measure $\nu_T$ on $\Cpx^I$, called the joint Brown measure,
and a family of $T$-hyperinvariant projections $P(T:X)$ defined for Borel subsets $X\subseteq\Cpx^I$, called the joint Haagerup--Schultz projections of $T$,
 satisfying the following:
\begin{enumerate}[(a)]
\item\label{it:nuTi}
Given $i\in I$, the marginal distribution of $\nu_T$ for the projection on the $i$-th coordinate $\Cpx$ of $\Cpx^I$
is the Brown measure $\nu_{T_i}$ of $T_i$.
\item\label{it:tauPTX}
$\tau(P(T:X))=\nu_T(X)$ for all Borel sets $X\subseteq\Cpx^I$.
\item\label{it:PTlattice} For any sequence $(X_n)_{n=1}^\infty$ of Borel sets in $\Cpx^I$, we have
\begin{align*}
\bigvee_{n=1}^\infty P(T:X_n)&=P\bigg(T:\bigcup_{n=1}^\infty X_n\bigg) \\
\bigwedge_{n=1}^\infty P(T:X_n)&=P\bigg(T:\bigcap_{n=1}^\infty X_n\bigg).
\end{align*}
\end{enumerate}
\end{thm}
\begin{proof}
The measure $\nu_T$ and projections $P(T:X)$ satisfying Property~\eqref{it:tauPTX}
are constructed in the results of this section and the previous section, culminating in
Propositions~\ref{prop:nuT} and~\ref{prop:PtTE}.
If $R=\prod_{i\in I}B_i\subseteq\Cpx^I$ is 
what we may call a one-coordinate rectangle,
namely, if $B_i=\Cpx$ for all but at most one value of $i_0\in I$,
then by construction (see Definition~\ref{def:PTR}), $P(T:R)=P(T_{i_0},B_{i_0})$ is a Haagerup--Schultz projection.
Thus, we have
\[
\nu_T(R)=\tau\big(P(T_{i_1},B_{i_0})\big)=\nu_{T_{i_0}}(B_{i_0}),
\]
which implies that property~\eqref{it:nuTi} holds.

The Haagerup--Schultz projection $P(T_{i_0},B_{i_0})$
is $T_{i_0}$-hyperinvariant,
so it is also $T$-hyperinvariant.
Since all of the joint Haagerup--Schultz projections $P(T:X)$ are constructed from these $P(T:R)$ for one-coordinate rectangles $R$ using countable
meets and joins, using Lemma~\ref{lem:hyperinvlattice},
we have that all $P(T:X)$ are $T$-hyperinvariant.

The proof of part~\eqref{it:PTlattice} follows from Lemma~\ref{lem:Ptdisj} just as the proofs of 
Theorems~\ref{thm:2sets} and~\ref{thm:countable} did from Lemma~\ref{disjoint lemma}.
\end{proof}

Our next main theorem is
an extension of Proposition~\ref{prop:nuTQ+}, Theorem~\ref{thm:PTQB} and Corollary~\ref{cor:P(1-Q)TB}
to the setting of several commuting operators.
\begin{thm}\label{thm:PtTQ}
Let $T=(T_i)_{i\in I}$ be a family of commuting elements of $\Mcal$ and suppose a projection $Q\in\Mcal$, $Q\notin\{0,1\}$,
is $T$-invariant.
Then the joint Brown measures satisfy 
\begin{equation}\label{eq:nuTQ+}
\nu_T=\tau(Q)\nu_{TQ}^{(Q)}+\tau(1-Q)\nu_{(1-Q)T}^{(1-Q)}
\end{equation}
and the joint Haagerup--Schultz projections satisfy,
for every Borel subset $X\subseteq\Cpx^I$,
\begin{align}
P^{(Q)}(TQ:X)&=P(T:X)\wedge Q \label{eq:PTQ} \\
P^{(1-Q)}((1-Q)T:X)&=\big(P(T:X)\vee Q\big)\wedge(1-Q), \label{eq:P(1-Q)T}
\end{align}
where $\nu_{TQ}^{(Q)}$ denotes the joint Brown measure of the commuting family $TQ=(T_iQ)_{i\in I}$ of operators
computed in the algebra $Q\Mcal Q$ with respect to the normalized trace $\tau(Q)^{-1}\tau\restrict_{Q\Mcal Q}$,
and similarly for 
$\nu_{(1-Q)T}^{(1-Q)}$
in the algebra $(1-Q)\Mcal(1-Q)$,
while $P^{(Q)}(TQ:X)$ and $P^{(1-Q)}((1-Q)T:X)$ denote the joint Haagerup--Schultz projections computed in the respective algebras
$Q\Mcal Q$ and $(1-Q)\Mcal(1-Q)$.
\end{thm}

We first prove a lemma.
\begin{lemma}\label{lem:PtTQ-rects}
Let $T$ and $Q$ be as in the statement of Theorem~\ref{thm:PtTQ}.
Suppose $X$ is a coordinate-finite rectangle in $\Cpx^I$.
Then the equalities
\begin{equation}\label{eq:nuTQ+atX}
\nu_T(X)=\tau(Q)\nu_{TQ}^{(Q)}(X)+\tau(1-Q)\nu_{(1-Q)T}^{(1-Q)}(X)
\end{equation}
and \eqref{eq:PTQ} and~\eqref{eq:P(1-Q)T} hold.
\end{lemma}
\begin{proof}
We observe that if $X$ is what we may call a one-coordinate rectangle,
namely, if $X=\prod_{i\in I}B_i$ with $B_i=\Cpx$ for all but at most one value of $i\in I$,
then by Theorem~\ref{thm:PTQB} and Corollary~\ref{cor:P(1-Q)TB}, 
the equalities~\eqref{eq:nuTQ+atX}, \eqref{eq:PTQ} and~\eqref{eq:P(1-Q)T} hold.

Now suppose $X=\prod_{i\in I}B_i\subseteq\Cpx^I$ is a coordinate-finite rectangle.
We can find $n\in\Nats$ and pairwise disjoint
coordinate finite rectangles $R^{(1)},\ldots,R^{(n)}$ such that $\bigcup_{k=1}^nR^{(k)}=\Cpx^I\setminus X$.
For example, if $X=B_1\times B_2$, then 
$(\Cpx\times\Cpx)\setminus X$ is the disjoint union 
\[
(B_1^c\times B_2)\;\cup\;(B_1\times B_2^c)\;\cup\;(B_1^c\times B_2^c).
\]
Let $R^{(0)}=X$.
Since every coordinate-finite rectangle $R^{(k)}$ is the intersection of finitely many one-coordinate rectangles $R^{(k,1)},\ldots,R^{(k,p)},$
we have, for each $k\in\{0,1,\ldots,n\}$,
\begin{align*}
P^{(Q)}(TQ:R^{(k)})&=\bigwedge_{j=1}^pP^{(Q)}(TQ:R^{(k,j)})
=\bigwedge_{j=1}^p\big(P(T:R^{(k,j)})\wedge Q\big) \\
&=\bigg(\bigwedge_{j=1}^pP(T:R^{(k,j)})\bigg)\wedge Q
=P(T:R^{(k)})\wedge Q
\end{align*}
and 
\begin{align}
P^{(1-Q)}((1-Q)T:R^{(k)})&=\bigwedge_{j=1}^pP^{(1-Q)}((1-Q)T:R^{(k,j)}) \notag \\
&=\bigwedge_{j=1}^p\Big(\big(P(T:R^{(k,j)})\vee Q\big)\wedge(1-Q)\Big) \notag \\
&=\bigg(\bigwedge_{j=1}^p\big(P(T:R^{(k,j)})\vee Q\big)\bigg)\wedge(1-Q) \notag \\
&\ge\bigg(\bigg(\bigwedge_{j=1}^pP(T:R^{(k,j)})\bigg)\vee Q\bigg)\wedge(1-Q) \label{eq:P(1-Q)wedge} \\
&=\big(P(T:R^{(k)})\vee Q)\wedge(1-Q). \notag
\end{align}
Combining these with Lemma~\ref{lem:tauPQcorners}, we have, for each $k$,
\begin{align}
\tau(Q)\nu^{(Q)}_{TQ}(R^{(k)}&)+\tau(1-Q)\nu^{(1-Q)}_{(1-Q)T}(R^{(k)}) \notag \\
&=\tau\big(P^{(Q)}(TQ:R^{(k)})\big)+\tau\big(P^{(1-Q)}((1-Q)T:R^{(k)})\big) \notag \\
&\ge\tau(P(T:R^{(k)})\wedge Q)+\tau\big((P(T:R^{(k)})\vee Q)\wedge(1-Q)\big) \label{eq:nuRk} \\
&=\tau\big(P(T:R^{(k)})\big)=\nu_T(R^{(k)}). \notag 
\end{align}
Thus, we have
\begin{multline*}
1=\tau(1-Q)+\tau(Q)=\sum_{k=0}^n\tau(1-Q)\nu^{(1-Q)}_{(1-Q)T}(R^{(k)})+\sum_{k=0}^n\tau(Q)\nu^{(Q)}_{TQ}(R^{(k)}) \\
\ge\sum_{k=0}^n\nu_T(R^{(k)})=1.
\end{multline*}
and, for each $k$, the inequality in~\eqref{eq:nuRk} must be an equality.
Consequently, also each inequality~\eqref{eq:P(1-Q)wedge} must be an equality.
\end{proof}

\begin{proof}[Proof of Theorem~\ref{thm:PtTQ}]
Since, by Lemma~\ref{lem:PtTQ-rects},
the two regular measures on the left and right sides of~\eqref{eq:nuTQ+} agree when evaluated at coordinate-finite rectangles,
they must also agree on all Borel subsets of $\Cpx^I$.
This proves the identity~\eqref{eq:nuTQ+}.

Let $\Bfr$ denote the set of all Borel subsets $X\subseteq\Cpx^I$ such that~\eqref{eq:PTQ} and~\eqref{eq:P(1-Q)T} hold.
We will now show that $\Bfr$ is closed under taking countable unions.
Suppose $X_1,X_2,\ldots\in\Bfr$ and let $X=\bigcup_{j=1}^\infty X_j$.
Then
\begin{equation}\label{eq:PQTQX}
\begin{aligned}
P^{(Q)}(TQ:X)&=\bigvee_{j=1}^\infty P^{(Q)}(TQ:X_j)=\bigvee_{j=1}^\infty\big(P(T:X_j)\wedge Q\big) \\
&\le\bigg(\bigvee_{j=1}^\infty P(T:X_j)\bigg)\wedge Q=P(T:X)\wedge Q,
\end{aligned}
\end{equation}
while
\begin{equation}\label{eq:P1-QT1-QX}
\begin{aligned}
P^{(1-Q)}(&(1-Q)T:X)=\bigvee_{j=1}^\infty P^{(1-Q)}((1-Q)T:X_j) \\
&=\bigvee_{j=1}^\infty\big(\big(P(T:X_j)\vee Q\big)\wedge(1-Q)\big) \\
&\le\bigg(Q\vee\bigvee_{j=1}^\infty P(T:X_j)\bigg)\wedge(1-Q)=\big(P(T:X)\vee Q)\wedge(1-Q).
\end{aligned}
\end{equation}
But then, taking traces, adding, using the identity~\eqref{eq:nuTQ+} and invoking Lemma~\ref{lem:tauPQcorners}, we have
\[
\begin{aligned}
\nu_T(X)&=\tau(Q)\nu_{TQ}^{(Q)}(X)+\tau(1-Q)\nu_{(1-Q)T}^{(1-Q)}(X) \\
&=\tau\big(P^{(Q)}(TQ:X)\big)+\tau\big(P^{(1-Q)}((1-Q)T:X)\big)\\
&\le\tau\big(P(T:X)\wedge Q\big)+\tau\big(\big(P(T:X)\vee Q)\wedge(1-Q)\big) \\
&=\tau(P(T:X))=\nu_T(X).
\end{aligned}
\]
We conclude that the inequalities in~\eqref{eq:PQTQX} and~\eqref{eq:P1-QT1-QX} must be equalities.
Thus, $X\in\Bfr$.

Since we have already shown that all coordinate-finite rectangles belong to $\Bfr$, the above result proves
$\Afr_0^\infty\subseteq\Bfr$, where $\Afr_0^\infty$ denotes the set of all countable unions of coordinate-finite rectangles.
Now, given an arbitrary Borel subset $X\subseteq\Cpx^I$, using Lemma~\ref{lem:updown}, 
and working in the directed set $\Omega$ considered in the proof of Proposition~\ref{prop:PtTE}, we can find a sequence $X_n\in\Afr_0^\infty$
such that $X_1\supseteq X_2\supseteq\cdots\supseteq X$ and so that each of $P^{(Q)}(TQ:X_n)$, $P^{(1-Q)}((1-Q)T:X_n)$ and, respectively, $P(T:X_n)$,
converges in strong operator topology to $P^{(Q)}(TQ:X)$, $P^{(1-Q)}((1-Q)T:X)$ and, respectively $P(T:X)$.
Since these are monotone decreasing sequences of projections, by taking limits, we conclude that~\eqref{eq:PTQ} and~\eqref{eq:P(1-Q)T} hold.
\end{proof}

\begin{thm}\label{thm:PTlargest}
Let $T=(T_i)_{i\in I}$ be a family of commuting elements of $\Mcal$.
\begin{enumerate}[(a)]
\item\label{it:conc} If $Q=P(T:X)$ and $Q\notin\{0,1\}$,
then $\nu^{(Q)}_{TQ}$ is concentrated in $X$ and $\nu^{(1-Q)}_{(1-Q)T}$ is concentrated in the complement of $X$.
More precisely, for every Borel set $Y\subseteq\Cpx^I$, we have
\begin{align}
\nu_{TQ}^{(Q)}(Y)&=\nu_T(X)^{-1}\nu_T(X\cap Y) \label{eq:nuTQXY} \\
\nu_{(1-Q)T}^{(1-Q)}(Y)&=(1-\nu_T(X))^{-1}\nu_T(X^c\cap Y), \label{eq:nu1-QTXY}
\end{align}
where $X^c=\Cpx^I\setminus X$.
\item\label{it:Plargest}
If $\nu_T(X)>0$, then
$P(T:X)$ is the largest $T$-invariant projection $Q\in\Mcal$ satisfying that the joint Brown
measure $\nu^{(Q)}_{TQ}$ is concentrated in $X$.
\end{enumerate}
\end{thm}
\begin{proof}
For part~\eqref{it:conc},
let $Q=P(T:X)$ and assume $Q\notin\{0,1\}$.
For every Borel subset $Y$ of $\Cpx^I$, using the lattice property of Theorem~\ref{thm:main.nu.P}\eqref{it:PTlattice}
and using~\eqref{eq:PTQ} but with $Y$ replacing $X$ there, we have
\begin{align*}
\nu_{TQ}^{(Q)}(Y)&=\tau(Q)^{-1}\tau(P^{(Q)}(TQ:Y))
=\nu_T(X)^{-1}\tau\big(P(T:Y)\wedge P(T:X)\big) \\
&=\nu_T(X)^{-1}\tau\big(P(T:X\cap Y)\big)=\nu_T(X)^{-1}\nu_T(X\cap Y),
\end{align*}
which proves~\eqref{eq:nuTQXY}.
Similarly, using~\eqref{eq:P(1-Q)T}, we have
\begin{align*}
\nu_{(1-Q)T}^{(1-Q)}(Y)&=\tau(1-Q)^{-1}\tau\big(P^{(1-Q)}((1-Q)T:Y)\big) \\
&=\big(1-\nu_T(X)\big)^{-1}\tau\big(\big(Q\vee P(T:Y)\big)\wedge(1-Q)\big) \\
&=\big(1-\nu_T(X)\big)^{-1}\tau\big(\big(Q\vee P(T:Y)\big)-Q\big) \\
&=\big(1-\nu_T(X)\big)^{-1}\big(\tau\big(P(T:X)\vee P(T:Y)\big)-\tau\big(P(T:X)\big)\big) \\
&=\big(1-\nu_T(X)\big)^{-1}\big(\nu_T(X\cup Y)-\nu_T(X)\big) \\
&=\big(1-\nu_T(X)\big)^{-1}\nu_T(Y\cap X^c),
\end{align*}
which proves~\eqref{eq:nu1-QTXY}.

For part~\eqref{it:Plargest},
Let $Q=P(T:X)$ and suppose $\nu_T(X)>0$.
If $Q=1$, then the assertion is obvious.
So we may assume $Q\ne1$.
From Theorem~\ref{thm:PTlargest}\eqref{it:conc}, we know that $\nu^{(Q)}_{TQ}$ is concentrated in $X$.
Suppose $Q'$ is any $T$-invariant projection so that $\nu^{(Q')}_{TQ'}$ is concentrated in $X$.
Then, from~\eqref{eq:PTQ}, we have
$$Q'=P^{(Q')}(TQ':X)=P(T:X)\wedge Q'=Q\wedge Q'\leq Q.$$
\end{proof}

Here is a corollary of Theorem~\ref{thm:PTlargest}\eqref{it:conc} and Theorem~\ref{thm:PtTQ}.
\begin{cor}\label{cor:middle}
Let $X_1\subseteq X_2\subseteq\Cpx^I$ be Borel subsets and suppose $\nu_T(X_1)<\nu_T(X_2)$.
Let $p=P(T:X_2)-P(T:X_1)$ and let $pTp=(pT_ip)_{i\in I}$.
Then the joint Brown measure of $pTp$, computed in the algebra $p\Mcal p$ with respect to the renormalized trace
$\tau(p)^{-1}\tau\restrict_{p\Mcal p}$, is the renormalized restriction of $\nu_T$ to $X_2\setminus X_1$, namely,
\[
\nu_{pTp}(Y)=\big(\nu_T(X_2)-\nu_T(X_1)\big)^{-1}\nu_T\big(Y\cap(X_2\setminus X_1)\big).
\]
\end{cor}
\begin{proof}
By Theorem~\ref{thm:PTlargest}\eqref{it:conc}, 
the joint Brown $\nu_{TP(T:X_2)}^{(P(T:X_2))}$
of $TP(T:X_2)$
is the renormalized restriction of $\nu_T$ to $X_2$.
By Theorem~\ref{thm:PtTQ} and the lattice properties,
\[
P\big(TP(T:X_2):X_1)=P(T:X_2)\wedge P(T:X_1)=P(T:X_1).
\]
Thus, by Theorem~\ref{thm:PTlargest}\eqref{it:conc},
the joint Brown measure of
\[
(P(T:X_2)-P(T:X_1))TP(T:X_2)
\]
is the renormalized restriction of $\nu_{TP(T:X_2)}^{(P(T:X_2))}$ to $X_1^c$,
which equals the renormalized restriction of $\nu_T$ to $X_2\setminus X_1$.
But
\begin{multline*}
\big(P(T:X_2)-P(T:X_1)\big)TP(T:X_2) \\
=\big(P(T:X_2)-P(T:X_1)\big)T\big(P(T:X_2)-P(T:X_1)\big)=pTp.
\end{multline*}
\end{proof}

Now we prove an analogue of Proposition~\ref{prop:HSPT*} for joint Haagerup--Schultz projections.
\begin{prop}\label{prop:PT*}
Let $E\subseteq\Cpx^I$ be a Borel set.
Then
\[
P(T^*:E)=1-P(T:\Cpx^I\setminus E^*),
\]
where $E^*=\{(\overline{z_i})_{i\in I}\mid (z_i)_{i\in I}\in E\}$.
Furthermore,
the Brown measure $\nu_{T^*}$ satisfies
\begin{equation}\label{eq:nuT*E}
\nu_{T^*}(E)=\nu_T(E^*).
\end{equation}
\end{prop}
\begin{proof}
Because of the convention $P(T:X)=P(T:X\cap Z)$
mentioned in Definition~\ref{def:PTBorel},
which translates into $P(T^*:X)=P(T^*:X\cap Z^*)$ for Borel subsets $X\subseteq Z^*=\prod_{i\in I}\sigma(T_i^*)$,
we may without loss of generality assume $E\subset Z^*$ and
then it will suffice to show
\begin{equation}\label{eq:PT*E}
P(T^*:E)=1-P(T:Z\setminus E^*).
\end{equation}
First, consider the case of a coordinate-finite rectangle
$R=\prod_{i\in I}B_i$ for Borel subsets $B_i$ of $\sigma(T_i^*)$.
Then the set
\[
I_0=\{i\in I\mid B_i\ne\sigma(T_i^*)\}
\]
is finite.
Using Proposition~\ref{prop:HSPT*} and Definition~\ref{def:PTR}, we have
\[
P(T^*:R)=\bigwedge_{i\in I_0}P(T_i^*,B_i)
=\bigwedge_{i\in I_0}\big(1-P(T_i,\Cpx\setminus B_i^*)\big)
=1-\bigvee_{i\in I_0}P(T_i,\Cpx\setminus B_i^*).
\]
For each $i\in I_0,$ let $F_i=\prod_{j\in I}D_j,$ where
$$
D_j=\begin{cases}
\sigma(T_j)\setminus B_i^*,&j=i\\
\sigma(T_j),&j\neq i.
\end{cases}
$$
Then $P(T_i,\Cpx\setminus B_i^*)=P(T:F_i)$ and we have
\[
P(T^*:R)=1-\bigvee_{i\in I_0}P(T:F_i)
=1-P\bigg(T:\bigcup_{i\in I_0}F_i\bigg)
=1-P(T:Z\setminus R^*),
\]
since $\bigcup_{i\in I_0}F_i$ is the complement of $R^*$.
Thus, we have
\[
\nu_{T^*}(R)=\tau(P(T^*:R))=1-\tau(P(T:Z\setminus R^*))=1-\nu_T(Z\setminus R^*)=\nu_T(R^*).
\]
Consequently, the equality~\eqref{eq:nuT*E} holds for all Borel subsets $E\subseteq Z^*$.

Let $\Afr_0^*$ denote the algebra of subsets of $Z^*$ generated by the set of all coordinate-finite rectangles in $Z^*$.
Suppose $E=\bigcup_{j=1}^nR_j\in\Afr_0^*$, for disjoint coordinate-finite rectangles $R_1,\ldots,R_n\subseteq Z^*$.
Then, using Definition~\ref{def:PTX}, we have
\begin{align*}
P(T^*:E)&=\bigvee_{j=1}^nP(T^*:R_j)
=\bigvee_{j=1}^n(1-P(T:Z\setminus R_j^*))
=1-\bigwedge_{j=1}^nP(T:Z\setminus R_j^*) \displaybreak[2]\\
&=1-P\bigg(T:\bigcap_{j=1}^n(Z\setminus R_j^*)\bigg)
=1-P\bigg(T:Z\setminus\bigg(\bigcup_{j=1}^nR_j^*\bigg)\bigg) \\
&=1-P(T:Z\setminus E^*).
\end{align*}
This shows that~\eqref{eq:PT*E} holds for all $E\in\Afr_0^*$.

Now suppose $E\in(\Afr_0^\infty)^*$, namely, that
$E=\bigcup_{n=1}^\infty A_n$ where $A_1\subseteq A_2\subseteq\cdots$ and $A_n\in\Afr_0^*$ for all $n$.
Then by the lattice properties (Theorem~\ref{thm:main.nu.P}\eqref{it:PTlattice}) and the case just proved, we have 
\begin{multline*}
P(T^*:E)=\bigvee_{n=1}^\infty P(T^*:A_n)
=\bigvee_{n=1}^\infty\big(1-P(T:Z\setminus A_n^*)\big)
=1-\bigwedge_{n=1}^\infty P(T:Z\setminus A_n^*) \\
=1-P\bigg(T:\bigcap_{n=1}^\infty Z\setminus A_n^*\bigg)
=1-P(T:Z\setminus E).
\end{multline*}

Finally, let $E\subseteq Z$ be any Borel set.
Using Lemma~\ref{lem:updown}, there is a decreasing sequence $E_n$ in $(\Afr_0^\infty)^*$ such that $E\subseteq E_n$ for every $n$ and
\[
P(T^*:E)=\bigwedge_{n=1}^\infty P(T^*:E_n).
\]
This of course implies
\begin{equation}\label{eq:nucapE}
\nu_{T^*}(E)=\nu_{T^*}(\bigcap_{n=1}^\infty E_n).
\end{equation}
By the case just proved and the lattice properties, we have
\[
P(T^*:E)=\bigwedge_{n=1}^\infty\big(1-P(T:Z\setminus E_n^*)\big)
=1-P\bigg(T:Z\setminus\bigcap_{n=1}^\infty E_n^*\bigg).
\]
Since $E\subseteq\bigcap_{n=1}^\infty E_n$, using~\eqref{eq:nucapE} and~\eqref{eq:nuT*E}, we have
\[
P(T:Z\setminus E^*)=P\bigg(T:Z\setminus\bigcap_{n=1}^\infty E_n^*\bigg).
\]
Altogether, we have proved~\eqref{eq:PT*E}, as desired.
\end{proof}

\section{Direct integrals of joint Haagerup--Schultz projections}
\label{sec:Pdi}

In this section, we prove some natural results about joint Brown measures, joint Haagerup--Schultz projections and direct integrals.
See~\cite{D81} for background on direct integrals.

Suppose that $\Mcal$ is a von Neumann algebra acting on a Hilbert space $\HEu$ and that
$\Dc\subseteq\Mcal$ is an abelian von Neumann subalgebra with separable predual.
Identifying $\Dc$ with $L^\infty(Z,\omega)$ for a Borel probability measure $\omega$ on a Polish space $Z$,
we may write $\HEu$ and the relative commutant $\Mcal\cap\Dc'$ of $\Dc$ in $\Mcal$ as direct integrals:
\begin{align*}
\HEu&=\int_Z^\oplus \HEu(\zeta)\,d\omega(\zeta) \\
\Mcal\cap\Dc'&=\int_Z^\oplus \Nc(\zeta)\,d\omega(\zeta),
\end{align*}
where $\Nc(\zeta)\subseteq B(\HEu(\zeta))$ is a finite von Neumann algebra equipped with a normal, faithful, tracial state $\tau_\zeta$.
Thus, $\tau$ is the direct integral
\[
\tau=\int_Z^\oplus\tau_\zeta\,d\omega(\zeta).
\]
Moreover, elements $S$ of $\Mcal\cap\Dc'$ are direct integrals
\begin{equation}\label{eq:Sdi}
S=\int_Z^\oplus S(\zeta)\,d\omega(\zeta)
\end{equation}
with $S(\zeta)\in\Nc(\zeta)$.
In particular, we have
\[
\tau(S)=\int_Z\tau_\zeta(S(\zeta))\,d\omega(\zeta).
\]
Here is a special case of Theorem~5.6 of~\cite{DNSZ16}, which we will use repeatedly.

\begin{thm}\label{thm:5.6}
Let $S\in\Mcal\cap\Dc'$ and write its direct integral representation as in~\eqref{eq:Sdi}.
Then for all Borel sets $B\subseteq\Cpx$, we have
\[
\nu_S(B)=\int_Z\nu_{S(\zeta)}(B)\,d\omega(\zeta).
\]
\end{thm}

Naturally enough, we will write
\[
\nu_S=\int_Z\nu_{S(\zeta)}\,d\omega(\zeta)
\]
to indicate that the above description holds.

The next result is that a Haagerup--Schultz projection is the direct integral of the corresponding Haagerup--Schultz projections.

\begin{thm}\label{thm:HSdi}
Suppose $T\in\Mcal\cap\Dc'$ and $B$ is a Borel subset of $\Cpx$.
Then
\begin{equation}\label{eq:PTBdi}
P(T,B)=\int_Z^\oplus P(T(\zeta),B)\,d\omega(\zeta),
\end{equation}
where $P(T(\zeta),B)$ denotes the Haagerup--Schultz projection of $T(\zeta)$ in $(\Nc(\zeta),\tau(\zeta))$.
\end{thm}
\begin{proof}
Let $Q$ be the operator on the right hand side of~\eqref{eq:PTBdi}.
It is clearly a projection
and we have
$$\tau(Q)=\int_Z\tau_\zeta(P(T(\zeta),B)\,d\omega(\zeta).$$
Let $Q_\zeta=P(T(\zeta),B)$.
Assuming $Q\ne0$, let $\omega_Q$ be the probability measure on $Z$ whose Radon--Nikodym
derivative with respect to $\omega$ is 
\[
\frac{d\omega_Q}{d\omega}(\zeta)=\frac{\tau_\zeta(Q_\zeta)}{\tau(Q)}.
\]
Then the von Neumann algebra $Q(\Mcal\cap\Dc') Q$,
the renormalized trace $\tau(Q)^{-1}\tau\restrict_{Q(\Mcal\cap\Dc')Q}$ and the operator $TQ=QTQ$ can be written as the direct integrals
\begin{align*}
Q(\Mcal\cap\Dc')Q&=\int_Z^\oplus Q_\zeta\Nc(\zeta)Q_\zeta\,d\omega_Q(\zeta) \\
\tau(Q)^{-1}\tau\restrict_{Q(\Mcal\cap\Dc')Q}
 &=\int_Z^\oplus \tau_\zeta(Q_\zeta)^{-1}\tau_\zeta\restrict_{Q_\zeta\Nc(\zeta)Q_\zeta}\,d\omega_Q(\zeta) \\
TQ&=\int_Z^\oplus T(\zeta)Q_\zeta\,d\omega_Q(\zeta).
\end{align*}
Thus, by Theorem~\ref{thm:5.6}, if $\nu_{T(\zeta)Q_\zeta}^{(Q_\zeta)}$ denotes the Brown measure of $T(\zeta)Q_\zeta=Q_\zeta T(\zeta)Q_\zeta$
with respect to the normalized trace $\tau_\zeta(Q_\zeta)^{-1}\tau_\zeta\restrict_{Q_\zeta\Nc(\zeta)Q_\zeta}$, then
the Brown measure $\nu_{TQ}^{(Q)}$ of $TQ$ with respect to the trace $\tau(Q)^{-1}\tau\restrict_{Q(\Mcal\cap\Dc')Q}$
is given by, for every Borel subset $X$ of $\Cpx$,
\[
\nu_{TQ}^{(Q)}(X)=\int_Z\nu_{T(\zeta)Q_\zeta}^{(Q_\zeta)}(X)\,d\omega_Q(\zeta).
\]
This implies, in particular, that $\nu_{TQ}^{(Q)}$ is concentrated in $B$, since each $\nu_{T(\zeta)Q_\zeta}^{(Q_\zeta)}(X)$ is concentrated in $B.$ 

Using the characterisation of Haagerup--Schultz projections described at the start of Section~\ref{subsec:HS},
we obtain $Q\leq P(T,B)$.
On the other hand, we have
$$\tau(Q)=\int_Z\tau_\zeta(P(T(\zeta),B)\,d\omega(\zeta)=\int_Z\nu_{T(\zeta)}(B)\,d\omega(\zeta)=\nu_T(B)=\tau(P(T,B)).$$
Hence, $Q=P(T,B).$
\end{proof}

Now we state and, for completeness, prove a couple of basic lemmas about direct integrals of projections.
\begin{lemma}\label{lem:P<=Qdi}
Suppose $P,Q\in B(\HEu)\cap\Dc'$ are projections and write
$$P=\int_Z^\oplus P(\zeta)\,d\omega(\zeta),\quad Q=\int_Z^\oplus Q(\zeta)\,d\omega(\zeta)$$
for projections $P(\zeta),Q(\zeta)\in B(\HEu(\zeta)).$ Then $P\le Q$ if and only if for $\omega$-almost every $\zeta\in Z$, we have $P(\zeta)\le Q(\zeta).$
\end{lemma}
\begin{proof}
We have $P\le Q$ if and only if $PQ=P.$ But
$$P-PQ=\int_Z^\oplus \big(P(\zeta)-P(\zeta)Q(\zeta)\big)\,d\omega(\zeta)$$
and this is zero if and only if $P(\zeta)-P(\zeta)Q(\zeta)=0$ for $\omega$-almost every $\zeta\in Z.$
\end{proof}

\begin{lemma}\label{lem:wedgeveedi}
Let $J$ be a countable set and suppose for each $j\in J$, $P_j\in\Dc'$ is a projection.
Then writing
\[
P_j=\int_Z^\oplus P_j(\zeta)\,d\omega(\zeta),
\]
we have
\begin{align}
\bigwedge_{j\in J}P_j&=\int_Z^\oplus\left(\bigwedge_{j\in J}P_j(\zeta)\right)\,d\omega(\zeta) \label{eq:wedgePidi} \\
\bigvee_{j\in J}P_j&=\int_Z^\oplus\left(\bigvee_{j\in J}P_j(\zeta)\right)\,d\omega(\zeta). \label{eq:veePidi}
\end{align}
\end{lemma}
\begin{proof}
We will prove~\eqref{eq:wedgePidi} and then~\eqref{eq:veePidi} will follow by using
\[
\bigvee_{j\in J}P_j=1-\bigwedge_{j\in J}(1-P_j).
\]

Let $Q$ be the right hand side of~\eqref{eq:wedgePidi}.
Using Lemma~\ref{lem:P<=Qdi}, we clearly have $P_j\ge Q$ for all $j$, so we get $\ge$ in~\eqref{eq:wedgePidi}.
On the other hand, since $\bigwedge_{j\in J}P_j$ commutes with $\Dc$, it has a direct integral decomposition
\[
\bigwedge_{j\in J}P_j=\int_Z^\oplus F(\zeta)\,d\omega(\zeta),
\]
for some projections $F(\zeta)\in B(\HEu(\zeta))$.
Thus, for every $j_0\in J$, we have
\[
\int_Z^\oplus P_{j_0}(\zeta)\,d\omega(\zeta)=P_{j_0}\ge
\bigwedge_{j\in J}P_j=\int_Z^\oplus F(\zeta)\,d\omega(\zeta)
\ge\int_Z^\oplus\left(\bigwedge_{j\in J}P_j(\zeta)\right)\,d\omega(\zeta).
\]
Thus, using Lemma~\ref{lem:P<=Qdi}, we find an $\omega$-null set $N_{j_0}$ such that for all $\zeta\in Z\setminus N_{j_0}$, we have
\[
P_{j_0}(\zeta)\ge F(\zeta)\ge\bigwedge_{j\in J}P_j(\zeta).
\]
Now letting $N=\bigcup_{j\in J}N_j$, we get $F(\zeta)=\bigwedge_{j\in J}P_j(\zeta)$ for all $\zeta\in Z\setminus N$.
Since $J$ is countable, also $N$ is a null set.
\end{proof}

\begin{thm}\label{thm:nuTdi}
Suppose $T=(T_i)_{i\in I}$ is a family of commuting elements $T_i\in\Mcal\cap\Dc'$,
and suppose that we have direct integral decompositions
\[
T_i=\int_Z^\oplus T_i(\zeta)\,d\omega(\zeta)
\]
so that $T(\zeta):=(T_i(\zeta))_{i\in I}$ is a family of commuting elements of $\Nc(\zeta)$ for $\omega$-almost every $\zeta$.
(Note that if $I$ is countable, then the existence of such decompositions is guaranteed.)
Then for every Borel set $X\subseteq\Cpx^I$, we have
\begin{align}
\nu_T(X)&=\int_Z\nu_{T(\zeta)}(X)\,d\omega(\zeta), \label{eq:nuTdi} \\
P(T:X)&=\int_Z^\oplus P(T(\zeta),X)\,d\omega(\zeta). \label{eq:PTdi}
\end{align}
\end{thm}
\begin{proof}
Let us first prove~\eqref{eq:PTdi} when $X=R$ is a coordinate-finite rectangle.
Then $R=\prod_{i\in I}B_i$ for Borel sets $B_i\subseteq\Cpx$ with $B_i=\Cpx$ for all $i\in I\setminus I_0$
for some finite set $I_0\subseteq I$.
Now, using Theorem~\ref{thm:HSdi} and Lemma~\ref{lem:wedgeveedi}, we have
\begin{multline*}
P(T:R)=\bigwedge_{i\in I_o}P(T_i,B_i)
=\bigwedge_{i\in I_0}\int_Z^\oplus P(T_i(\zeta),B_i)\,d\omega(\zeta) \\
=\int_Z^\oplus\left(\bigwedge_{i\in I_0}P(T_i(\zeta),B_i)\right)\,d\omega(\zeta)
=\int_Z^\oplus P(T(\zeta):R)\,d\omega(\zeta).
\end{multline*}
This proves~\eqref{eq:PTdi} when $X=R$ is a coordinate-finite rectangle.

Taking $\tau$ of both sides, we get
\[
\nu_T(R)=\tau(P(T:R))=\int_Z\tau_\zeta\big(P(T(\zeta),R)\big)\,d\omega(\zeta)
=\int_Z\nu_{T(\zeta)}(R)\,d\omega(\zeta),
\]
which proves~\eqref{eq:nuTdi} when $X=R$ is a coordinate-finite rectangle.
Now let $\mu$ denote the measure such that $\mu(X)$ is given by the right hand side of~\eqref{eq:nuTdi}.
We have just shown that $\mu$ and $\nu_T$ agree when evaluated on coordinate-finite rectangles.
Since these generate the Borel $\sigma$-algebra of $\Cpx^I$ and since both $\mu$ and $\nu_T$ are regular,
we get $\nu_T=\mu$, namely, the equality~\eqref{eq:nuTdi} holds for all Borel sets $X$.

It remains to show that the equality~\eqref{eq:PTdi} holds for all Borel sets $X$.
Recall that $\Afr_0^\infty$ denotes that set of all countable unions of coordinate finite rectangles in $\Cpx^I$.
Suppose $E\in\Afr_0^\infty$ and let $R_1,R_2,\ldots$ be coordinate-finite rectangles such that $E=\bigcup_{j=1}^\infty R_j$.
Using Lemma~\ref{lem:wedgeveedi} and the lattice properties (Theorem~\ref{thm:main.nu.P}\eqref{it:PTlattice}), we get
\begin{multline*}
P(T:E)=\bigvee_{j=1}^\infty P(T:R_j)
=\bigvee_{j=1}^\infty\int_Z^\oplus P(T(\zeta):R_j)\,d\omega(\zeta) \\
=\int_Z^\oplus\left(\bigvee_{j=1}^\infty P(T(\zeta):R_j)\right)\,d\omega(\zeta)
=\int_Z^\oplus P(T(\zeta):E)\,d\omega(\zeta).
\end{multline*}
Thus, the equality~\eqref{eq:PTdi} holds when $X\in\Afr_0^\infty$.

Now let $X$ be an arbitrary Borel subset of $\Cpx^I$.
Using Lemma~\ref{lem:updown}, we find a sequence $(E_n)_{n=1}^\infty$ in $\Afr_0^\infty$ such that
\[
E_1\supseteq E_2\supseteq\cdots\supseteq X
\]
and
$$P(T:X)=\bigwedge_{n=1}^\infty P(T:E_n).$$
Thus, using Lemma~\ref{lem:wedgeveedi} again, we have
\begin{equation}\label{eq:PTXEn}
P(T:X)=\int_Z^\oplus\left(\bigwedge_{n=1}^\infty P(T(\zeta):E_n)\right)\,d\omega(\zeta).
\end{equation}
Clearly, for all $\zeta\in Z$ we have
$$\bigwedge_{n=1}^\infty P(T(\zeta):E_n)\ge P(T(\zeta):X).$$
Let
\[
d(\zeta)=\tau_\zeta\left(\left(\bigwedge_{n=1}^\infty P(T(\zeta):E_n)\right)-P(T(\zeta):X)\right).
\]
Then $d(\zeta)\ge0$ for all $\zeta$.
But
\begin{multline*}
\int_Zd(\zeta)\,d\omega(\zeta)
=\tau\left(\int_Z^\oplus\left(\bigwedge_{n=1}^\infty P(T(\zeta):E_n)\right)\right)\,d\omega(\zeta)-\int_Z\tau_\zeta(P(T(\zeta):X))\,d\omega(\zeta) \\
=\tau(P(T:X))-\int_Z\nu_{T(\zeta)}(X)\,d\omega(\zeta)
=\nu_T(X)-\nu_T(X)=0,
\end{multline*}
where we have used~\eqref{eq:nuTdi} in the penultimate equality.
We conclude that $d(\zeta)=0$ for $\omega$-almost every $\zeta\in Z.$ Thus, we have
\[
\bigwedge_{n=1}^\infty P(T(\zeta):E_n)=P(T(\zeta):X)
\]
for almost every $\zeta$, and from~\eqref{eq:PTXEn} we get the desired equality~\eqref{eq:PTdi}.
\end{proof}

\section{Joint spectra of commuting operators}
\label{sec:jointspec}

In this section, we we show that the joint Brown measure $\nu_T$
of a finite commuting family $T=(T_1,\ldots,T_n)$ of elements of $\Mcal$ is a joint spectral distribution measure.
Namely, we show that for a finite commuting tuple $T$,
the support of $\nu_T$ is contained in several versions of the joint spectrum of $T$.

The classical definition of the joint spectrum
for a finite tuple of elements $a=(a_1,\ldots,a_n)$ in a commutative, unital Banach algebra $\Afr$
is the set $\sigma_\Afr(a)$ of all $\lambda=(\lambda_1,\ldots,\lambda_n)\in\Cpx^n$ such that the (algebraic) ideal generated by
the set $\{a_1-\lambda_1,\ldots,a_n-\lambda_n\}$ is not all of $\Afr$.
This coincides with the set of all values $(\phi(a_1),\ldots,\phi(a_n))\in\Cpx^n$ where $\phi$ ranges over the set of non-zero characters of $\Afr$.
This notion of spectrum may, of course, depend on the algebra $\Afr$;
the spectrum $\sigma_\Afr(a)$ decreases when $\Afr$ increases.

Joseph Taylor~\cite{T70} defined a joint spectrum $\Sp(T)$ for commuting bounded operators $T=(T_1,\ldots,T_n)$
on a Banach space $X$ and, in~\cite{T70b}, proved that a holomorphic functional calculus is valid for this notion of spectrum.
Taylor's joint spectrum is a subset of the joint spectrum $\sigma_\Afr(T)$ considered in any commutative unital Banach subalgebra $\Afr$
of $B(X)$ containing $T_1,\ldots, T_n$, so his functional calculus is richer than the
functional calculus due to Arens~\cite{Ar61}.
(See the exposition found after the proof of Lemma~\ref{lem:holofX} and further references mentioned there.)

Robin Harte~\cite{Ha72} defined the following notion of joint spectrum
for an $n$-tuple of (not necessarily commuting) elements $A=(A_1,\ldots,A_n)$ in a unital Banach algebra $\Afr$.
\renewcommand{\labelitemi}{$\bullet$}
\begin{itemize}
\item The left joint spectrum $\spec_\Afr^\ell(A)$ is the set of all $(\lambda_1,\ldots,\lambda_n)\in\Cpx^n$ such that the (algebraic)
left ideal of $\Afr$ generated by 
the set $\{A_1-\lambda_1,\ldots,A_n-\lambda_n\}$ is not all of $\Afr$.
\item The right joint spectrum $\spec_\Afr^r(A)$ is the set of all $(\lambda_1,\ldots,\lambda_n)\in\Cpx^n$ such that the (algebraic)
right ideal of $\Afr$ generated by 
$\{A_1-\lambda_1,\ldots,A_n-\lambda_n\}$ is not all of $\Afr$.
\item The joint spectrum $\spec_\Afr(A)$ is $\spec_\Afr^\ell(A)\cup\spec_\Afr^r(A)$.
\end{itemize}
The Harte joint spectrum $\spec_\Afr(A)$ may be empty, but it is always a compact subset of the product of the usual spectra of the $A_i$.

The following observations are standard, but for convenience we indicate some proofs.
\begin{prop}\label{prop:char}
Let $n\ge1$ and let $A=(A_1,\ldots,A_n)$ be any $n$-tuple in a unital Banach algebra $\Afr$.
\begin{enumerate}[(i)]
\item\label{it:a*} If $\Afr$ is a Banach $*$-algebra, then letting $A^*=(A_1^*,\ldots,A_n^*)$, we have
\[
\spec_\Afr^\ell(A^*)=\big(\spec^r_\Afr(A)\big)^*,
\qquad\spec_\Afr(A^*)=\big(\spec_\Afr(A)\big)^*,
\]
where $(\spec^r_\Afr(A))^*$, is the set obtained from $\spec^r_\Afr(A)$ by taking the complex conjugate in every coordinate, etc.
\item\label{it:a*a}
If $\Afr$ is a C$^*$-algebra, then $(\lambda_1,\ldots,\lambda_n)\in\spec_\Afr^\ell(A)$ if and only if the positive element
\begin{equation}\label{eq:Al*Al}
(A_1-\lambda_1)^*(A_1-\lambda_1)+\cdots+(A_n-\lambda_n)^*(A_n-\lambda_n)
\end{equation}
is not invertible in $\Afr$.
\item\label{it:aa*}
If $\Afr$ is a C$^*$-algebra, then $(\lambda_1,\ldots,\lambda_n)\in\spec_\Afr^r(A)$ if and only if the positive element
\[
(A_1-\lambda_1)(A_1-\lambda_1)^*+\cdots+(A_n-\lambda_n)(A_n-\lambda_n)^*
\]
is not invertible in $\Afr$.
\end{enumerate}
\end{prop}
\begin{proof}
Part~\eqref{it:a*} is elementary.
For~\eqref{it:a*a}, assume without loss of generality $\lambda_j=0$ for all $j$.
Let $D$ be the element in~\eqref{eq:Al*Al}.
We observe that if the element $D$ is invertible in $\Afr$, then taking $B_j=D^{-1}A_j^*$ yields $B_1A_1+\cdots+B_nA_n=1$,
which shows $(0,\ldots,0)\notin\spec^\ell_\Afr(A)$.
For the opposite implication, consider
\[
X=\left(\begin{smallmatrix}
A_1&0&\cdots&0\\
A_2&0&\cdots&0\\
\vdots&\vdots&\vdots&\vdots \\
A_n&0&\cdots&0
\end{smallmatrix}\right)\in M_n(\Afr).
\]
Note that $X^*X=\diag(D,0,\ldots,0)$.
If $(0,\ldots,0)\notin\spec^\ell_\Afr(A)$,
then there exist $B_1\ldots,B_n\in\Afr$ such that $B_1A_1+\cdots+B_nA_n=1$.
Letting
\[
Y=\left(\begin{smallmatrix}
B_1&B_2&\cdots&B_n\\
0&0&\cdots&0\\
\vdots&\vdots&\vdots&\vdots \\
0&0&\cdots&0
\end{smallmatrix}\right)\in M_n(\Afr),
\]
we have $YX=\diag(1,0,\ldots,0)$.
Thus, we have
\[
\diag(D,0,\ldots,0)=X^*X\ge\|Y\|^{-2}X^*Y^*YX=\|Y\|^{-2}\diag(1,0,\ldots,0),
\]
so that $D\ge\|Y\|^{-2}1$ is invertible in $\Afr$.
This completes the proof of~\eqref{it:a*a}.

The proof of~\eqref{it:aa*} follows in a similar fashion, or by combining~\eqref{it:a*} and~\eqref{it:a*a}.
\end{proof}

As a corollary, we have that
the left and right Harte spectra in C$^*$-algebras $\Afr$ enjoy
spectral permanence, namely, do not depend on the C$^*$-algebra $\Afr$.
\begin{cor}\label{cor:specperm}
If $A$ is an $n$-tuple of elements of a unital C$^*$-algebra $\Afr$ and if $\Afr\subset\Bfr$ is a unital inclusion of C$^*$-algebras,
then $\spec^\ell_\Afr(A)=\spec^\ell_\Bfr(A)$, $\spec^r_\Afr(A)=\spec^r_\Bfr(A)$, and  $\spec_\Afr(A)=\spec_\Bfr(A)$.
\end{cor}
\begin{proof}
This follows immediately from the characterisations~\eqref{it:a*a} and~\eqref{it:aa*} of Proposition~\ref{prop:char} and the spectral
permanence property of individual elements in C$^*$-algebras, namely, that an element $D\in\Afr$ is invertible in $\Afr$ if and only if 
it is invertible in $\Bfr$.
\end{proof}


Next we show Harte's joint spectrum $\spec(T)$ is contained in
Taylor's joint spectrum $\Sp(T)$, in the case of commuting elements 
acting on Hilbert space.
Taylor's spectrum of a commuting $n$-tuple $T=(T_1,\ldots,T_n)$ is defined in terms a the Koszul complex, which is a finite-length chain complex of
exterior powers and maps.
The last of these maps is
\[
\delta_0:\HEu^{\oplus n}\to\HEu,
\]
given by
\[
\delta_0(v_1\oplus\cdots\oplus v_n)=\sum_{j=1}^nT_jv_j.
\]
The first of these maps is called $\delta_n$.
For convenience, we make  suitable renaming and changes of sign to identify $\delta_n$ with
\[
\delta_n':\HEu\to\HEu^{\oplus n}
\]
given by
\begin{equation}\label{eq:deltan'}
\delta_n'(x)=T_1x\oplus T_2x\oplus\cdots\oplus T_nx.
\end{equation}
(To be more precise, in Taylor's notation from Section~1 of~\cite{T70}, the domain $\HEu\otimes E^n_n$ for $\delta_n$ is identified
with $\HEu$ in the obvious way, choosing the basis element $e_1\wedge\cdots\wedge e_n$ for $E^n_n$,
while the range $\HEu\otimes E^n_{n-1}$ for Taylor's $\delta_n$ is identified with $\HEu^{\oplus n}$ by choosing the basis
$\{(-1)^{i-1}e_1\wedge\cdots\wedge\widehat{e_i}\wedge\cdots\wedge e_n\}_{i=1}^n$ for $E^n_{n-1}$;
with these identifications, $\delta_n$ is transformed into $\delta_n'$.)

\begin{lemma}\label{lem:delta0}
Let $n\in\Nats$ and let $(T_1,\ldots,T_n)$ be an $n$-tuple of commuting elements of $B(\HEu)$.
Then $\delta_0$ is surjective if and only if there exist $B_1,\ldots,B_n\in B(\HEu)$ such that
\[
T_1B_1+\cdots+T_nB_n=1.
\]
\end{lemma}
\begin{proof}
Sufficiency is clear.
For necessity, suppose $\delta_0$ is surjective.
Let $\Vc=(\ker\delta_0)^{\perp}$.
Then the restriction of $\delta_0$ to $\Vc$ is injective and surjective, so by the Open Mapping Theorem, it is an isomorphism.
The inverse of $\delta_0$, when viewed as a mapping from $\HEu$ into $\HEu^{\oplus n}$, is of the form
$v\mapsto(B_1v,B_2v,\ldots,B_nv)$, for some $B_1,\ldots,B_n\in B(\HEu)$.
Thus, we have
$T_1B_1v+\cdots+T_nB_nv=v$ for all $v\in\HEu$, so $T_1B_1+\cdots+T_nB_n=1$.
\end{proof}

\begin{lemma}\label{lem:deltan}
Let $n\in\Nats$ and let $(T_1,\ldots,T_n)$ be an $n$-tuple of commuting elements of $B(\HEu)$.
Then $\delta_n$ is injective and has closed range if and only if there exist $S_1,\ldots,S_n\in B(\HEu)$ such that
\begin{equation}\label{eq:St1}
S_1T_1+\cdots+S_nT_n=1.
\end{equation}
\end{lemma}
\begin{proof}
The map $\delta_n$ is injective and has closed range if and only if the map $\delta_n'$
given in~\eqref{eq:deltan'} has these properties.

Suppose $\delta_n'$ is injective and has closed range.
Let $\Wc$ be the image of $\delta_n'$ and let $P_\Wc:\HEu^{\oplus n}\to\Wc$ be the orthogonal projection onto $\Wc$.
By the Open Mapping Theorem, there exists a bounded operator $S:\Wc\to\HEu$ such that $S\circ\delta_n'$ is the identity map on $\HEu$.
Let $\gamma_j:\HEu\to\HEu^{\oplus n}$ be the canonical isometry $v\mapsto 0\oplus\cdots\oplus 0\oplus v\oplus0\oplus\cdots\oplus0$ onto
the $j$-th summand.
Let $S_j=S\circ P_\Wc\circ\gamma_j\in B(\HEu)$.
If $v=v_1\oplus\cdots\oplus v_n\in\Wc$, then
\[
Sv=S\circ P_\Wc v=\sum_{j=1}^nS\circ P_\Wc\circ\gamma_jv_j=\sum_{j=1}^nS_j v_j.
\]
Thus, for all $\xi\in\HEu$, we have
\[
\xi=S\circ\delta_n'\xi=S(T_1\xi\oplus\cdots\oplus T_n\xi)
=\sum_{j=1}^nS_jT_j\xi.
\]
Thus, we have $\sum_{j=1}^nS_jT_j=1$.

Now suppose there exist $S_1,\ldots,S_n\in B(\HEu)$ such that~\eqref{eq:St1} holds.
We will show that $\delta_n'$ is bounded below, which will imply that $\delta_n'$ is injective and has closed range.
Given $\xi\in\HEu$, we have
\[
\|\xi\|\le\sum_{i=1}^n\|S_i\|\|T_i\xi\|\le\bigg(\sum_{i=1}^n\|S_i\|^2\bigg)^{1/2}\bigg(\sum_{i=1}^n\|T_i\xi\|^2\bigg)^{1/2}
=\bigg(\sum_{i=1}^n\|S_i\|^2\bigg)^{1/2}\|\delta_n'\xi\|.
\]
This completes the proof.
\end{proof}

\begin{prop}\label{prop:HarteTaylor}
Let $\HEu$ be a Hilbert space,
let $n\in\Nats$ and let $T=(T_1,\ldots,T_n)$ be an $n$-tuple of commuting operators in $B(\HEu)$.
Then
\[
\spec_{B(\HEu)}(T)\subseteq\Sp(T).
\]
\end{prop}
\begin{proof}
It will suffice to show that if $T$ is non-singular in the sense of Taylor, namely, if $(0,\ldots,0)\notin\Sp(T)$,
then $T$ generates all of $B(\HEu)$ as a right ideal and as a left ideal.
Nonsingularity in the sense of Taylor means that the entire Koszul complex is an exact sequence,
which entails that
the boundary map $\delta_0$ is surjective and the boundary map $\delta_n$ is injective and has closed range.
Now Lemmas~\ref{lem:delta0} and~\ref{lem:deltan} give the desired result.
\end{proof}

This section's main result follows. It is that for a tuple $T=(T_1,\ldots,T_n)$ of commuting operators in a tracial von Neumann algebra $\Mcal$, the support of the probability measure $\nu_T$ lies in the left Harte joint spectrum $\spec^\ell_\Mcal(T)$ and in Taylor's joint spectrum $\Sp(T)$.

\begin{prop}\label{prop:suppspec}
Suppose $\Mcal$ is a von Neumann algebra with normal, faithful, tracial state $\tau$, let $n\in\Nats$ and
suppose that $T=(T_1,\ldots,T_n)$ is an $n$-tuple of commuting elements of $\Mcal$.
Then
\begin{equation}\label{eq:nuspec}
\supp(\nu_T)\subseteq\spec^\ell_\Mcal(T).
\end{equation}
Choosing any normal, faithful representation of $\Mcal$ on a Hilbert space $\HEu$ and thereby realising $T$ 
as a tuple of bounded operators on $\HEu$,
we also have
\[
\supp(\nu_T)\subseteq\Sp(T).
\]
\end{prop}
\begin{proof} Suppose $(0,\ldots,0)\notin\spec^\ell_\Mcal(T).$ By definition of $\spec^\ell_\Mcal(T),$ there exist elements $X_1,\ldots,X_n\in\Mcal$ such that $X_1T_1+\cdots+X_nT_n=1.$ Let
$$M=\max(\|X_1\|,\ldots,\|X_n\|,\|T_1\|,\ldots,\|T_n\|,1)$$
and let $\eps>0$ be so small that $2n(nM)^{2n}\eps<1.$ We claim that
\begin{equation}\label{epsilon ball eq}
\bigwedge_{i=1}^n P(T_i,\overline{B_\eps(0)})=0.
\end{equation}
This will imply the inclusion~\eqref{eq:nuspec}.

Assume the contrary and let $Q$ be the projection on the left hand side in \eqref{epsilon ball eq}. Since $T_iQ=QT_iQ,$ it follows that
$$\sum_{i=1}^nQX_iQ\cdot T_iQ=Q=1_{Q\mathcal{M}Q}.$$
Denote $Y_i=QX_iQ$ and $S_i=T_iQ$.
By the basic property of Haagerup--Schultz projections, we have that $\nu_{S_i}^{(Q)}$ is supported in the ball $\overline{B_\eps(0)}$.

We now work in the algebra $Q\mathcal{M}Q$.
Writing
$$1=Y_1S_1+\cdots+Y_nS_n=Y_1\big(Y_1S_1+\cdots+Y_nS_n\big)S_1+\cdots+Y_n\big(Y_1S_1+\cdots+Y_nS_n\big)S_n$$
and so on, we obtain, for every $m\in\Nats$,
$$1=\sum_{i_1,\ldots,i_m\in\{1,\ldots,n\}}Y_{i_1}Y_{i_2}\cdots Y_{i_m}S_{i_m}\cdots S_{i_2}S_{i_1}.$$
Since the $S_i$ commute with each other
and since in each term of the above summation at least one of the $S_j$ must be repeated at least $m/n$ times,
by reordering we obtain, for $m=nk$
\[
1=\sum_{j=1}^nL_{j,k}S_j^k,
\]
where for each $j,$ $L_{j,k}$ is an operator of norm no greater than $(nM)^{2nk}.$ 

Recall the standard inequality
$$\|A+B\|_q^q\leq\|A\|_q^q+\|B\|_q^q,\quad A,B\in L_q(\mathcal{M},\tau),\quad 0<q<1.$$
Using this inequality for $q=\frac1k,$ we obtain
$$\|1\|_{\frac1k}^{\frac1k}\leq\sum_{j=1}^n\|L_{j,k}S_j^k\|_{\frac1k}^{\frac1k}\leq (nM)^{2n}\sum_{j=1}^n\|S_j^k\|_{\frac1k}^{\frac1k}=(nM)^{2n}\sum_{j=1}^n\||S_j^k|^{\frac1k}\|_1.$$

By Theorem~\ref{thm:HSsot}, $|S_j^k|^{\frac1k}$ converges strongly as $k\to\infty$ to some $N_j\ge0$ and, since the Brown measure of $S_j$ is supported in
$\overline{B_\eps(0)}$, we have
$\|N_j\|_{\infty}\leq\epsilon$.
By Lemma \ref{lem:sotqnequivs}, $|S_j^k|^{\frac1k}\to N_j$ in $L_1(\mathcal{M},\tau)$.
In particular, we have $\||S_j^k|^{\frac1k}\|_1\to\|N_j\|_1$ as $k\to\infty$.
Since $\|N_j\|_1\leq\epsilon$, it follows that $\||S_j^k|^{\frac1k}\|_1\leq2\epsilon$ for every sufficiently large $k$.
Hence, for large $k$, we have
$$1\leq 2n(nM)^{2n}\epsilon.$$
This contradicts our choice of $\epsilon.$ Hence, our assumption that $Q\neq0$ is false. This proves the inclusion~\eqref{eq:nuspec}.

Combining~\eqref{eq:nuspec} with  Proposition~\ref{prop:HarteTaylor} and the permanence property (Corollary \ref{cor:specperm}), we have
\[
\supp(\nu_T)\subseteq\spec^\ell_\Mcal(T)=\spec^\ell_{B(\HEu)}(T)\subseteq\Sp(T).
\]
\end{proof}

\begin{ques}
Consider a tuple $T=(T_1,\ldots,T_n)$ of commuting operators on some Hilbert space.
\begin{enumerate}[(a)]
\item\label{it:nonempty} Must the Harte spectrum $\spec_{B(\HEu)}(T)$ be non-empty?
\item\label{it:propincl} Do we ever have proper inclusion $\spec_{B(\HEu)}(T)\subsetneq\Sp(T)$ (compare Proposition~\ref{prop:HarteTaylor}).
\item If \lq\lq yes\rq\rq{}  to part~\eqref{it:propincl}, is there a holomorphic functional calculus satisfying nice properties for functions holomorphic in a neighborhood of $\spec_{B(\HEu)}(T)$?
\end{enumerate}
\end{ques}

\begin{ques}
Consider a tuple $T=(T_1,\ldots,T_n)$ of commuting operators in a finite von Neumann algebra $\Mcal$.
Must the left and right Harte spectra $\spec^\ell_\Mcal(T)$ and $\spec^r_\Mcal(T)$ agree?
\end{ques}

\section{Simultaneous upper triangularization}
\label{sec:ut}

Suppose $(q_t)_{0\le t\le 1}$ is an increasing net of projections in $\Mcal$, with $q_0=0$ and $q_1=1$
and let $\Dc=W^*(\{q_t\mid 0\le t\le 1\})$.
\begin{defi}\label{def:UT}
We say that $T\in\Mcal$ is {\em upper triangular} with respect to $(q_t)_{0\le t\le1}$ if each $q_t$ is invariant under $T$,
i.e., if $Tq_t=q_tTq_t$ for every $t\in[0,1]$.
\end{defi}

Let
$$\Uc=\{S\in\Mcal\mid S\text{ upper triangular with respect to }(q_t)_{0\le t\le 1}\}.$$

\begin{lemma}\label{lem:U}
The set $\Uc$ is a subalgebra of $\Mcal$ that is closed in the weak operator topology.
Moreover, the restriction of $\Ec_{\Mcal\cap\Dc'}$ to $\Uc$ is an algebra homomorphism from $\Uc$ into $\Mcal\cap\Dc'$.
\end{lemma}
\begin{proof}
Since $\Uc$ is the set of all $S\in\Mcal$ such that $(1-q_t)Sq_t=0$ for all $t\in[0,1]$, it is clearly a subspace that is closed in the weak operator topology. If $S_1,S_2\in\Uc$ then 
$$S_1S_2q_t=S_1q_tS_2q_t=q_tS_1q_tS_2q_t=q_tS_1S_2q_t,\quad t\in[0,1],$$
so $S_1S_2\in\Uc$.
Thus, $\Uc$ is also a subalgebra of $\Mcal$.

Clearly $\Ec_{\Mcal\cap\Dc'}$ is linear. We need only show that it is multiplicative. There is an increasing family $\Dc_1\subseteq\Dc_2\subseteq\cdots$ of finite dimensional, unital $*-$subalgebras of $\Dc$ whose union is dense (in strong operator topology) in $\Dc$
and such that
each $\Dc_n$ is the linear span of  a
set $\{q_{t(n,1)},\ldots,q_{t(n,k(n))}\}$ with
$0<t(n,1)<\cdots<t(n,k(n))=1$.
Thus, the relative commutants $\Mcal\cap\Dc_n'$ are decreasing in $n$ and their intersection is $\Mcal\cap\Dc'$.
Set $t(n,0)=0$.
Then we have
$$\Ec_{\Mcal\cap\Dc_n'}(A)=\sum_{j=1}^{k(n)}(q_{t(n,j)}-q_{t(n,j-1)})A(q_{t(n,j)}-q_{t(n,j-1)}),\quad (A\in\mathcal{M}).$$

Suppose $X,Y\in\Uc.$
For every $1\leq j\leq k(n),$ we have
\begin{multline*}
(q_{t(n,j)}-q_{t(n,j-1)})XY(q_{t(n,j)}-q_{t(n,j-1)}) \\
\begin{aligned}
&=q_{t(n,j)}(1-q_{t(n,j-1)})XYq_{t(n,j)}(1-q_{t(n,j-1)}) \\
&=q_{t(n,j)}(1-q_{t(n,j-1)})X(1-q_{t(n,j-1)})q_{t(n,j)}Yq_{t(n,j)}(1-q_{t(n,j-1)}) \\
&=(q_{t(n,j)}-q_{t(n,j-1)})X(q_{t(n,j)}-q_{t(n,j-1)})Y(q_{t(n,j)}-q_{t(n,j-1)}).
\end{aligned}
\end{multline*}
Summing over $1\leq j\leq k(n)$,
we obtain
$$\Ec_{\Mcal\cap\Dc_n'}(XY)=\Ec_{\Mcal\cap\Dc_n'}(X)\Ec_{\Mcal\cap\Dc_n'}(Y),\quad (X,Y\in\Uc).$$
Since $\bigcap_{n\ge1}(\Mcal\cap\Dc_n')=\Mcal\cap\Dc'$, we have
$$\Ec_{\Mcal\cap\Dc_n'}(A)\to \Ec_{\Mcal\cap\Dc'}(A),\quad n\to\infty,\quad (A\in\mathcal{M}),$$
in strong operator topology.
Since multiplication on bounded sets is continuous in strong operator topology, the assertion follows.
\end{proof}

The following is an easy consequence of Lemma~22 of~\cite{DSZ15}.
\begin{lemma}\label{lem:22}
If $S\in\Mcal$ and if $S$ is upper triangular with respect to the net $(q_t)_{0\le t\le 1}$, then
$$\nu_S=\nu_{\Ec_{\Mcal\cap\Dc'}(S)}.$$
\end{lemma}
\begin{proof} By Lemma~22 of~\cite{DSZ15}, we have
$$\Delta(S-\lambda)=\Delta(\Ec_{\Mcal\cap\Dc'}(S)-\lambda),\quad \lambda\in\Cpx,$$
where $\Delta$ is the Fuglede--Kadison determinant. Thus, the Brown measures agree.
\end{proof}

We now turn to the setting of Theorem~\ref{thm:simultut}.
Let $\rho$ be a continuous spectral ordering as in~\eqref{eq:rho} and let $\Dc$ be the associated
abelian von Neumann algebra
as in~\eqref{eq:D}.
By adjusting $\rho$, if necessary, we may, without affecting the algebra $\Dc$, assume $\nu_T(\{\rho(0)\})=0$ and for convenience we do so.
As described in the preamble to Lemma~\ref{lem:spfill}, from the measure $\nu_T$ we obtain
a probability measure $\sigma$ on $[0,1]$ satisfying, for all $t\in[0,1]$,
\[
\sigma([0,t])=\nu_T(\rho([0,t]))=\tau(P(T:\rho([0,t]))).
\]
Now $\Dc$ can be identified with $L^\infty([0,1],\sigma)$ for the Borel probability measure $\sigma$ on $[0,1]$, so that the restriction
of $\tau$ to $\Dc$ corresponds to integration with respect to $\sigma$
and so that each $q_t$ is identified with the indicator function of the interval $[0,t]$.
Indeed, $\Dc$ is the von Neumann algebra generated by the set of projections $\{q_t\mid 0\le t\le1\}$, which form an increasing chain, and
$L^\infty([0,1],\sigma)$ is the von Neumann algebra generated by the set of projections $\{1_{[0,t]}\mid0\le t\le 1\}$.
The map $q_t\mapsto 1_{[0,t]}$
is trace preserving and order preserving and extends to a trace preserving $*$-isomorphism of von Neumann algebras.

By Lemma~\ref{lem:spfill}, we have
\begin{equation}\label{eq:nuTrhosig}
\nu_T=\rho_*\sigma.
\end{equation}

We write $\Mcal\cap\Dc'$ and the restriction of $\tau$ to this algebra as direct integrals with respect to $\Dc$:
\[
\Mcal\cap\Dc'=\int_{[0,1]}^\oplus\Nc(t)\,d\sigma(t),\qquad\tau=\int_{[0,1]}^\oplus\tau_t\,d\sigma(t),
\]
for normal, faithful, tracial states $\tau_t$ on von Neumann algebras $\Nc(t)$.
Then an element $A$ of $\Mcal\cap\Dc'$ is written
\begin{equation}\label{eq:XintXt}
A=\int_{[0,1]}^\oplus A(t)\,d\sigma(t),
\end{equation}
with $A(t)\in\Nc(t)$, and we have $A\in\Dc$ if and only if $A(t)\in\Cpx1$ for $\sigma$-almost every $t\in[0,1]$.
Furthermore, for $A$ as in~\eqref{eq:XintXt}, we have
\begin{equation}\label{eq:EDX}
\Ec_{\Dc}(A)=\int_{[0,1]}^\oplus\tau_t(A(t))\,d\sigma(t).
\end{equation}

The following result is similar to the proof of Theorem~1.2 of~\cite{DSZ}, adapted to the setting here.
\begin{lemma}\label{lem:di}
Writing
$$X_j:=\Ec_{\Mcal\cap\Dc'}(T_j)=\int_{[0,1]}^\oplus X_j(t)\,d\sigma(t),$$
we have $\nu_{X_j(s)}=\delta_{\rho_j(s)}$ for $\sigma$-almost every $s\in[0,1].$
\end{lemma}
\begin{proof}
By Theorem~\ref{thm:5.6}, the Brown measure of $X_j$ equals the integral of Brown measures of the $X_j(t)$, namely, for every Borel set
$B\subseteq\Cpx$, we have
\[
\nu_{X_j}(B)=\int_{[0,1]}\nu_{X_j(t)}(B)\,d\sigma(t).
\]
Given $0\le t(1)<t(2)\le1$, the projection $q_{t(2)}-q_{t(1)}$ is identified with the characteristic function
$1_{(t_1,t_2]}$ and we have
\begin{equation}\label{eq:qt2t1X}
(q_{t(2)}-q_{t(1)})X_j(q_{(t(2)}-q_{t(1)})=\int_{(t(1),t(2)]}^\oplus X_j(t)\,d\sigma(t).
\end{equation}
If $\tau(q_{t(1)})<\tau(q_{t(2)})$, then
by Theorem~\ref{thm:5.6},
the Brown measure of the element~\eqref{eq:qt2t1X} 
computed with respect to the renormalisation of the restriction of $\tau$, is given (in the notation mentioned after that theorem) by
\[
\nu_{(q_{t(2)}-q_{t(1)})X_j(q_{t(2)}-q_{t(1)})}^{(q_{t(2)}-q_{t(1)})}=\frac1{\sigma((t(1),t(2)])}\int_{(t(1),t(2)]}\nu_{X_j(t)}\,d\sigma(t).
\]
The operator $(q_{t(2)}-q_{t(1)})X_j(q_{t(2)}-q_{t(1)})$ is the value of the conditional expectation $\Ec_{((q_{t(2)}-q_{t(1)})\Dc)'}$,
from $(q_{t(2)}-q_{t(1)})\Mcal(q_{t(2)}-q_{t(1)})$ onto $(q_{t(2)}-q_{t(1)})(\Dc'\cap\Mcal)$, applied to the operator $(q_{t(2)}-q_{t(1)})T_j(q_{t(2)}-q_{t(1)})$.
Thus, by applying Lemma~\ref{lem:22} in this setting, we have
\begin{equation}\label{eq:nuqXjq}
\nu_{(q_{t(2)}-q_{t(1)})X_j(q_{t(2)}-q_{t(1)})}^{(q_{t(2)}-q_{t(1)})}=\nu_{(q_{t(2)}-q_{t(1)})T_j(q_{t(2)}-q_{t(1)})}^{(q_{t(2)}-q_{t(1)})}.
\end{equation}
By Corollary~\ref{cor:middle}, the joint Brown measure
\begin{equation}\label{eq:nuqTq}
\nu_{(q_{t(2)}-q_{t(1)})T(q_{t(2)}-q_{t(1)})}^{(q_{t(2)}-q_{t(1)})}
\end{equation}
is concentrated in $\rho([0,t(2)])\setminus\rho([0,t(1)])$, which is contained in $\rho((t(1),t(2)])$.
Since the Brown measure~\eqref{eq:nuqXjq} is the $j$-th marginal distribution
of the measure~\eqref{eq:nuqTq},
the former must be concentrated in the closed convex hull, $\overline{\operatorname{conv}}(\rho_j((t(1),t(2)]))$ of $\rho_j((t(1),t(2)])$.
Therefore, for $\sigma$-almost every $t\in(t(1),t(2)]$,
the Brown measure of $X_j(t)$ is concentrated in $\overline{\operatorname{conv}}(\rho_j((t(1),t(2)]))$.
The same statement is tautologically true when $\tau(q_{t(1)})=\tau(q_{t(2)})$, for then $\sigma((t(1),t(2)])=0$.

Thus, we find a $\sigma$-null set $N\subseteq[0,1]$
so that for all $t\in N^c$ and all rational $t(1)$ and $t(2)$ with $0\le t(1)<t\le t(2)\le1$, the
Brown measure of $X_j(t)$ is concentrated in $\overline{\operatorname{conv}}(\rho_j((t(1),t(2)]))$.
By continuity of $\rho_j$, for all $t\in N^c$ we have $\nu_{X_j(t)}=\delta_{\rho_j(t)}$.
\end{proof}

\begin{lemma}\label{lem:fX}
Let $X_j$ be as in Lemma~\ref{lem:di}.
Let $f$ be a polynomial in $n$ commuting variables and let $Y=f(X_1,\ldots,X_n)$.
Writing $Y$ as a direct integral over $\Dc$, we have
\begin{equation}\label{eq:Ydi}
Y=\int_{[0,1]}^\oplus Y(t)\,d\sigma(t),
\end{equation}
where $Y(t)=f(X_1(t),\ldots,X_n(t))$.
Then for $\sigma$-almost every $t\in[0,1]$, the Brown measure $\nu_{Y(t)}$ of $Y(t)$ is the Dirac mass $\delta_{f(\rho_1(t),\ldots,\rho_n(t))}$.
\end{lemma}
\begin{proof}
Recall that the Brown measure of an element $A\in\Mcal$ is the Dirac mass at $z\in\Cpx$ if and only if $A-z$ is s.o.t.-quasinilpotent.
By Lemma~\ref{lem:di}, there is a $\sigma$-null set $N\subseteq[0,1]$ such that
for all $t\in N^c$ and every $j\in\{1,\ldots,n\}$, $X_j(t)=\rho_j(t)+X_j\oup(t)$, where $X_j\oup(t)$ is an
s.o.t.-quasinilpotent operator.
By applying Proposition~\ref{prop:polyqn} with translation, it follows that $Y(t)=f(\rho_1(t),\ldots,\rho_n(t))+Y\oup(t)$,
where $Y\oup(t)$ is s.o.t.-quasinilpotent.
Thus, for every $t\in N^c$, $\nu_{Y(t)}=\delta_{f(\rho_1(t),\ldots,\rho_n(t))}$, as required.
\end{proof}

\begin{proof}[Proof of Theorem \ref{thm:simultut}]
We have $S=f(T_1,\ldots,T_n)$ for a polynomial $f$ in $n$ commuting variables.
By Lemma~\ref{lem:U}, $S$ is upper triangular with respect to $(q_t)_{0\le t\le 1}$.
Let $Y=\Ec_{\Mcal\cap\Dc'}(S)$ and $N=\Ec_\Dc(S)$.
By Lemma~\ref{lem:22}, $\nu_S=\nu_Y$ and $\nu_{S-N}=\nu_{Y-N}$.

Letting $X_j$ be as in Lemma~\ref{lem:di}, by Lemma~\ref{lem:U}, we have $Y=f(X_1,\ldots,X_n)$.
We write $Y$ as a direct integral as in~\eqref{eq:Ydi} in Lemma~\ref{lem:fX}.
By Theorem~\ref{thm:5.6} and Lemma~\ref{lem:fX}, we have
\[
\nu_Y=\int_{[0,1]}\delta_{f(\rho(t))}\,d\sigma(t).
\]
Thus,
\[
\nu_Y=(f\circ\rho)_*\sigma=f_*(\rho_*\sigma)=f_*\nu_T,
\]
where for the last equality we used~\eqref{eq:nuTrhosig}.

From the direct integral decomposition~\eqref{eq:Ydi} of $Y$, applying $\Ec_\Dc$ and~\eqref{eq:EDX}
we get the direct integral decomposition
\begin{equation}\label{eq:Ndi}
N=\int^\oplus_{[0,1]}\tau_t(Y(t))\,d\sigma(t).
\end{equation}
Brown's version of Lidskii's theorem~\cite{B86}
states that for any operator $T$ in a finite von Neumann algebra with trace $\tau$, we have $\tau(T)=\int_\Cpx z\,d\nu_T(z)$,
where the Brown measure $\nu_T$ is taken with respect to $\tau$.
This, combined with
Lemma~\ref{lem:fX} implies
\begin{equation}\label{eq:tauYt}
\tau_t(Y(t))=f(\rho(t))
\end{equation}
for $\sigma$-almost all $t\in[0,1]$.
From~\eqref{eq:Ndi} and Theorem~\ref{thm:5.6},
since the Brown measure of a scalar operator is the Dirac mass at that scalar, we have
\[
\nu_N=\int\delta_{f(\rho(t))}\,d\sigma(t)=\nu_Y.
\]
This implies $\nu_N=\nu_S$.

Similarly, 
\[
Y-N=\int^\oplus_{[0,1]}\big(Y(t)-\tau_t(Y(t))\big)\,d\sigma(t).
\]
By Lemma~\ref{lem:fX}, $\nu_{Y(t)-\tau_t(Y(t))}=\delta_0$ for $\sigma$-almost all $t\in[0,1]$.
Thus, using again Theorem~\ref{thm:5.6}, we have
\[
\nu_{Y-N}=\int_{[0,1]}\delta_0\,d\sigma(t)=\delta_0.
\]
This implies $\nu_{S-N}=\delta_0$, which means that $S-N$ is s.o.t.-quasinilpotent.
\end{proof}

The following is essentially a corollary of the above proof of Theorem \ref{thm:simultut}.
\begin{prop}\label{prop:EDhomom}
For every polynomial $f$ in $n$ commuting variables, we have
\begin{equation}\label{eq:EDT}
\Ec_\Dc\big(f(T_1,\ldots,T_n)\big)=\int^\oplus_{[0,1]}f(\rho(t))\,d\sigma(t).
\end{equation}
Thus,
the restriction of $\Ec_\Dc$ to the unital Banach algebra, $\Afr$, that is generated by $\{T_1,\ldots,T_n\}$
is an algebra homomorphism.
\end{prop}
\begin{proof}
The equality~\eqref{eq:EDT} follows from~\eqref{eq:Ndi}
and~\eqref{eq:tauYt}.
Now it follows immediately that the restriction of $\Ec_\Dc$ to the algebra generated by $\{T_1,\ldots,T_n\}$
is an algebra homomorphism.
By boundedness of $\Ec_\Dc$, it follows that its restriction to $\Afr$ is also an algebra homomorphism.
\end{proof}

\section{Taylor's holomorphic functional calculus and direct integrals}
\label{sec:holofc}

In this section we show
that the Taylor joint spectrum~\cite{T70} and the Taylor holomorphic function calculus~\cite{T70b}
of commuting operators on a Hilbert space
thread through direct integrals in the natural way.
For this, we will use F.-H.\ Vasilescu's formula~\cite{V78} for the Taylor holomorphic functional calculus,
that is an analogue of Martinelli's formula.
This result was also described in~\cite{V79} and we will use the notation employed there.
See also the related work of V.\ M\"uller~\cite{M02}.

Consider an $n$-tuple $T=(T_1,\ldots,T_n)$ of commuting operators on a Hilbert space $\HEu$.
Taylor's joint spectrum of $T$ is denoted by $\Sp(T)$,
and the Taylor functional calculus assigns, to every function $f$ that is holomorphic
on a neighborhood of $\Sp(T)$, an operator $f(T)$,
which belongs to the double commutant $\Afr''$ of the unital Banach algebra $\Afr$ generated by $\{T_1,\ldots,T_n\}$.
This map $f\mapsto f(T)$ is an algebra homomorphism
from the set of germs of holomorphic functions on neighborhoods of $\Sp(T)$ into $\Afr''$,
sending the $j$-th coordinate function to $T_j$, for every $j$.
(Taylor's results are more general, namely for commuting operators on a Banach space.)

We now describe Vasilescu's formula for the Taylor functional calculus.
We write $s=(s_1,\ldots,s_n)$ for indeterminates $s_1,\ldots,s_n$ and we let
$\Lambda[s]$ be the algebra of all exterior forms
in $s_1,\ldots,s_n$, equipped with the wedge product.
We give it the inner product so that 
\[
\{\Omega\}\cup\{s_{i_1}\wedge\cdots\wedge s_{i_p}\mid 1\le i_1<i_2<\cdots<i_p\le n\}
\]
is an orthonormal basis, where $\Omega$ is the unit $0$-form such that $s_j\wedge\Omega=s_j$.
Let $L(s_j)$ denote the linear map on $\Lambda[s]$ given by $\xi\mapsto s_j\wedge\xi$.
This is a partial isometry.

Let $\delta_T$ be the operator on the Hilbert space $\Lambda[s]\otimes\HEu$ given by 
\[
\delta_T=\sum_{j=1}^n L(s_j)\otimes T_j.
\]
Note that $\delta_T^2=0$.
Set 
$$\alpha_T=\delta_T+(\delta_T)^*.$$
In~\cite{V77}, Vasilescu showed that $\Sp(T)$ consists of all $w=(w_1,\ldots,w_n)$ such that $\alpha_{T-w}$
is not invertible in $B(\Lambda[s]\otimes\HEu)$, where $T-w=(T_1-w_1,\ldots,T_n-w_n).$

For an open subset $U\subseteq\Cpx^n$ and a Hilbert space $\HEu$, let $C^\infty(U,\HEu)$ denote the set of all $\HEu$-valued,
infinitely differentiable functions (in the variables $z_1,\overline{z_1},\ldots,z_n,\overline{z_n}$, say).

If $U\cap\Sp(T)=\emptyset,$ then the function
\[
z\mapsto(\alpha_{T-z})^{-1},\quad (z\in U),
\]
is infinitely differentiable and is, moreover,
analytic in the variables $z_1,\overline{z_1},\ldots,\allowbreak z_n,\overline{z_n}$,
in the sense of being given locally by power series in these $2n$ variables with positive radii of convergence.
(Note that we do not mean that the map is holomorphic in $z_1,\ldots,z_n$.)
Indeed, for $z,w\in\Cpx^n$, we have 
$$\alpha_{z-w}=\Big(\sum_{j=1}^n(z_j-w_j)L(s_j)+(\overline{z_j}-\overline{w_j})(L(s_j))^*\Big)\otimes I_\HEu.$$
Since each $L(s_j)$ is a partial isometry whose square is zero,
the inequality $\|\alpha_{z-w}\|\le \|z-w\|_1$ holds.
By linearity, when $w\in\Cpx^n\setminus\Sp(T)$, we have
$$\alpha_{T-z}=\alpha_{T-w}+\alpha_{w-z}=(1+\alpha_{w-z}(\alpha_{T-w})^{-1})\cdot \alpha_{T-w}.$$
Thus, for $\|z-w\|_1<\|(\alpha_{T-w})^{-1}\|^{-1}$, we have
\begin{equation}\label{eq:alphaexp}
\alpha_{T-z}^{-1}=\alpha_{T-w}^{-1}\big(1-\alpha_{z-w}(\alpha_{T-w})^{-1}\big)^{-1}=\sum_{k=0}^\infty(\alpha_{T-w})^{-1}\big(\alpha_{z-w}(\alpha_{T-w})^{-1}\big)^k.
\end{equation}
This yields a power series expansion in the variables
\begin{equation}\label{eq:zwvars}
z_1-w_1,\,\overline{z_1}-\overline{w_1},\cdots,z_n-w_n,\,\overline{z_n}-\overline{w_n}
\end{equation}
which converges whenever $\|z-w\|_1<\|(\alpha_{T-w})^{-1}\|^{-1}$ and whose coefficients are noncommutative polynomials in the
$L(s_j)\otimes I_\HEu$, $L(s_j)^*\otimes I_\HEu$ and $(\alpha_{T-w})^{-1}$.

Consider also the symbols $\dsp\zbar=(\dsp\overline{z_1},\ldots,\dsp\overline{z_n})$ and the 
exterior algebra $\Lambda[s,\dsp\zbar]$ in $2n$ indeterminates.
As vector spaces, we have $\Lambda[s,\dsp\zbar]=\Lambda[\dsp\zbar]\otimes\Lambda[s]$ which we could also write as
$\Lambda[\dsp\zbar]\wedge\Lambda[s]$.
We will use this decomposition when convenient, also write $L(s_j)$ and $L(\dsp\overline{z_j})$
for the left wedging operators in $B(\Lambda[s,\dsp\zbar])$.
Define the operator
\[
\beta_T:C^\infty(U,\Lambda[s,\dsp\zbar]\otimes\HEu)\to C^\infty(U,\Lambda[s,\dsp\zbar]\otimes\HEu)
\]
by 
\[
(\beta_Tg)(z)=\big(I_{\Lambda[\dsp\zbar]}\otimes(\alpha_{z-T})^{-1}\big)g(z).
\]
We consider also the operator
\[
\overline{\partial}=\big(L(\dsp\overline{z_1})\otimes I_\HEu\big)\frac{\partial}{\partialsp\overline{z_1}}+\cdots
+\big(L(\dsp\overline{z_n})\otimes I_\HEu\big)\frac{\partial}{\partialsp\overline{z_n}}
\]
on $C^\infty(U,\Lambda[s,\dsp\zbar]\otimes\HEu)$, given by 
\[
(\overline{\partial}g)(z)=\sum_{j=1}^n\big(L(\dsp\overline{z_j})\otimes I_\HEu\big)\frac{\partial g}{\partialsp\overline{z_j}}(z).
\]
We consider the linear operator
\[
M_T=\beta_T\circ\big(\overline{\partial}\circ\beta_T\big)^{n-1}\circ L(s_1)\circ L(s_2)\circ \cdots\circ L(s_n)
\]
from $C^\infty(U,\Lambda[s,\dsp\zbar]\otimes\HEu)$ to itself.
This is the same as the operator $M_T$
used by Vasilescu and described in Equation~(2.1) of~\cite{V79}
(with some minor notational differences).
We now state Vasilescu's Theorem~2.1 of~\cite{V79}, which expresses Taylor's holomorphic functional calculus
for commuting Hilbert space operators using an analogue of Martinelli's formula.
We identify $\HEu$ with $\Omega\otimes\HEu\subseteq\Lambda[s,\dsp\zbar]\otimes\HEu$, where $\Omega$ is the
unit $0$-form in $\Lambda[s,\dsp\zbar]$,
and we let $\Hol(W)$
denote the set of all complex analytic functions on an open subset $W$ of $\Cpx^n$,
endowed with the topology of uniform convergence on compact subsets.

\begin{thm}(\cite{V79}, \cite{V78}).
\label{thm:Vasilescu}
Suppose $T=(T_1,\ldots,T_n)$ is a tuple of commuting operators on a Hilbert space $\HEu$.
Let $W$ be an open subset of $\Cpx^n$ containing the Taylor joint spectrum $\Sp(T)$.
Then the formula
\begin{equation}\label{eq:fTint}
f(T)x=\frac1{(2\pi i)^n}\int_{\partial\Delta}f(z)(M_Tx)(z)\wedge dz_1\wedge\cdots\wedge dz_n,\qquad(x\in\HEu),
\end{equation}
where $\Delta$ is any bounded, open subset of $W$ with $\Sp(T)\subseteq\Delta$ and whose boundary $\partial\Delta$
is a finite union of smooth surfaces contained in $W$,
defines a continuous, unital homomorphism $f\mapsto f(T)$ from the unital algebra $\Hol(W)$ into $B(\HEu)$,
sending the coordinate function $z_j$ to $T_j$ for each $j$.
Moreover, the integral~\eqref{eq:fTint} does not depend on the choice of $\Delta$ satisfying the above conditions.
\end{thm}

In \cite{V78}, Vasilescu does not claim that his functional calculus is the same as Taylor's.
For application in Section~\ref{sec:hfcnu}, we don't require it to be so.
However, that the two calculi coincide follows from results of M.\ Putinar~\cite{P83} (or of V.\ M\"uller~\cite{M02}).

Suppose now that $\HEu$ is a Hilbert space that is a direct integral
\[
\HEu=\int_Z^\oplus\HEu(\zeta)\,d\omega(\zeta)
\]
of Hilbert spaces $\HEu(\zeta)$ for a Borel probability measure $\omega$ on a Polish space $Z$.
Let $\Dc\cong L^\infty(Z,\omega)$ denote the diagonal operators (with respect to this direct integral decomposition) in $B(\HEu)$
and let $\Dc'\cap B(\HEu)$ denote the commutant of $\Dc$ in $B(\HEu)$.
Thus, for $A\in\Dc'\cap B(\HEu)$ we write
\[
A=\int_Z^\oplus A(\zeta)\,d\omega(\zeta)
\]
for $A(\zeta)\in B(\HEu(\zeta))$.

Suppose $T=(T_1,\ldots,T_n)$ is a tuple of commuting operators in $\Dc'$
and consider direct integral representations
\begin{equation}\label{eq:Tjdi}
T_j=\int_Z^\oplus T_j(\zeta)\,d\omega(\zeta)
\end{equation}
for each $j$.
After redefining these
on an $\omega$-null set, if necessary, we assume $T_j(\zeta)$ and $T_i(\zeta)$ commute for all $\zeta\in Z$ and all $i$ and $j$.
Let us write $T(\zeta)=(T_1(\zeta),\ldots,T_n(\zeta))$.

\begin{lemma}\label{lem:SpTzeta}
For every $z\in\Cpx^n\setminus\Sp(T)$, there exists a neighborhood $V$ of $z$ in $\Cpx^n$ and an $\omega$-null set $N\subseteq Z$
such that for every $\zeta\in Z\setminus N$ we have $\Sp(T(\zeta))\cap V=\emptyset$.
\end{lemma}
\begin{proof}
We may without loss of generality assume $z=(0,\ldots,0)$.

We have the direct integral decomposition
\[
\Lambda[s]\otimes\HEu=\int^\oplus_Z(\Lambda[s]\otimes\HEu(\zeta))\,d\omega(\zeta)
\]
with respect to which, for every $z\in U$, we have that
\begin{equation}\label{eq:alphaTdi}
\alpha_{T-z}=\int^\oplus_Z\alpha_{T(\zeta)-z}\,d\omega(\zeta).
\end{equation}
Since $(0,\ldots,0)\notin\Sp(T)$, the operator $\alpha_T$ is invertible.
Thus, there is an $\omega$-null set $N$ such that for all $\zeta\in Z\setminus N$, $\alpha_{T(\zeta)}$ is invertible and
$\|(\alpha_{T(\zeta)})^{-1}\|\le\|(\alpha_T)^{-1}\|.$
Suppose that $\|z\|_1<\|(\alpha_T)^{-1}\|$.
By \eqref{eq:alphaexp}, we have
$$\alpha_{T(\zeta)-z}^{-1}=\sum_{n=0}^\infty(\alpha_{T(\zeta)})^{-1}\big(\alpha_z(\alpha_{T(\zeta)})^{-1}\big)^n,$$
where the series is convergent. Thus, letting $V=\{z\in\Cpx^n\mid \|z\|_1<\|(\alpha_T)^{-1}\|^{-1}\}$,
we have $\Sp(T(\zeta))\cap V=\emptyset$ for every $\zeta\in Z\setminus N$, as required.
\end{proof}

\begin{prop}\label{prop:SpTzeta}
Let $T=(T_1,\ldots,T_n)$ be commuting operators in $\Dc'\cap B(\HEu)$, each with direct integral decomposition as in~\eqref{eq:Tjdi},
and assume without loss of generality that $T(\zeta)=(T_1(\zeta),\ldots,T_n(\zeta))$
is a commuting tuple for all $\zeta\in Z$.
Then there is an $\omega$-null set $N\subseteq Z$ such that $\Sp(T(\zeta))\subseteq\Sp(T)$ for every $\zeta\in Z\setminus N$.
\end{prop}
\begin{proof}
Applying Lemma~\ref{lem:SpTzeta}, we find an open cover $\Vc=(V_i)_{i\in I}$ of $\Cpx^n\setminus\Sp(T)$ such that for every $i\in I$, there is
an $\omega$-null set $N_i\subseteq Z$ such that for every $\zeta\in Z\setminus N_i$, $\Sp(T(\zeta))\cap V_i=\emptyset$.
There is a countable subcover $(V_{i(j)})_{j=1}^\infty$ of $\Vc$.
Letting $N=\bigcup_{j=1}^\infty N_{i(j)}$, we have $\Sp(T(\zeta))\subseteq\Sp(T)$ for every $\zeta\in Z\setminus N$.
\end{proof}

Here is the main result of this section:
\begin{prop}\label{prop:hfcdi}
Let $T=(T_1,\ldots,T_n)$ be commuting operators in $\Dc'\cap B(\HEu)$ and let $f$ be a function that is holomorphic on an open neighborhood $W$ of the Taylor joint spectrum $\Sp(T)$.
Then applying the Taylor functional calculus, we have the direct integral representation
\begin{equation}\label{eq:fTdi}
f(T)=\int_Z^\oplus f(T(\zeta))\,d\omega(\zeta).
\end{equation}
\end{prop}
\begin{proof} 
Write each $T_j$ as a direct integral as in~\eqref{eq:Tjdi} and assume without loss of generality that $T(\zeta)=(T_1(\zeta),\ldots,T_n(\zeta))$
is a commuting tuple with $\Sp(T(\zeta))\subseteq\Sp(T)$, for all $\zeta\in Z$.
 
As Hilbert spaces, we have $\Lambda[s,\dsp\zbar]=\Lambda[\dsp\zbar]\otimes\Lambda[s]$.
Thus, we have the direct integral decomposition
\begin{equation}\label{eq:LsdzHdi}
\Lambda[s,\dsp\zbar]\otimes\HEu=\int^\oplus_Z(\Lambda[s,\dsp\zbar]\otimes\HEu(\zeta))\,d\omega(\zeta)
\end{equation}
whose diagonal operators form the algebra
\[
I_{\Lambda[s,\dsp\zbar]}\otimes\Dc\subseteq B(\Lambda[s,\dsp\zbar])\otimes B(\HEu).
\]
Moreover, writing $\alpha_{T-w}$ and $\alpha_{T(\zeta)-w}$
for $I_{\Lambda[\dsp\zbar]}\otimes\alpha_{T-w}$ and $I_{\Lambda[\dsp\zbar]}\otimes\alpha_{T(\zeta)-w}$, respectively in
$B(\Lambda[\dsp\zbar])\otimes B(\Lambda[s]\otimes\HEu)$ and $B(\Lambda[\dsp\zbar])\otimes B(\Lambda[s]\otimes\HEu(\zeta))$,
the direct integral decomposition~\eqref{eq:alphaTdi} applies also for the Hilbert space decomposition~\eqref{eq:LsdzHdi}.
We also write $L(s_j)$ for the corresponding operator in $B(\Lambda[s,\dsp\zbar])$, identified with
$I_{\Lambda[\dsp\zbar]}\otimes L(s_j)\in B(\Lambda[\dsp\zbar])\otimes B(\Lambda[s])$,
and we let $L(\dsp\overline{z_j})\in B(\Lambda[s,\dsp\zbar])$ denote the operator $\xi\mapsto\dsp\overline{z_j}\wedge\xi$.

Fixing $w\in W\setminus\Sp(T)$ and using the Leibniz rule and the power series expansion obtained from~\eqref{eq:alphaexp},
we find, for all $z$ close enough to $w$,
\[
(\overline{\partial}\circ\beta_Tg)(z)=\sum_{k=1}^nL(\dsp\overline{z_k})\left(F_k(z)\frac{\partial g}{\partialsp\overline{z}_k}(z)+G_k(z)g(z)\right),
\]
where $F_k(z)$ and $G_k(z)$ are given by
power series expansions in the variables~\eqref{eq:zwvars}
with positive radius of convergence and coefficients that are noncommutative polynomials in the $L(s_j)\otimes I_\HEu$,
$L(s_j)^*\otimes I_\HEu$ and $(\alpha_{T-w})^{-1}$.
Thus, by iterating, we find that, for all $x\in\HEu$,
\[ 
(M_Tx)(z)=\sum_{k=1}^n
L(\dsp\overline{z_1})\cdots\widehat{L(\dsp\overline{z_k})}\cdots L(\dsp\overline{z_n})H^{(n)}_k(z)x,
\] 
where $\widehat{L(\dsp\overline{z_k})}$ means that $L(\dsp\overline{z_k})$ is omitted and where each $H^{(n)}_k(z)$ is given by a
power series expansion in the variables~\eqref{eq:zwvars}
with positive radius of convergence and coefficients that are noncommutative polynomials
in the $L(s_j)\otimes I_\HEu$, $L(s_j)^*\otimes I_\HEu$ and $(\alpha_{T-w})^{-1}$.
Thus, the function
\begin{multline*}
\Cpx^n\setminus\Sp(T)\ni z\mapsto G_T(z):=\sum_{k=1}^n
L(\dsp\overline{z_1})\cdots\widehat{L(\dsp\overline{z_k})}\cdots L(\dsp\overline{z_n})H^{(n)}_k(z) \\
\in B(\Lambda[s,\dsp\zbar]\otimes\HEu)\cap(I_{\Lambda[s,\dsp\zbar]}\otimes\Dc')
\end{multline*}
is real analytic in the variables $z_1,\overline{z_1},\ldots,z_n,\overline{z_n}$
(note that the variables $\overline{z_1},\ldots,\overline{z_n}$ are to be distinguished from the differentials $\dsp\overline{z_1},\ldots,\dsp\overline{z_n}$)
and we have
\[
(M_Tx)(z)=G_T(z)x
\]
for all $x\in\HEu$.

As noted above, for each fixed $w$, we may write $(\alpha_{T-w})^{-1}$ as a direct integral
\begin{equation}\label{eq:alT-wdi}
(\alpha_{T-w})^{-1}=\int^\oplus_Z(\alpha_{T(\zeta)-w})^{-1}\,d\omega(\zeta)
\end{equation}
with respect to the decomposition~\eqref{eq:LsdzHdi}.
We take a countable open cover of $\Cpx^n\setminus\Sp(T)$ consisting of sets on each of which
$G_T$ has a single power series expansion.
Now in each open set of this cover,
using~\eqref{eq:alT-wdi} in the coefficients of the series expansion for each $H^{(n)}_k(z)$, above, we find an $\omega$-null set $N\subseteq Z$
such that
\[
G_T(z)=\int^\oplus_ZG_{T(\zeta)}(z)\,d\omega(\zeta),
\]
and such that, for every $\zeta\in Z\setminus N$, $G_{T(\zeta)}(z)\in B(\Lambda[s,\dsp\zbar]\otimes\HEu(\zeta))$
is given by the same power series expansion as is $G_T(z)$, but with $(\alpha_{T(\zeta)-w})^{-1}$
replacing each $(\alpha_{T-w})^{-1}$.
The same analysis as above shows that, for almost every $\zeta\in Z$ and every $z\in\Cpx^n\setminus\Sp(T)$, we have
\[
(M_{T(\zeta)}y)(z)=G_{T(\zeta)}(z)y,\qquad(y\in\HEu(\zeta)).
\]
This implies that, for every $x\in\HEu$, taking direct integral decomposition
\[
x=\int^\oplus_Zx(\zeta)\,d\omega(\zeta)
\]
we have, for all $z\in\Cpx^n\setminus\Sp(T)$,
\[
(M_Tx)(z)=G_T(z)x=\int^\oplus_ZG_{T(\zeta)}(z)x(\zeta)\,d\omega(\zeta)=\int^\oplus_Z(M_{T(\zeta)}x(\zeta))(z)\,d\omega(\zeta).
\]
We now use the power series expansions in sets of the open cover considered above and think of moduli of continuity.
So doing, we see that, for every compact subset $K$ of $\Cpx^n\setminus\Sp(T)$,
there is a modulus of continuity $\rho$ for the function $K\ni z\mapsto G_T(z)$,
so that the same modulus of continuity holds for the function $K\ni z\mapsto G_{T(\zeta)}(z)$, for every $\zeta\in Z\setminus N$.
Using this and standard approximations
and performing the real $(2n-1)$-dimensional Riemann integration in Vasilescu's formula~\eqref{eq:fTint}
found in Theorem~\ref{thm:Vasilescu}, we obtain
\[
f(T)x=\int^\oplus_Z f(T(\zeta))x(\zeta)\,d\zeta.
\]
This proves the desired formula~\eqref{eq:fTdi}.
\end{proof}

\section{Holomorphic functional calculus and joint Brown measures}
\label{sec:hfcnu}

In this section, we assume, as described before Lemma~\ref{lem:di}, that
$T=(T_1,\ldots,T_n)$ is a tuple of commuting elements of $\Mcal$, $\rho:[0,1]\to\Cpx^n$ is a continuous spectral ordering for $T$,
$q_t=P(T:\rho([0,t]))$ with $q_0=0$ is the corresponding increasing net of joint Haagerup--Schultz projections and $\Dc=W^*(\{q_t\mid 0\le t\le1\}$.
As in Section~\ref{sec:ut}, we write elements $S$ of $\Mcal\cap\Dc'$ as direct integrals
\[
S=\int_{[0,1]}^\oplus S(t)\,d\sigma(t),
\]
where $\sigma$ is the measure on $[0,1]$ defined by
\[
\sigma([0,t])=\tau(P(T:\rho([0,t])))=\nu_T(\rho([0,t])).
\]
Let $X_j=\Ec_{\Mcal\cap\Dc'}(T_j)$, for $1\le j\le n$.
Note that, by Lemma~\ref{lem:U}, $X_1,\ldots,X_n$ commute.

The next result generalizes Lemma~\ref{lem:fX}
to the Taylor holomorphic functional calculus.

\begin{lemma}\label{lem:holofX}
Let $f$ be a holomorphic function of $n$ variables defined on a neighborhood of the Taylor joint spectrum
$\Sp(X_1,\ldots,X_n)$.
Let $Y=f(X_1,\ldots,X_n)$.
By Proposition~\ref{prop:hfcdi}, we have
\begin{equation}\label{eq:holoYdi}
Y=\int_{[0,1]}^\oplus Y(t)\,d\sigma(t),
\end{equation}
where $Y(t)=f(X_1(t),\ldots,X_n(t))$ for $\sigma$-almost every $t\in[0.1]$.
Then for $\sigma$-almost every $t\in[0,1]$, the Brown measure $\nu_{Y(t)}$ of $Y(t)$ is the Dirac mass $\delta_{f(\rho(t))}$.
\end{lemma}
\begin{proof}
Fix $t$ and let $h(z_1,\ldots,z_n)=f(\rho_1(t)+z_1,\ldots,\rho_n(t)+z_n)-f(\rho_1(t),\ldots,\rho_n(t))$.
Since $h(0,\ldots,0)=0$, there exist functions $g_1,\ldots,g_n$, each with the same domain of holomorphy as $h$, such that
\begin{equation}\label{eq:hg}
h(z_1,\ldots,z_n)=\sum_{j=1}^nz_jg_j(z_1,\ldots,z_n).
\end{equation}
Indeed, for $n=1$ this is clear, while for $n\ge2$,
letting
\[
g_n(z_1,\ldots,z_n)=\begin{cases}\frac{h(z_1,\ldots,z_n)-h(z_1,\ldots,z_{n-1},0)}{z_n},&z_n\ne0 \\
\frac{\partial}{\partial z_n}\Big|_{z_n=0}h(z_1,\ldots,z_n),&z_n=0\end{cases}
\]
we have $h(z_1,\ldots,z_n)=h(z_1,\ldots,z_{n-1},0)+z_ng_n(z_1,\ldots,z_n)$
and we may argue by induction on $n$.

By Lemma~\ref{lem:di}, for $\sigma$-almost every $t\in[0,1]$ and every $j\in\{1,\ldots,n\}$, we have
$X_j(t)=\rho_j(t)+X_j\oup(t)$, where $X_j\oup(t)$ is an s.o.t.-quasinilpotent operator.
Using~\eqref{eq:hg}, we have
\[
Y(t)=f(\rho_1(t),\ldots,\rho_n(t))+\sum_{j=1}^nX_j\oup(t)g_j(X_1\oup(t),\ldots,X_n(t)\oup).
\]
By Lemmas~\ref{lem:prodqn} and~\ref{lem:sumqn}, we have, for almost all $t\in[0,1]$, 
\[
Y(t)=f(\rho_1(t),\ldots,\rho_n(t))+Y\oup(t),
\]
where $Y\oup(t)$ is s.o.t.-quasinilpotent.
Thus, for such values of $t$, we have $\nu_{Y(t)}=\delta_{f(\rho_1(t),\ldots,\rho_n(t))}$, as required.
\end{proof}

In this section we employ the holomorphic functional calculus in a Banach algebra, due to Arens~\cite{Ar61}
(see also Waelbrock~\cite{W54} and the exposition of Bourbaki~\cite{B67}).
Given a unital, commutative, Banach algebra $\Afr$, $n\in\Nats$ and $A_1,\ldots,\allowbreak A_n\in\Afr$,
letting $\sigma_\Afr(A_1,\ldots,A_n)$ denote the classical joint spectrum (see the start of Section~\ref{sec:jointspec} for the definition),
given an open neighborhood $U$ of this joint spectrum and given $f:U\to\Cpx$ holomorphic on $U$, the functional calculus assigns an
element $f(A_1,\ldots,A_n)\in\Afr$.
This map is the unique continuous algebra homomorphism from the algebra $\Hol(\sigma_\Afr(A_1,\ldots,A_n))$
of germs of holomorphic functions on $\sigma_\Afr(A_1,\ldots,A_n)$ into $\Afr$
that sends the $j$-th coordinate function to $a_j$, for every $j\in\{1,\ldots,n\}$,
where $\Hol(\sigma_\Afr(A_1,\ldots,A_n))$ is endowed with the appropriate topology
of uniform convergence on compact subsets of neighborhoods $U$ of $\sigma_\Afr(A_1,\ldots,A_n)$.
(See~\cite{B67} for details.)
Thus, it takes the expected values when $f$ is a polynomial or is given by an absolutely convergent power series.
Moreover, (see~\cite{B67}, Ch.\ I, \S4, Prop.\ 2.), if $\pi:\Afr\to\Bfr$ is a bounded unital algebra homomorphism between Banach algebras, then
$\pi(f(A_1,\ldots,A_n))=f(\pi(A_1),\ldots,\pi(A_n))$.

Let us observe that, if $A=(A_1,\ldots,A_n)$ is a tuple of commuting operators on a Hilbert space $\HEu$ and if $\Afr$ is the unital Banach
subalgebra of $B(\HEu)$ that they generate, then for every $f$ holomorphic on an open set containing the joint spectrum
$\sigma_\Afr(A_1,\ldots,A_n)$, the operator $f^{\text{Taylor}}(A)\in B(\HEu)$ defined by the Taylor functional calculus equals the operator
$f^{\text{Arens}}(A)\in\Afr$ defined by the Arens functional calculus in $\Afr$.
Indeed, $f\mapsto f^{\text{Taylor}}(A)$ is an algebra homomorphism from $\Hol(\sigma_{\Afr''}(A))$ into $\Afr''$,
that is continuous with respect to the topology of uniform convergence on compact subsets of open sets containing $\sigma_{\Afr''}(A)$;
this can be deduced, for example from the formula in Theorem~\ref{thm:Vasilescu}.
By uniqueness, this map must agree with the Arens functional calculus in $\Afr''$;
so we may write $f^{\text{Taylor}}(A)=f^{\text{Arens}(\Afr'')}(A)$.
However, considering the inclusion $\Afr\hookrightarrow\Afr''$ and the intertwining property of the Arens functional calculus
for bounded algebra homomorphisms,
the Arens functional calculus $f^{\text{Arens}(\Afr'')}(A)$ when $f$ is holomorphic on an open set containing $\sigma_\Afr(A_1,\ldots,A_n)$,
equals the Arens functional calculus taken in $\Afr$.
Thus, for such functions $f$, we have $f^{\text{Taylor}}(A)=f^{\text{Arens}}(A)$.

For the rest of this section, we let $\Afr\subseteq\Mcal$ be the unital Banach algebra generated by $\{T_1,\ldots,T_n\}$.
The next result generalizes our earlier simultaneous upper triangularization result, Theorem~\ref{thm:simultut},
from polynomials to holomorphic functions of operators, in the Arens functional calculus relative to $\Afr$.
\begin{thm}\label{thm:holomut}
Suppose $f$ is a holomorphic
function of $n$ variables whose domain contains a neighborhood of the joint spectrum $\sigma_\Afr(T_1,\ldots,T_n)$.
Let $S=f(T_1,\ldots,T_n)$.
Let $N=\Ec_\Dc(S)$.
Then $\nu_N=\nu_S$ and $S-N$ is s.o.t.-quasinilpotent.
In particular, the Brown measure $\nu_S$ equals the push-forward measure $f_*\nu_T$ of the joint Brown measure $\nu_T$
by $f$.
\end{thm}
\begin{proof}
Let $Y=\Ec_{\Mcal\cap\Dc'}(S)$.
We write $Y$ as a direct integral as in~\eqref{eq:holoYdi}.
By Lemma~\ref{lem:U} and Proposition~\ref{prop:EDhomom}, respectively, $\Ec_{\Mcal\cap\Dc'}\restrict_\Afr$ and $\Ec_{\Dc}\restrict_\Afr$ are
algebra homomorphisms.
Since the holomorphic functional calculus is intertwined with bounded algebra homomorphisms, we have
$Y=f(X_1,\ldots,X_n)$, where $X_j=\Ec_{\Mcal\cap\Dc'}(T_j)$.
The joint spectrum of the tuple $(X_1,\ldots,X_n)$ in the unital Banach algebra that it generates
is contained in the joint spectrum of $(T_1,\ldots,T_n)$ in $\Afr$.
Recall we observed, above,
that the Arens functional calculus and the Taylor functional calculus agree for functions, like $f$, that are holomorphic
on neighborhoods of $\sigma_\Afr(T_1,\ldots,T_n)$.
Thus, using Lemma~\ref{lem:holofX} and Brown's version of Lidskii's Theorem, we have
$\tau_t(Y(t))=\int_\Cpx z\,d\nu_{Y(t)}(z)=f(\rho(t))$ for $\sigma$-almost all $t\in[0,1]$.
Thus, we have
\begin{equation}\label{eq:Nintf}
N=\Ec_\Dc(Y)=\int_{[0,1]}^\oplus \tau_t(Y(t))\,d\sigma(t)=\int_{[0,1]}^\oplus f(\rho(t))\,d\sigma(t).
\end{equation}
Now Theorem~\ref{thm:5.6} yields
\[
\nu_N=\int_{[0,1]}\delta_{f(\rho(t))}\,d\sigma(t)=(f\circ\rho)_*\sigma,
\]
where the latter is the push-forward measure of $\sigma$ under $f\circ\rho$.
Since $\sigma([0,t])=\nu_t(\rho([0,1])$, by Lemma~\ref{lem:spfill} we have $\nu_T=\rho_*\sigma$.
Thus, $\nu_N=f_*\nu_T$.

By Lemma~\ref{lem:22}, $\nu_S=\nu_Y$ and $\nu_{S-N}=\nu_{Y-N}$.
But from~\eqref{eq:Nintf} we get
\[
Y-N=\int_{[0,1]}^\oplus\big(Y(t)-f(\rho(t))\big)\,d\sigma(t).
\]
Using Theorem~\ref{thm:5.6} and Lemma~\ref{lem:holofX}, we get
\begin{gather*}
\nu_Y=\int_{[0,1]}\nu_{Y(t)}\,d\sigma(t)=\int_{[0,1]}\delta_{f(\rho(t))}\,d\sigma(t)=(f\circ\rho)_*\sigma, \\
\nu_{Y-N}=\int_{[0,1]}\nu_{Y(t)-f(\rho(t))}\,d\sigma(t)=\int_{[0,1]}\delta_0\,d\sigma(t)=\delta_0. \\
\end{gather*}
Thus, $\nu_S=\nu_N$ and $\nu_{S-N}=\delta_0$
Therefore, $S-N$ is s.o.t.-quasinilpotent.
\end{proof}

Let $m\in\Nats$ and let each of $h_1,\ldots,h_m$ be a holomorphic function of $n$ variables with domain containing the joint spectrum 
$\sigma_\Afr(T_1,\ldots,T_n)$.
Let $h=(h_1,\ldots,h_m)$ denote the $\Cpx^m$-valued function.
We write $h_j(T)=h_j(T_1,\ldots,T_n)$ for the Arens functional calculus applied to the $n$-tuple $T$ and we write
\[
h(T)=(h_1(T),\ldots,h_m(T)).
\]
Note that $h(T)$ is an $m$-tuple of commuting operators.

Given $k\in\Nats$, $z=(z_1,\ldots,z_k)\in\Cpx^k$ and $\eps>0$, we let $B_\eps(z)$ denote the open polydisk
\[
B_\eps(z)=\{(w_1,\ldots,w_k)\in\Cpx^k\mid \forall j\,|z_j-w_j|<\eps\}
\]
\begin{lemma}\label{lem:PhTBeps}
Let $z\in\Cpx^n$ and let $\eps>0$.
Let $c>0$ be at least as large as $\sqrt n$ times the Lipschitz constant of $h$ in $B_\eps(z)$.
Then
\[
P(T:B_\eps(z))\le P\big(h(T):B_{c\eps}(h(z))\big).
\]
\end{lemma}
\begin{proof}
Since
\[
P\big(h(T):B_{c\eps}(h(z))\big)=\bigwedge_{j=1}^m P\big(h_j(T):B_{c\eps}(h_j(z))\big),
\]
it will suffice to show, for every $j$, the inequality
\[
P(T:B_\eps(z))\le P\big(h_j(T):B_{c\eps}(h_j(z))\big).
\]
If $P(T:B_\eps(z))=0$, then there is nothing to show;
so we may without loss of generality assume $\nu_T(B_\eps(z))>0$.
Let $Q=P(T:B_\eps(z))$.
Now $Q$ is invariant under every $T_j$ and, therefore, also under every element of the Banach algebra $\Afr$.
Since $h_j(T)$ belongs to $\Afr$,
by Theorem~\ref{thm:PTQB} we have
\begin{equation}\label{eq:PhTQB}
P\big(h_j(T):B_{c\eps}(h_j(z))\big)\wedge Q=P^{(Q)}\big(h_j(T)Q:B_{c\eps}(h_j(z))\big),
\end{equation}
where $P^{(Q)}$ means the Haagerup--Schultz projection computed in $Q\Mcal Q$ with respect to the renormalized trace.
But the map $\Afr\to Q\Afr Q$ given by $A\mapsto AQ$ is a bounded algebra homomorphism, so we have
$h_j(T)Q=h_j(TQ)$,
where $TQ=(T_1Q,\ldots,T_nQ)$.
By Theorem~\ref{thm:holomut}, the Brown measure, $\nu_{h_j(TQ)}^{(Q)}$, of $h_j(TQ)$ computed in $Q\Mcal Q$
is equal to the push-forward measure $(h_j)_*\nu_{TQ}^{(Q)}$
of the joint Brown measure $\nu_{TQ}^{(Q)}$ under $h_j$.
Thus, the measure $(h_j)_*\nu_{TQ}^{(Q)}$ is concentrated in $h_j(B_\eps(z))$.
But $\nu_{TQ}^{(Q)}$ is the renormalized restriction of $\nu_T$ to the polydisk $B_\eps(z)$ and $h_j$ maps this set into the disk $B_{c\eps}(h_j(z))$.
Thus, the Brown measure $\nu_{h_j(TQ)}^{(Q)}$ is concentrated in this latter disk, and the Haagerup--Schultz projection on the right hand side
of~\eqref{eq:PhTQB} is equal to the identity of $Q\Mcal Q$, namely, $Q$.
Therefore, we have 
\[
P\big(h_j(T):B_{c\eps}(h_j(z))\big)\ge Q,
\]
as required.
\end{proof}

\begin{lemma}\label{lem:PfTfX}
Suppose $U\subseteq\Cpx^n$ is an open set whose closure is a compact subset of the domain of $h$.
Then
\[
P(T:U)\le P\big(h(T):\overline{h(U)}\;\big).
\]
\end{lemma}
\begin{proof}
Let $c$ be at least $\sqrt n$ times the Lipschitz constant of $h$ on $U$.
For each $\eps>0$, let
\[
X_\eps=\{z\in\Cpx^n\mid\dist_\infty(z,U^c)\ge\eps\},\qquad Y_\eps=\{w\in\Cpx^m\mid \dist(w,h(U))<c\eps\}.
\]
Here $\dist_\infty$ is the distance with respect to the norm, $\|z-w\|_\infty=\max_j|z_j-w_j|.$ Clearly, $X_\eps$ is compact and $Y_\eps$ is open.
By Lemma~\ref{lem:PhTBeps}, for every $z\in X_\eps$, we have
\[
P(T:B_\eps(z))\le P\big(h(T):B_{c\eps}(h(z))\big)\le P(h(T):Y_\eps).
\]
By compactness, we may choose $z^{(1)},\ldots,z^{(k)}\in X_\eps$ such that $X_\eps\subseteq\bigcup_{j=1}^kB_\eps(z^{(j)})$.
Then, by the lattice properties (Theorem~\ref{thm:main.nu.P}\eqref{it:PTlattice}), we have
\[
P(T:X_\eps)\le P\left(T:\bigcup_{j=1}^kB_\eps(z^{(j)})\right)
=\bigvee_{j=1}^nP(T:B_\eps(z^{(j)}))\le P(h(T):Y_\eps).
\]
Let $\eps(k)$ decrease to $0$ as $k\to\infty$.
Then
\[
U=\bigcup_{k=1}^\infty X_{\eps(k)},\qquad\overline{h(U)}=\bigcap_{k=1}^\infty Y_{\eps(k)}.
\]
By the lattice properties again and since $P(T:X_{\eps(k)})$ is increasing and $P(h(T):Y_{\eps(k)})$ is decreasing in $k$, we get
\[
P(T:U)=\bigvee_{k=1}^\infty P(T:X_{\eps(k)})\le\bigwedge_{k=1}^\infty P(h(T):Y_{\eps(k)})=P(h(T):\overline{h(U)}\;),
\]
as required.
\end{proof}

\begin{thm}\label{thm:PhT}
We have
\begin{equation}\label{eq:nuhT}
\nu_{h(T)}=h_*\nu_T
\end{equation}
and for every Borel set $X\subseteq\Cpx^m$,
\begin{equation}\label{eq:PhTX}
P(h(T):X)=P(T:h^{-1}(X)).
\end{equation}
\end{thm}
\begin{proof}
First, suppose $V\subseteq\Cpx^m$ is bounded and open.
Let $U=h^{-1}(V)\cap B_M(0)$, where $M=1+\max_{1\le k\le n}\|T_k\|$.
Note that the support of $\nu_T$ is contained in $B_M(0)$ and that 
$U$ is bounded and open.
Thus, by using Lemma~\ref{lem:PfTfX} we have
\[
P(T:h^{-1}(V))=P(T:U)\le P(h(T):\overline{h(U)}\;)\le P(h(T):\overline{V}\;).
\]
Now suppose $K\subseteq\Cpx^m$ is compact and for $\eps>0$, let
\[
K_\eps=\{z\in\Cpx^m\mid\dist(z,K)<\eps\}.
\]
By the case just proved, we have
\[
P(T:h^{-1}(K_\eps))\le P(h(T):\overline{K_\eps}).
\]
But as $\eps$ decreases to $0$, the set $\overline{K_\eps}$ decreases to $K$ and the set $h^{-1}(K_\eps)$ decreases to $h^{-1}(K)$.
Thus, choosing a sequence $\eps(j)$ decreasing to zero and using the lattice properties (Theorem~\ref{thm:main.nu.P}\eqref{it:PTlattice}),
we have
\begin{equation}\label{eq:PThK}
P(T:h^{-1}(K))=\bigwedge_{j=1}^\infty P(T:h^{-1}(K_{\eps(j)}))\le\bigwedge_{j=1}^\infty P(h(T):\overline{K_{\eps(j)}})=P(h(T):K).
\end{equation}
Taking traces of both sides, we get 
\[
h_*\nu_T(K)=\nu_T(h^{-1}(K))=\tau(P(T:h^{-1}(K)))\le\tau(P(h(T):K))=\nu_{h(T)}(K).
\]

Let $X\subseteq\Cpx^m$ be any Borel set.
Since both $h_*\nu_T$ and $\nu_{h(T)}$ are regular measures on $\Cpx^m$, there exist compact sets $K_1\subseteq K_2\subseteq\cdots\subseteq X$
such that
\begin{equation}\label{eq:h*nuTlim}
\lim_{j\to\infty}\nu_{h(T)}(K_j)=\nu_{h(T)}(X),\qquad\lim_{j\to\infty}h_*\nu_T(K_j)=h_*\nu_T(X).
\end{equation}
Since for each $j$ we have $h_*\nu(K_j)\le\nu_{h(T)}(K_j)$, we get
\[
h_*\nu_T(X)\le\nu_{h(T)}(X).
\]
Since both $h_*\nu_T$ and $\nu_{h(T)}$ are probability measures, by considering complements we obtain the equality~\eqref{eq:nuhT}.

For every compact $K\subseteq\Cpx^m$, from~\eqref{eq:PThK} we have $P(T:h^{-1}(K))\le P(h(T):K)$.
Using~\eqref{eq:nuhT}, we have
\[
\tau(P(T:h^{-1}(K))=h_*\nu_T(K)=\nu_{h(T)}(K)=\tau(P(h(T):K)),
\]
so we must, in fact, have 
\[
P(T:h^{-1}(K))=P(h(T):K).
\]

Now for an arbitrary Borel set $X\subseteq\Cpx^m$, letting $K_j$ be an increasing sequence of compact subsets of $X$ so that~\eqref{eq:h*nuTlim} holds,
we have
\[
P(h(T):X)=\lim_{j\to\infty}P(h(T):K_j)=\lim_{j\to\infty}P(T:h^{-1}(K_j))=P(T:h^{-1}(X)),
\]
where the limits are in strong operator topology.
This, of course, is the desired equality~\eqref{eq:PhTX}.
\end{proof}

\begin{ques}\label{ques:TFC}
Do
analogues of 
Theorem~\ref{thm:holomut} and Theorem~\ref{thm:PhT} hold for the holomorphic functional calculus of Taylor?
\end{ques}

A formally easier question is:
\begin{ques}\label{ques:AFC''}
Do
analogues of 
Theorem~\ref{thm:holomut} and Theorem~\ref{thm:PhT} hold for the Arens functional calculus in $\Afr''$?
\end{ques}
The impediment to answering Question~\ref{ques:AFC''} in the same manner that we proved
Theorems~\ref{thm:holomut} and~\ref{thm:PhT}
is the question of whether the restriction of $\Ec_\Dc$ to $\Afr''$
is an algebra homomorphism.
We have $\Afr\subseteq\Afr'$, so $\Afr''\subseteq\Afr'$ and,
since the projections $q_t$ are $T$-hyperinvariant, we conclude that all elements of $\Afr''$ are upper triangular with respect to the
family $(q_t)_{0\le t\le 1}$ of projections.
Thus, Lemma~\ref{lem:U} gives us that the restriction of $\Ec_{\Mcal\cap\Dc'}$ to $\Afr''$ is an algebra homomorphism.
However, we don't know if the restriction of $\Ec_\Dc$ to $\Ec_{\Mcal\cap\Dc'}(\Afr'')$ is an algebra homomorphism.

The inclusions
\[
\Sp(T_1,\ldots,T_n)\subseteq\sigma_{\Afr''}(T_1,\ldots,T_n)\subseteq\sigma_\Afr(T_1,\ldots,T_n)
\]
are known.
Moreover, the right-most inclusion above can be proper:
even when $n=1$, we have the standard example of $\Mcal=L^\infty(\mathbb{T})$ and $T_1$ the function mapping $z\mapsto z$;
then $\Afr$ is the disk algebra and $\Afr''=\Mcal$, so $\sigma_\Afr(T_1)$ is the closed unit disk, while $\sigma_{\Afr''}(T_1)$ is the unit circle.
However, as far as we know the following question is open:
\begin{ques}
Can the inclusion
\[
\Sp(T_1,\ldots,T_n)\subseteq\sigma_{\Afr''}(T_1,\ldots,T_n)
\]
be proper for commuting operators $T_1,\ldots,T_n$ in a finite von Neumann algebra?
\end{ques}
The answer is \lq\lq yes\rq\rq if we ask instead about commuting operators on a Banach space, as was shown by Taylor~\cite{T70}.
See also~\cite{A79} for further interesting and related examples.

\section{Similarities}
\label{sec:sims}

It is clear from the definition of Brown measure that it is the same for any operators in the same similarity class.
In this section, we show that for any $S,T\in\Mcal$ with $S$ invertible,
the Haagerup--Schultz projections of $T$ and $STS^{-1}$ are related as follows:
$P(STS^{-1},B)$ is the projection onto the image of $SP(T,B)$.
We then show analogous results for families of commuting operators.

As usual, $\Mcal$ will be a von Neumann algebra equipped with a normal, faithful tracial state $\tau$, acting via a normal representation on some Hilbert space.
For $T\in\Mcal$, we will let $[T]$ denote the range projection of $T$, namely, the projection onto the closure of the range of $T$.
It is equal to the spectral projection $1_{(0,\infty)}(TT^*)\in\Mcal$.  In particular, it is independent of the action on a Hilbert space.

We will use the following easy lemmas repeatedly.
\begin{lemma}\label{lem:[SP]}
Let $S\in\Mcal$ be invertible and let
$P$ be a projection in $\Mcal$.
Then
\[
\tau([SP])=\tau(P).
\]
\end{lemma}
\begin{proof}
Writing $T=SP$, we have $[SP]=1_{(0,\infty)}(TT^*)$ and, since $S$ is invertible, $P=1_{(0,\infty)}(T^*T)$.
Thus, these two projections are unitarily conjugate, and have the same trace.
\end{proof}

\begin{lemma}\label{lem:[SQ]}
Let $S\in\Mcal$ be invertible and let
$(Q_j)_{j\in J}$ be projections in $\Mcal$.
Then we have
\[
\bigvee_{j\in J}[SQ_j]=\left[S\left(\bigvee_{j\in J}Q_j\right)\right],\qquad
\bigwedge_{j\in J}[SQ_j]=\left[S\left(\bigwedge_{j\in J}Q_j\right)\right].
\]
\end{lemma}
\begin{proof}
It is straightforward to verify this for $J$ finite, and then we may take a limit of a (monotone) net for arbitrary $J$.
\end{proof}

\begin{thm}\label{thm:simHS}
Suppose $A,S\in\Mcal$ with $S$ invertible.
Then 
\begin{enumerate}[(a)]
\item\label{it:nuSAS}
$\nu_{SAS^{-1}}=\nu_A$
\item\label{it:PSAS}
for every Borel set $X\subseteq\Cpx$, 
the Haagerup--Schultz projections satisfy 
\[
P(SAS^{-1},X)=[SP(A,X)].
\]
\end{enumerate}
\end{thm}
\begin{proof}
The equality~\eqref{it:nuSAS} is well known, and follows from 
the multiplicativity of the Fuglede--Kadison determinant and the definition of Brown measure,
since we have $\Delta(A-\lambda)=\Delta(SAS^{-1}-\lambda)$ for every $\lambda\in\Cpx$.

To prove~\eqref{it:PSAS}, first suppose $X$ is a closed disk $\overline{B_r(\lambda)}$ in $\Cpx$.
Using the description expressed at~\eqref{eq:ET} in Section~\ref{subsec:HS}, we have
\begin{multline*}
P(SAS^{-1},\overline{B_r(\lambda)})\HEu \\
=\{\xi\in\HEu\mid\exists\,(\xi_n)_{n=1}^\infty\subseteq\HEu,\,\lim_{n\to\infty}\|\xi_n-\xi\|=0,\,
\limsup_{n\to\infty}\|(SAS^{-1}-\lambda)^n\xi_n\|^{1/n}\le r\}.
\end{multline*}
Now using
\begin{align*}
\|S^{-1}\|^{-1}\|(A-\lambda)^nS^{-1}\xi_n\|\le\|(SAS^{-1}&-\lambda)^n\xi_n\|\le\|S\|\,\|(A-\lambda)^nS^{-1}\xi_n\|, \\
\|S\|^{-1}\|\xi_n-\xi\|\le\|S^{-1}\xi_n&-S^{-1}\xi\|\le\|S^{-1}\|\,\|\xi_n-\xi\|,
\end{align*}
we have
\begin{align*}
P&(SAS^{-1},\overline{B_r(\lambda)})\HEu \\
&=\begin{aligned}[t]
\{\xi\in\HEu\mid\exists\,(\xi_n)_{n=1}^\infty\subseteq\HEu,\,&\lim_{n\to\infty}\|S^{-1}\xi_n-S^{-1}\xi\|=0, \\
&\limsup_{n\to\infty}\|(A-\lambda)^nS^{-1}\xi_n\|^{1/n}\le r\}
\end{aligned} \displaybreak[2] \\
&=\{\xi\in\HEu\mid\exists\,(\eta_n)_{n=1}^\infty\subseteq\HEu,\,\lim_{n\to\infty}\|\eta_n-S^{-1}\xi\|=0,\,
\limsup_{n\to\infty}\|(A-\lambda)^n\eta_n\|^{1/n}\le r\} \displaybreak[2] \\
&=S\{\eta\in\HEu\mid\exists\,(\eta_n)_{n=1}^\infty\subseteq\HEu,\,\lim_{n\to\infty}\|\eta_n-\eta\|=0,\,
\limsup_{n\to\infty}\|(A-\lambda)^n\eta_n\|^{1/n}\le r\} \\
&=SP(A,\overline{B_r(\lambda)})\HEu.
\end{align*}
This proves~\eqref{it:PSAS} when $X$ is a closed disk.

Using Lemma~\ref{lem:[SQ]} and the lattice properties of Haagerup--Schultz projections,
if
\[
Y=\bigcup_{j=1}^n\overline{B_{r(j)}}(\lambda(j))
\]
is a union of closed disks in $\Cpx$, then
\begin{multline*}
P(SAS^{-1},Y)=\bigvee_{j=1}^nP(SAS^{-1},\overline{B_{r(j)}(\lambda(j))})
=\bigvee_{j=1}^n[SP(A,\overline{B_{r(j)}(\lambda(j))})] \\
=\left[S\left(\bigvee_{j=1}^nP(A,\overline{B_{r(j)}(\lambda(j))})\right)\right]
=[SP(A,Y)].
\end{multline*}

Suppose $K$ is a compact subset of $\Cpx$.
Then by a standard compactness argument, there is a sequence $Y_1\supseteq Y_2\supseteq\cdots$ such that
each $Y_j$ is a union of finitely many closed disks in $\Cpx$ and $K=\bigcap_{j=1}^\infty Y_j$.
Then by the lattice properties of Haagerup--Schultz projections, we have
\begin{multline*}
P(SAS^{-1},K)=\bigwedge_{j=1}^\infty P(SAS^{-1},Y_j)
=\bigwedge_{j=1}^\infty [SP(A,Y_j)] \\
=\left[S\left(\bigwedge_{j=1}^\infty P(A,Y_j)\right)\right]
=[SP(A,K)].
\end{multline*}

Suppose $X$ is an arbitrary Borel subset of $\Cpx$.
Since $\nu_A$ is regular, there is an increasing sequence $K_1\subseteq K_2\subseteq\cdots\subseteq X$ of compact
subsets of $X$ such that
$\nu_A(X)=\nu_A(\bigcup_{j=1}^\infty K_j)$.
Thus, (see Lemma~\ref{lem:PTBsymdif}) we have
\begin{multline*}
P(SAS^{-1},X)=P(SAS^{-1},\bigcup_{j=1}^\infty K_j)
=\bigvee_{j=1}^\infty P(SAS^{-1},K_j) \\
=\bigvee_{j=1}^\infty [SP(A,K_j)]
=\left[S\left(\bigvee_{j=1}^\infty P(A,K_j)\right)\right]
=[SP(A,X)].
\end{multline*}
\end{proof}

We now extend the result to the case of Brown measures and joint Haagerup--Schultz projections
for commuting families of operators.

\begin{thm}\label{thm:sim}
Let $I$ be a set and 
suppose $T=(T_i)_{i\in I}$ is a family of commuting operators in $\Mcal$.
Suppose $S\in\Mcal$ is invertible.
Then writing $STS^{-1}=(ST_iS^{-1})_{i\in I}$, we have
\begin{enumerate}[(a)]
\item\label{it:nuSTS}
$\nu_{STS{^-1}}=\nu_T$
\item\label{it:PSTS}
for every Borel set $X\subseteq\Cpx^I$, 
$P(STS^{-1}:X)=[SP(T:X)]$.
\end{enumerate}
\end{thm}
\begin{proof}
Let $Z=\prod_{i\in I}\sigma(T_i)$.
Since $\sigma(T_i)=\sigma(ST_iS^{-1})$, it will suffice to show~\eqref{it:PSTS} only for Borel sets $X\subseteq Z$.
To show~\eqref{it:nuSTS}, it will suffice to show that $\nu_T$ and $\nu_{STS^{-1}}$ agree on Borel subsets of $Z$.

First suppose $X\subseteq Z$ is a coordinate-finite rectangle, $X=\prod_{i\in I}B_i$.
Then using Lemma~\ref{lem:[SQ]} and Theorem~\ref{thm:simHS}, we have
\begin{multline*}
P(STS^{-1}:X)=\bigwedge_{i\in I}P(ST_iS^{-1},B_i)=\bigwedge_{i\in I}[SP(T_i,B_i)] \\
=\left[S\left(\bigwedge_{i\in I}P(T_i,B_i)\right)\right]=[SP(T:X)].
\end{multline*}
In particular, by Lemma~\ref{lem:[SP]}, we have
\[
\nu_{STS^{-1}}(X)=\tau(P(STS^{-1}:X))=\tau([SP(T:X)])=\tau(P(T:X))=\nu_T(X).
\]
Since $\nu_{STS}^{-1}$ and $\nu_T$ agree on all coordinate-finite rectangles, they agree on all Borel subsets of $Z$.
So~\eqref{it:nuSTS} holds.

Now let $X\subseteq Z$ be an arbitrary Borel set.
By Definition~\ref{def:PTBorel} and Proposition~\ref{prop:PtTE}, we have
\[
P(STS^{-1}:X)=\bigwedge\left\{\bigvee_{j=1}^\infty P(STS^{-1}:R_j)\;\bigg|\;X\subseteq\bigcup_{j=1}^\infty R_j\right\},
\]
where the $R_j$ are required to be coordinate-finite rectangles in $Z$.
Thus, by the case just shown, we have
\begin{align*}
P(STS^{-1}:X)
&=\bigwedge\left\{\bigvee_{j=1}^\infty [SP(T:R_j)]\;\bigg|\;X\subseteq\bigcup_{j=1}^\infty R_j\right\} \displaybreak[2] \\
&=\bigwedge\left\{\left[S\left(\bigvee_{j=1}^\infty P(T:R_j)\right)\right]\;\bigg|\;X\subseteq\bigcup_{j=1}^\infty R_j\right\} \displaybreak[2] \\
&=\left[S\left(\bigwedge\left\{\bigvee_{j=1}^\infty P(T:R_j)\;\bigg|\;X\subseteq\bigcup_{j=1}^\infty R_j\right\}\right)\right]
=[SP(T:X)].
\end{align*}
\end{proof}

\appendix

\section{Examples of Projections}
\label{app:projs}

This appendix provides examples of projections in $B(\HEu)$, for $\HEu$
separable, infinite dimensional Hilbert space,
showing that the conclusions of Lemmas~\ref{lem:PwQ} and~\ref{lem:pvqwedger}
can fail in $B(\HEu)$.

\begin{example}\label{ex:PwQ}
Let $(e_n)_{n=1}^\infty$ be an orthonormal basis for $\HEu$.
Let $P_n$ be the projection onto $e_n^\perp$ and let $Q_n$ be the projection onto $(\frac1ne_1+e_n)^\perp$.
Both sequences $(P_n)_{n=1}^\infty$ and $(Q_n)_{n=1}^\infty$ of codimension-one projections converge in strong operator topology
to $1$.
But $P_n\wedge Q_n$ is the projection onto $\{e_1,e_n\}^\perp$, and converges in strong operator topology to the projection onto $e_1^\perp$.
This shows that the conclusions of Lemma~\ref{lem:PwQ} may fail without existence of a trace.
\end{example}

\begin{example}\label{ex:pvqwedger}
In order to show that the conclusions of Lemma~\ref{lem:pvqwedger} may fail without existence of a trace,
we will construct projections $P$, $Q$ and $R$ in $B(\HEu)$ so that $P\le R$, $Q\wedge R=0$ but $(P\vee Q)\wedge R\ne P$.
Equivalently, we find closed subspaces $\EEu$, $\FEu$ and $\GEu$ of $\HEu$ so that
\begin{equation}\label{eq:EFG}
\EEu\subseteq\GEu,\quad\text{and}\quad\FEu\cap\GEu=\{0\},\quad\text{but}\quad\overline{(\EEu+\FEu)}\cap\GEu\ne\EEu.
\end{equation}
Consider
\begin{equation}\label{eq:HC2}
\HEu=\bigoplus_{n=0}^\infty\Cpx^2,
\end{equation}
an orthogonal direct sum of infinitely many two-dimensional subspaces.
For each $n\ge0$, let $\{e_n,f_n\}$ be a basis for the $n$-th two-dimensional subspace in~\eqref{eq:HC2}
so that $\|e_n\|=\|f_n\|=1$ and so that $\langle e_n,f_n\rangle=1+o(n^{-2})$ as $n\to\infty$.
Let
\begin{align*}
\EEu&=\clspan\{e_0+ne_n\mid n\ge1\}, \\
\FEu&=\clspan\{f_n\mid n\ge0\}, \\
\GEu&=\clspan\{e_n\mid n\ge0\}.
\end{align*}
Then we clearly have $\EEu\subseteq\GEu$.
We have $\FEu\cap\GEu=\{0\}$, because if $y\in\FEu\cap\GEu$, then letting $P_n$ be the projection from $\HEu$
onto the $n$-th two-dimensional subspace in~\eqref{eq:HC2}, we must have $P_n(y)=0$ for every $n\ge0$.
This implies $y=0$.

Let
\[
x=e_0-\sum_{n=1}^\infty \frac1ne_n.
\]
Then $x\perp\EEu$.
But $\langle x,e_0\rangle=1$, so $e_0\not\in\EEu$.
However, since
\[
\|e_n-f_n\|=\big(2-2\RealPart\langle e_n,f_n\rangle\big)^{1/2}=o(n^{-1})
\]
as $n\to\infty$, we have
\[
e_0=\lim_{n\to\infty}(e_0+ne_n)-nf_n\in\overline{(\EEu+\FEu)}.
\]
Thus, $e_0\in\overline{(\EEu+\FEu)}\cap\GEu$.
This proves the last assertion of~\eqref{eq:EFG}.
\end{example}

\end{document}